\documentclass[11pt]{amsart}
\usepackage[left=1.5in, right=1.5in]{geometry}
\usepackage[latin1]{inputenc}
\usepackage{amsfonts}
\usepackage[toc,page,title,titletoc,header]{appendix}
\usepackage{graphicx,psfrag,epsfig,multirow,caption,subcaption,color}
\usepackage{amssymb,amsmath,amscd,amsthm,amssymb,verbatim,hyperref,setspace}
\numberwithin{equation}{section}
\usepackage{mathrsfs,wrapfig}
\usepackage{indentfirst}
\usepackage{extarrows}
\newtheorem{theo}{Theorem}[section]
\newtheorem{lem}{Lemma}[section]
\newtheorem{pro}{Proposition}[section]

\newtheorem{defi}{Definition}[section]
\newtheorem{nota}{Notation}[section]

\def\Z{\mathbb Z}
\def\T{\mathbb T}

\def\R{\mathbb R}

\def\bmt{\left[\begin{array}}
\def\emt{\end{array}\right]}

\title[Arnold diffusion in nearly integrable Hamiltonian systems]{The genericity of Arnold diffusion in nearly integrable Hamiltonian systems}
\author{Chong-Qing Cheng}
\address{Department of mathematics, Nanjing Univerisity, Nanjing 210093, China}
\email{chengcq@nju.edu.cn}
\begin{document}
\maketitle
\begin{abstract}
In this paper, we prove that the net of transition chain is $\delta$-dense for nearly integrable positive definite Hamiltonian systems with 3 degrees of freedom in the cusp-residual generic sense in $C^r$-topology,  $r\ge 6$. The main ingredients of the proof existed in \cite{CZ,C17a,C17b}. As an immediate consequence, Arnold diffusion exists among this class of Hamiltonian systems. The question of \cite{C17c} is answered in Section 9 of the paper.
\end{abstract}
\renewcommand\contentsname{Index}

\section{Introduction}
\setcounter{equation}{0}
After he constructed the celebrated example of {\it a priori} unstable systems in \cite{A64}, Arnold raised the conjecture in \cite{A66} on the dynamical instability of nearly integrable Hamiltonian
\begin{equation}\label{eq1}
H(p,q)=h(p)+\epsilon P(p,q),\qquad (p,q)\in\mathbb{R}^n\times\mathbb{T}^n.
\end{equation}
\noindent{\bf Conjecture}.
{\it The ``general case" for a Hamiltonian system $($\ref{eq1}$)$ with $n\ge 3$ is represented by the situation that for an arbitrary pair of neighborhoods of toruses $p=p'$, $p=p''$ in one component of the level set $h(p')=h(p'')$, there exists, for sufficiently small $\epsilon$, an orbit intersecting both neighborhoods.}

The research on the conjecture has two stages: {\it a priori} unstable and {\it a priori} stable cases. After the problem in {\it a priori} unstable case was solved, one has to study how to cross double resonance. It was pointed out by Arnold in \cite{A66} that {\it in order to take the final step in the proof of the above conjecture, it is necessary to examine the transition from single to double resonance}. Indeed, one was able to establish the existence of global transition chain (Theorem 5.1 in \cite{C17b}) after the double resonance problem was solved there. A positive answer to the conjecture for smooth and positive definite Hamiltonian with $n=3$ is an immediate consequence, see Section 9.

The main part of the paper is to prove the existence of the $\delta$-dense transition chain (Theorem \ref{chain}), a slightly stronger form of Theorem 5.1 in Section 5 of \cite{C17b}. The main ingredients of the proof are included in \cite{CZ,C17a,C17b}.

To study the problem, one needs to specify what is the genericity. Mather used the cusp-residual genericity \cite{M04}, we follow him.

\begin{defi}
Let $B_D=\{p\in\mathbb{R}^3:\|p\|\le D\}$. Let $\mathfrak{S}_a,\mathfrak{B}_a\subset C^r(B_D\times\mathbb{T}^3,\mathbb{R})$ denote the sphere and the ball about the origin of radius $a>0$ respectively: $F\in\mathfrak{S}_a$ if and only $\|F\|_{C^r}=a$ and $F\in\mathfrak{B}_a$ if and only $\|F\|_{C^r}\le a$. They inherit the topology from $C^r(B_D\times\mathbb{T}^3,\mathbb{R})$. 

Let $\mathfrak{R}_a$ be a set residual in $\mathfrak{S}_a$, each $P\in\mathfrak{R}_a$ is associated with a set $R_P$ residual in the interval $[0,a_P]$ with $0<a_P\le a$. A set $\mathfrak{C}_a$ is said to be \emph{cusp-residual} in $\mathfrak{B}_a$ if
$$
\mathfrak{C}_a=\{\lambda P:P\in\mathfrak{R}_a,\lambda\in R_P\}.
$$
\end{defi}

A function $h$ is called positive definite if its Hessian matrix $\partial^2h$ is positive definite.
\begin{theo}\label{mainth}
Assume $h\in C^r(B_D,\mathbb{R})$ is positive definite, $r\ge 6$. For any small $\delta>0$, $E>\min h$ with $h^{-1}(E)\subset B_D$ and any two points $p^{\star},p^*\in h^{-1}(E)$, there exists a cusp-residual set $\mathfrak{C}_{\epsilon_0}\subset C^r(B_D\times\mathbb{T}^3,\mathbb{R})$ such that for each $\epsilon P\in\mathfrak{C}_{\epsilon_0}$ there exists an orbit $(p(t),q(t))$ of $\Phi_H^t$ which intersects the $\delta$-neighborhood of $p^{\star}$ and of $p^*$, namely, some $t^{\star},t^*\in\mathbb{R}$ exist such that $\|p(t^{\star})-p^{\star}\|<\delta$ and $\|p(t^*)-p^*\|<\delta$.
\end{theo}
This theorem proves the conjecture for positive definite Hamiltonian systems with three degrees of freedom in the cusp-residual generic sense in $C^r$-topology with $r\ge 6$.
We apply the variational method for the proof, it is based on Mather's theory and the weak KAM theory \cite{M91,M93,Fa}. Since we study the dynamics on the energy level set $\{H^{-1}(E)\}$, we can modify $h$ outside of a neighborhood of $\{h^{-1}(E)\}$ such that $h$ is Tonelli. A Hamiltonian $H(p,q,t)$ is called Tonelli if it satisfies the  conditions:
\begin{enumerate}
  \item its Hessian matrix $\partial^2_{p_ip_j}H$ in $p$ is positive definite everywhere;
  \item for each $(q,t)$ it holds that $H(p,q,t)/\|p\|\to\infty$ as $\|p\|\to\infty$;
  \item each solution of the Hamilton's equation has all of $\mathbb{R}$ as its domain of definition.
\end{enumerate}
For autonomous system, the third condition is automatically satisfied since each orbit lie on compact energy level set.

The definition of transition chain is recalled and the theorem of global transition chain is stated in Section 2 (Theorem \ref{chain}). In Section 3, we derive the normal form of $H$ when it is restricted in a neighborhood of double resonant point. In Section 4, we show that as a path, the candidate of transition chain is covered by discs around double resonances with controlled periods. In each disc, one Hamiltonian normal form holds.
In Section 5, we distinguish strong from weak double resonances and prove that there are only finitely many strong double resonances. The weak double resonance can be reduced to {\it a priori} unstable case such the problem
is reduced to the finite number of strong double resonances. In Section 6, we construct transition chain crossing strong double resonances by applying the main results of \cite{CZ,C17a,C17b}. By preparing some technical estimates for the nearly integrable system including the deviation of the rotation vectors, the location
of the flat and the estimate of orbits in the Aubry sets in Section 7, we prove Theorem \ref{chain} in Section 8. As an immediate consequence, Theorem \ref{mainth} is proved in Section 9.

A lot of works have been contributed to the topic since the conjecture was raised half a century ago. Normally hyperbolic invariant cylinder is assumed by the {\it a priori} unstable condition, along which the diffusion is well understood, by variational method and geometric methods, cf. \cite{B08,CY1,CY2,DLS1,LC,Tr,Zh1}. There are also many works for the problem, for instance, see  \cite{Bs,BCV,DH1,DH2,DLS2,FM,GL,GR1,GR2,KL1,KL2,X}.

Nearly integrable Hamiltonian is also called {\it a priori} stable system. Unlike {\it a  priori}
unstable system, multiple resonant points destruct the cylinder into many small pieces.
Away from the multiple resonant points, some piece of invariant cylinder was found in \cite{B10} and the method for {\it a priori} unstable system was applied in \cite{BKZ} to obtain local diffusion. Restricted in a neighborhood of multiple resonant point $p''$, $\{\|p-p''\|\le K\sqrt{\epsilon}\}$ with $K\gg 1$, the normal form is non-integrable.
So, it is a challenge to construct cylinder in such a disc. The condition $n=3$ allows us to apply a variational method to construct cylinder which extends $o(\sqrt{\epsilon})$-close to double resonant point, see \cite{CZ,C17a}.
Because of the result and by a new cohomology equivalence, we found a way in \cite{C17b} to pass through the small neighborhood by turning around the strong double resonant point and joining two cylinders. Recently, the mechanism of turning around the point was observed numerically in \cite{GSV}.

Earlier than us, Mather suggested a way to cross the double resonance \cite{M04}, which is based on an observation that the periodic orbits of the averaged system (\ref{averaged}) may approach two homoclinics simultaneously. He suggested to move
the first cohomology class in the channel determined by the prescribed homology class and switch it to the channel determined by one of the homoclinics when it is getting close to the double resonance. From geometric point of view, one expects to construct diffusion orbit that moves along the cylinder with the prescribed homology class, jumps to the cylinder with hole and passes through the neighborhood of double resonance in a way similar to {\it a priori} unstable case, as it was announced in \cite{KZ2}, see \cite{KZ1,Mar} also. For this approach one needs to consider the possibility that bifurcations of NHICs generically appear and could a priori accumulate when they approach the double resonant point.

\section{The definition of the transition chain}
\setcounter{equation}{0}

The terminology {\it $($generalized$)$ transition chain} used in the paper is defined in \cite{CY1,CY2,LC}, borrowed from \cite{A64} where it is defined by geometrical language. Definition in our setting is in a variational language.

For the definition, let $\check\pi:\check M\to\mathbb{T}^n$ be a finite covering of $\mathbb{T}^n$, let $\mathcal{N}(c,\check M)$, $\mathcal{A}(c,\check M)$ denote the Ma\~n\'e set, Aubry set with respect to $\check M$. The condition ({\bf HA}) ({\it hypothesis of Arnold}) is a variational version of Arnold's condition, the stable manifold of a circle intersects its unstable manifold transversally. Such intersection points lie in the Ma\~n\'e set, but not in the Aubry set.

({\bf HA}): there exists a finite covering $\check\pi:\check M\to M$ such that
\begin{enumerate}
  \item in time-periodic case: $\check\pi\mathcal{N}(c,\check M)|_{t=0}\backslash (\mathcal{A}(c,\check M)|_{t=0}+\delta)\ne\varnothing$ is totally disconnected, where $\mathcal{A}(c,\check M)|_{t=0}+\delta=\{x:\mathrm{dist}(x,\mathcal{A}(c,\check M)|_{t=0})\le\delta\}$;
  \item in autonomous case: $\check\pi\mathcal{N}(c,\check M)|_\Sigma\backslash (\mathcal{A}(c,\check M)+\delta)\ne\varnothing$ is totally disconnected, where $\Sigma$ is a section of $\check M$.
\end{enumerate}
It is not necessary to work always in nontrivial finite covering space. If the Aubry set contains more than one class, one can choose $\check M=\mathbb{T}^n$.

To state the definition of transition chain, we also need the concept of {\it cohomology equivalence}. The first version was introduced in \cite{M93}, however, it does not apply to interesting problem in autonomous systems (cf. \cite{B02}).
A new version of cohomology equivalence was introduced for autonomous system in \cite{LC}. For a Tonelli Lagrangian  defined on $T\mathbb{T}^n$, it is defined not with respect to the whole $\mathbb{T}^n$ as in \cite{M93}, but to a section. For $n$-torus $\mathbb{T}^n$, the section is chosen as a non-degenerately embedded section $(n-1)$-dimensional torus. We call $\Sigma_c$ non-degenerately embedded ($n-1$)-dimensional torus by assuming a smooth injection $\varphi$: $\mathbb{T}^{n-1}\to\mathbb{T}^n$ such that $\Sigma_c$ is the image of $\varphi$, and the induced map $\varphi_*$: $H_1(\mathbb{T}^{n-1}, \mathbb{Z})\to H_1(\mathbb{T}^{n},\mathbb{Z})$ is an injection.

For a first cohomology class $c$, we assume that there is a non-degenerate embedded $(n-1)$-dimensional torus $\Sigma_c\subset\mathbb{T}^n$ such that each $c$-semi static curve $\gamma$ transversally intersects $\Sigma_c$. Let
$$
\mathbb{V}_{c}=\bigcap_U\{i_{U*}H_1(U,\mathbb{R}): U\, \text{\rm is a neighborhood of}\, \mathcal {N}(c) \cap\Sigma_c\},
$$
here $i_U$: $U\to M$ denotes inclusion map. $\mathbb{V}_{c}^{\bot}$ is defined to be the annihilator of $\mathbb{V}_{c}$, i.e. if $c'\in H^1(\mathbb{T}^n,\mathbb{R})$, then $c'\in \mathbb{V}_{c}^{\bot}$ if and only if $\langle c',h \rangle =0$ for all $h\in \mathbb{V}_c$. Clearly,
$$
\mathbb{V}_{c}^{\bot}=\bigcup_U\{\text{\rm ker}\, i_{U}^*: U\, \text{\rm is a neighborhood of}\, \mathcal {N}(c) \cap\Sigma_c\}.
$$
There is a neighborhood $U$ of $\mathcal {N}(c)\cap\Sigma_c$ such that $\mathbb{V}_c=i_{U*}H_1(U,\mathbb{R})$ and $\mathbb{V}_{c}^{\bot}=\text{\rm ker}i^*_U$.

\begin{defi}
In autonomous case, $c,c'\in H^1(\mathbb{T}^n,\mathbb{R})$ are said to be cohomologically equivalent if there is a continuous curve $\Gamma$: $[0,1]\to H^1(\mathbb{T}^n,\mathbb{R})$ such that $\Gamma(0)=c$, $\Gamma(1)=c'$, $\alpha(\Gamma(s))$ keeps constant along $\Gamma$, and for each $s_0\in [0,1]$ there exists $\delta>0$ such that $\Gamma(s)-\Gamma(s_0)\in \mathbb{V}_{{\Gamma}(s_0)}^{\bot}$ whenever $s\in [0,1]$ and $|s-s_0|<\delta$.
\end{defi}

With the terminologies introduced as above, we are able to state the definition of transition chain for autonomous system.

\begin{defi}\label{chaindef1}
Two cohomolgy classes $c,c'\in H^1(M,\mathbb{R})$ are joined by a generalized transition chain if a continuous curve $\Gamma$: $[0,1]\to H^1(M,\mathbb{R})$ exists such that $\alpha(\Gamma(s))$ keeps constant and for each $s\in [0,1]$  at least one of the following cases takes place:
\begin{enumerate}
  \item the condition $(${\bf HA}$)$ holds for $\Gamma(s)$, $\mathcal{A}(\Gamma(s'))$ lies in a small neighborhood of $\mathcal{A}(\Gamma(s))$ provided $|s'-s|$ is small;
  \item there is $\delta_s>0$, for each $s'\in (s-\delta_s,s+\delta_s)$, $\Gamma(s')$ is cohomologically equivalent to $\Gamma(s)$.
\end{enumerate}
\end{defi}

It is proved in Theorem 3.1 of \cite{LC} for autonomous case that $\tilde{\mathcal{A}}(\Gamma(0))$ is dynamically  connected to $\tilde{\mathcal{A}}(\Gamma(1))$, namely, there is an orbit of the system which takes $\tilde{\mathcal{A}}(\Gamma(0))$ and $\tilde{\mathcal{A}}(\Gamma(1))$ as its $\alpha$-limit and $\omega$-limit set respectively.

A Tonelli Lagrangian $L$ is uniquely related to a Tonelli Hamiltonian $H$ through Legendre transformation $L(q,\dot q,t)=\max_p\langle \dot q,p\rangle-H(p,q,t)$, which determines a map $\mathscr{L}_H$: $T^*\mathbb{T}^n\times\mathbb{T}\to T\mathbb{T}^n\times\mathbb{T}$: $(p,q,t)\to (\dot q,q,t)$ with $\dot q=\partial_{p}H(p,q,t)$.

\begin{defi}
An orbit $(p(t),q(t),t)\subset T^*\mathbb{T}^n\times\mathbb{T}$ is said to be \emph{$\tilde c$-semi-static (static)} if $\mathscr{L}_H(p(t),q(t),t)=(\dot q(t),q(t),t)$ is $\tilde c$-semi-static $($static$)$. If the system is autonomous, we skip the component of $t$ $($see \cite{Man,M93}$)$.
\end{defi}

Since $H^1(\mathbb{T}^3,\mathbb{R})=\mathbb{R}^3$, we treat the first cohomology class $\tilde c\in H^1(\mathbb{T}^3,\mathbb{R})$ as a point $\tilde c\in\mathbb{R}^3$. In this case, the choice of diffusion path relies on the observation as follows. It holds along each $\tilde c$-semi-static orbit of the integrable Hamiltonian $h(p)$ that $p(t)\equiv\tilde c$. Since Ma\~n\'e set is upper semi-continuous with respect to the perturbation (see Lemma 2.3 of \cite{CY2} and follow the proof there), $\|p(t)-\tilde c\|\ll 1$ holds along any $\tilde c$-semi static orbit for the perturbed system.

Given two points $p^{\star},p^*\in h^{-1}(E)$ and small $\delta>0$, $\exists$ points $\bar p^{\star},\bar p^*\in h^{-1}(E)$ and vectors $k^{\star},k^*\in\mathbb{Z}^3\backslash\{0\}$ such that $\|p^{\star}-\bar p^{\star}\|<\frac \delta 2$, $\|p^*-\bar p^*\|<\frac \delta 2$, $\langle k^{\star},\partial h(\bar p^{\star})\rangle=0$ and $\langle k^*,\partial h(\bar p^*)\rangle=0$. One can choose $(\bar p^{\star},k^{\star})$ and $(\bar p^*,k^*)$ such that $k^{\star}$ and $k^*$ are totally irreducible. A vector $k=(k_1,k_2,k_3)\in\mathbb{Z}\backslash\{0\}$ is said to be {\it totally irreducible} if the greatest common divisor of $k_i$ and $k_j$ is equal to 1 for any $i\ne j$ and $i,j=1,2,3$. It is based on the observation that, for any $k\in\mathbb{Z}^3\backslash\{0\}$, one can choose totally irreducible $k'\in\mathbb{Z}^3\backslash\{0\}$ such that $\langle k,k'\rangle/\|k\|\|k'\|$ is close to 1. For a point $p=(p_1,p_2,p_3)\in\mathbb{R}^3$, we adopt the following notation for its maximum and Euclidean norm respectively
$$
|p|=\max\{|p_1|,|p_2|,|p_3|\}, \qquad \|p\|=\Big(\sum_{i=1}^3p_i^2\Big)^{\frac 12}.
$$

An integer $k\in\mathbb{Z}^3\backslash\{0\}$ determines a path of single resonance (a circle on a sphere)
$$
\Gamma_{k}=\{p\in\mathbb{R}^3:h(p)=E>\min h;\langle k,\partial h(p)\rangle=0\}.
$$
As $h$ is positive definite, $\partial h$ maps $\{p\in\mathbb{R}^3:h(p)\le E\}$ to a ball containing the origin. The circles $\Gamma_{k^{\star}}$ and $\Gamma_{k^*}$ intersect at two points if $k^\star$ is independent of $k^*$, otherwise $\Gamma_{k^*}=\Gamma_{k^\star}$. In both cases, one has a path $\Gamma$ connecting $\bar p^{\star}$ to $\bar p^*$. If $\Gamma_{k^{\star}}$ intersects $\Gamma_{k^*}$, it starts from the point $\bar p^{\star}$, moves along the circle $\Gamma_{k^{\star}}$ until it reaches the intersection point of $\Gamma_{k^{\star}}$ with $\Gamma_{k^*}$, after that, it moves along the circle $\Gamma_{k^*}$ until it arrives at the point $\bar p^*$. If $\Gamma_{k^{\star}}=\Gamma_{k^*}$, $\Gamma$ is just a piece of the circle, connecting $\bar p^{\star}$ to $\bar p^*$, see the figure below.

\begin{figure}[htp] 
  \centering
  \includegraphics[width=5.0cm,height=3.0cm]{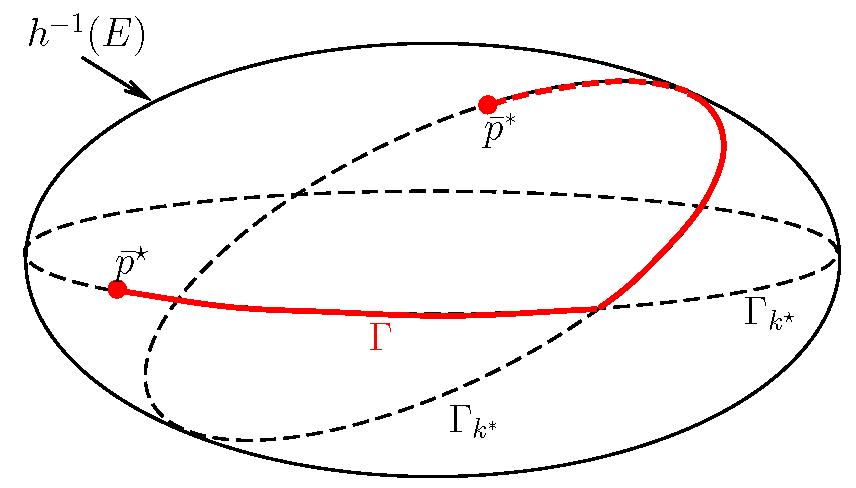}
\end{figure}

Treating $\Gamma$ as a path in $H^1(\mathbb{T}^3,\mathbb{R})$, we have a \emph{candidate of the transition chain}. We shall show that the transition chain lies in a small neighborhood of $\Gamma$.

The following theorem is a slightly stronger version of Theorem 5.1 of \cite{C17b}, the main part of this paper is for the proof of this theorem.

\begin{theo}\label{chain}
Assume $h\in C^r(B_D,\mathbb{R})$ is positive definite with $r\ge 6$. For any small $\delta>0$, $E>\min h$ with $h^{-1}(E)\subset B_D$, there is a cusp-residual set $\mathfrak{C}_{\epsilon_0}\subset C^r(B_D\times\mathbb{T}^3,\mathbb{R})$ such that for each $\epsilon P\in\mathfrak{C}_{\epsilon_0}$ and any two points $p^{\star},p^*\in h^{-1}(E)$, there is a transition chain that connects the class $\tilde c$ to the class $\tilde c'$ which satisfy the condition $\alpha(\tilde c)=\alpha(\tilde c')=E$, $|p^{\star}-\tilde c|<\delta$ and $|p^*-\tilde c'|<\delta$.
\end{theo}

The definition of cohomology equivalence can be further extended to more general version if {\it we treat the time $t$ as an angle variable and choose a section in the extended configuration space $\mathbb{T}^{n+1}$ where the extra dimension is for $t$}. If we write the cohomology class in coordinates $\tilde c=(c,-\alpha(c))$, the section $\Sigma_{\tilde c}$ is chosen for $\mathbb{T}^{n+1}$, $\mathbb{V}_{\tilde c}$ and $\mathbb{V}_{\tilde c}^{\bot}$ are defined in $H_1(\mathbb{T}^{n+1},\mathbb{R})$ and $H^1(\mathbb{T}^{n+1},\mathbb{R})$ respectively.

\begin{defi}\label{def}
In time-periodic case, $c,c'\in H^1(\mathbb{T}^n,\mathbb{R})$ are said to be cohomologically equivalent if there exists a continuous curve $\tilde\Gamma$: $[0,1]\to H^1(\mathbb{T}^{n+1},\mathbb{R})$ such that $\tilde\Gamma(0)=(c,-\alpha(c))$, $\tilde\Gamma(1)=(c',-\alpha(c'))$, and for each $s_0\in [0,1]$ there exists $\delta>0$ such that $\tilde\Gamma(s)-\tilde\Gamma(s_0)\in \mathbb{V}_{{\tilde\Gamma}(s_0)}^{\bot}$ whenever $s\in [0,1]$ and $|s-s_0|<\delta$.
\end{defi}

\section{Normal form}
Given an irreducible integer vector $k'\in\mathbb{Z}^3\backslash\{0\}$, one has a path of single resonance $\Gamma_{k'}=\{p\in H^{-1}(E):\langle k',\partial h(p)\rangle=0\}$.
A point $p''\in\Gamma_{k'}$ is said to be {\it double resonant} if there is an additional vector $k''\in\mathbb{Z}^3\backslash\{0\}$ independent of $k'$ such that $\langle k'',\partial h(p'')\rangle=0.$ Along each resonant path $\Gamma_{k'}$, there are many double resonant points.

We assume that $k'=(k'_1,k'_2,k'_3)\in\mathbb{Z}^3\backslash\{0\}$ is totally irreducible. In this case, there exist $k^*,k^{\star}\in\mathbb{Z}^3$ such that the matrix $M_0^t=(k',k^*,k^{\star})$ is uni-modular. Indeed, if $k'$ contains two non-zero entries, e.g. $k'_1,k'_2\ne 0$, we set $k^{\star}=(0,0,1)$ and $k^*=(k^*_1,k^*_2,0)$ such that $k'_1k^*_2-k'_2k^*_1=1$. If $k'=e_1$, we set $k^*=e_2$ and $k^{\star}=e_3$, where we use the notation that all other entries of $e_i$ are equal to zero except for the $i$-th entry, which is equal to 1. Other cases can be handled similarly.

Under a linear canonical transformation $\mathfrak{M}_0$: $(\bar p,\bar q)\to(p,q)$ such that $q=M^{-1}_0\bar q$, $p=M_0^t\bar p$, we obtain the Hamiltonian $\bar H=\bar h+\epsilon\bar P$ where $\bar h=\mathscr{M}_0^*h$ and $\bar P=\mathscr{M}_0^*P$. It holds along the path $\bar\Gamma_{k'}=M_0^{-t}\Gamma_{k'}$ that $\partial\bar h(\bar p)=(0,\omega_2,\omega_3)$.

In this case, the first and the second resonant condition are $M_0^{-t}k'=\bar k'=e_1$ and $M_0^{-t}k''=\bar k''=(0,\bar k''_2,\bar k''_3)$ respectively. If we introduce the canonical transformation of coordinates $\mathscr{M}$\footnote{In \cite{C17a}, the matrix $M_0M$ is set to be uni-modular. It is not always possible and not necessary.} $(u,v)\to(\bar p,\bar q)$ further
\begin{equation}\label{xuanzuan}
\bar q=M^{-1}u, \qquad \bar p=M^tv,
\end{equation}
where
$$
M^t=\left[\begin{matrix}
1 & 0 & 0\\ 0 & 1 & 0\\ 0 & \frac{\bar k''_3}{\bar k''_2} & 1
\end{matrix}\right],
\qquad
M^{-1}=\left[\begin{matrix}
1 & 0 & 0\\ 0 & 1 & \frac{-\bar k''_3}{\bar k''_2}\\ 0 & 0 & 1
\end{matrix}\right]
$$
if $|\bar k''_2|\ge |\bar k''_3|$ and
$$
M^t=\left[\begin{matrix}
1 & 0 & 0\\ 0 & 1 & \frac{\bar k''_2}{\bar k''_3}\\ 0 & 0 & 1
\end{matrix}\right],
\qquad
M^{-1}=\left[\begin{matrix}
1 & 0 & 0\\ 0 & 1 & 0\\ 0 & \frac{-\bar k''_2}{\bar k''_3} & 1
\end{matrix}\right]
$$
if $|\bar k''_2|\le |\bar k''_3|$. The function $\mathscr{M}^*\bar H$ is $2\pi$-periodic in $(u_1,u_2)$, $2|\bar k''|\pi$ in $u_3$ if $|\bar k''_2|\ge |\bar k''_3|$ and it is $2\pi$-periodic in $(u_1,u_3)$, $2|\bar k''|\pi$ in $u_2$ if $|\bar k''_3|\ge |\bar k''_2|$.

By the construction of $M$ which may not be uni-modular, we see that the function $\mathscr{M}^*H$ respects two symmetries in $u$.
\begin{defi}
Let $M$ be a non-degenerate matrix. A function $f(u)\in C^r(\mathbb{R}^n,\mathbb{R})$ is said to respect the symmetry $M$ if
$$
f(u+2\pi Me_i)=f(u), \qquad \forall\ u\in\mathbb{R}^n,\ e_i\in\mathbb{Z}^n.
$$
\end{defi}
Since $\mathscr{M}^*\bar H$ is $2\pi$-periodic in $(u_1,u_2)$, $2|\bar k''|\pi$-periodic in $u_3$ in the case that $|\bar k''_2|\ge |\bar k''_3|$ and $2\pi$-periodic in $(u_1,u_3)$, $2|\bar k''|\pi$ in $u_2$ in the case that $|\bar k''_3|\ge |\bar k''_2|$
$$
\mathscr{M}^*\bar H(u,v)=\sum_{k\in\mathbb{Z}^3}\bar H_k(M^tv)e^{i\langle k,M^{-1}u\rangle},
$$
$\mathscr{M}^*\bar H$ respects two symmetries in the variable $u$, $M$ and $\mathrm{diag}\{1,1,|\bar k''|\}$ for $|\bar k''_2|\ge |\bar k''_3|$, $M$ and $\mathrm{diag}\{1,|\bar k''|,1\}$ for $|\bar k''_2|\le |\bar k''_3|$ respectively.

At a double point $\bar p''$, the rotation vector $\bar\omega=(0,\bar\omega_2,\bar\omega_3)=\partial\bar h(\bar p'')$ is rational,  i.e. $\exists$ $T>0$ so that $T\omega\in\mathbb{Z}^n\backslash\{0\}$. If $t\omega\notin\mathbb{Z}^n$ $\forall$ $t\in (0,T)$, $T=T(\omega)$ is called \emph{the $($minimal$)$ period}. Since $\bar k''_2\bar\omega_2+\bar k''_3\bar\omega_3=0$, $T=|\bar k''|$. We consider those double resonant points $\{\bar p''\in\bar \Gamma_{k'}\}$ such that $T=T(\partial\bar  h(\bar p''))\le K^*\epsilon^{-\frac 13(1-3\kappa)}$ with $\kappa\in(0,\frac 16)$ and $K^*>0$ is independent of $\epsilon$. In this case, $|\bar k''|\sqrt{\epsilon}\to 0$ as $\epsilon\to 0$.
\begin{lem}\label{normalf}
Assume the second resonant condition at $\bar p''\in\bar \Gamma_{k'}$ is $\bar k''=(0,\bar k''_2,\bar k''_3)$, $\delta'\in(0,1/2)$. Then, there exists a small number $\epsilon_0>0$ such that for each $\epsilon\in(0,\epsilon_0]$, restricted on the level set $\bar H^{-1}(E)$ contained in $\tilde\Sigma_{\epsilon}\times\mathbb{T}^3$ with
$$
\tilde\Sigma_{\epsilon}=\{\bar p:|\bar p-\bar p''|\le K^{-1}\epsilon^{\kappa}\},
$$
where $K=\eta T$, $\eta\in(0,1]$ is independent of $\epsilon$, the Hamiltonian $\bar H$ is reduced, by a symplectic transformation and an energetic reduction, to a time-periodic perturbation of mechanical system with two degrees of freedom
\begin{equation}\label{2dHmil}
G_{\epsilon}(x,y,\theta)=\frac 12\langle By,y\rangle-V(x_1,|\bar k''|x_2)+ R_{\epsilon}\Big(x,y,\frac{\omega_{3}}{\sqrt{\epsilon}}\theta\Big),
\end{equation}
where the $2\times 2$ matrix $B$ is positive definite, $V\in C^r$ is $2\pi$-periodic in $(x_1,|\bar k''|x_2)$, $R_{\epsilon}(x,y,\vartheta)\in C^{r-2} (\mathbb{T}^2\times\Sigma'_{\epsilon}\times|\bar k''|\mathbb{T},\mathbb{R})$, $\vartheta=\omega_{3}\sqrt{\epsilon}^{-1}\theta$ and $\Sigma'_{\epsilon}$ satisfies the condition
$$
\{y:|y|\le (1-\delta')K^{-1}\epsilon^{\kappa-\frac 12}\}\subseteq \Sigma'_{\epsilon}\subseteq
\{y:|y|\le (1+\delta')K^{-1}\epsilon^{\kappa-\frac 12}\}.
$$
Restricted in $\mathbb{T}^2\times\Sigma'_{\epsilon}\times|\bar k''|\mathbb{T}$, some number $a_0=a_0(h,E,k')>0$ exists, independent of $T$ and $P$, such that for each $P\in\mathfrak{B}_1$ one has
$$
\|R_{\epsilon}\|_{C^{r-2}(\mathbb{T}^2\times\Sigma'_{\epsilon}\times|\bar k''|\mathbb{T},\mathbb{R})}\le a_0\epsilon^{\kappa},
$$
if it is treated as a function in $(x,y,\vartheta)$. Finally, the remainder $R_{\epsilon}(x,y,\vartheta)$ respects the symmetries $M$ and $\mathrm{diag}\{1,1,|\bar k''|\}$ in $(x,\vartheta)$ and the symplectic transformation is uniformly bounded for any second resonant condition $\bar k''$.
\end{lem}
\noindent{\bf Remark}. Because $V$ is independent of $\vartheta$, the symmetries $M$ and $\mathrm{diag}\{1,1,|\bar k''|\}$ for $V$ are the same as the identity.
\begin{proof}[Proof of Lemma \ref{normalf}]
To get the normal form, we introduce a coordinate transformation $\Phi_{\epsilon F}$ which is defined as the time-$2\pi$-map $\Phi_{\epsilon F}=\Phi^t_{\epsilon F}|_{t=2\pi}$ of the Hamiltonian flow generated by the function $\epsilon F(p,q)$. The function $F$ solves the homological equation
\begin{equation*}\label{homologicalequation}
\Big\langle\frac{\partial\bar  h}{\partial\bar  p}(\bar p''),\frac{\partial F}{\partial\bar  q}\Big\rangle=-\bar P(\bar p,\bar q)+ Z(\bar p,\bar q)
\end{equation*}
where
\begin{equation}\label{resonant term}
Z(\bar p,\bar q)=\frac 1T\int_0^T\bar P(\bar p,\bar q+\bar \omega t)dt=\sum_{(\ell_1,\ell_2)\in\mathbb{Z}^2}\bar P_{\ell}(\bar p)e^{i(\ell_1\langle\bar k',\bar q\rangle+\ell_2\langle\bar k'',\bar q\rangle)},
\end{equation}
Clearly, the function
$$
F(p,q)=\frac 1T\int_0^T\bar P(\bar p,\bar q+\bar \omega t)tdt
$$
solves the homological equation and $\|F\|< T\|\bar P\|$.

Under the transformation $\Phi_{\epsilon F}$ we obtain a new Hamiltonian
\begin{equation}\label{KAMform}
\begin{aligned}
\Phi_{\epsilon F}^*\bar H=&\bar h(\bar p)+\epsilon Z(\bar p,\bar q)+\epsilon\Big\langle\frac{\partial\bar  h}{\partial \bar p}(\bar p)-\frac{\partial\bar  h}{\partial\bar p}(\bar p''),\frac{\partial F}{\partial\bar q}\Big\rangle\\
&+\frac {\epsilon^2}2\int_0^1(1-t)\{\{\bar H,F\},F\}\circ\Phi_{\epsilon F}^tdt.
\end{aligned}
\end{equation}

To simplify the situation further, we introduce another canonical transformation of coordinates $\mathscr{M}$: $(u,v)\to(\bar p,\bar q)$ defined in (\ref{xuanzuan}).
Although the norm $\|\bar k''\|$ will be large if the second resonant condition is weak, the linear coordinate transformation (\ref{xuanzuan}) is uniformly bounded in the second resonant condition since $|\bar k''_3/\bar k''_2|\le 1$ in the first case and $|\bar k''_2/\bar k''_3|\le 1$ in the second case.

We consider the case that $|\bar k''_2|\ge |\bar k''_3|$. In the new coordinates $(u,v)$, the rotation vector takes the form of $(0,0,\omega_3)$. Clearly, $|\omega_3|$ is uniformly lower bounded above zero for all double resonant points on $\Gamma_{k'}$. Because $(\bar k''_2,\bar k''_3)$ is irreducible, it follows from (\ref{xuanzuan}) that  $\mathscr{M}^*\Phi_{\epsilon F}^*\bar H$ is $2\pi$-periodic in $(u_1,u_2)$ and $2|\bar k''|\pi$-periodic in $u_3$.

By the construction, the function $\mathscr{M}^*\Phi_{\epsilon F}^*\bar H$ possesses the symmetries of $M$ and $\mathrm{diag}\{1,1,|\bar k''|\}$. We need to be careful when a perturbation is added, it should respect the symmetries as well. By the relation (\ref{xuanzuan}) one has
$$
\langle k',q\rangle=\langle k',(MM_0)^{-1}u\rangle=u_1,\ \qquad \langle k'',q\rangle=\langle k'',(MM_0)^{-1}u\rangle=\bar k''_2u_2.
$$
It follows from Formula (\ref{resonant term}) and the transformation (\ref{xuanzuan}) that the resonant term has the form of $\mathscr{M}^*\mathscr{M}_0^*Z(p,\langle k',q\rangle,\langle k'',q\rangle)=Z'(v,u_1,\bar k''_2u_2)$.

Let $h'=\mathscr{M}^*\bar h$, $F'=\mathscr{M}^*F$ and $H'=\mathscr{M}^*\Phi_{\epsilon F}^*\bar H$. Since the transformations (\ref{xuanzuan}) is canonical, it preserves the Poison bracket. We obtain from Formula (\ref{KAMform}) that
\begin{align*}
H'=&h'(v)+\epsilon Z'(v,u_1,|\bar k''|u_2)+\epsilon\Big\langle\frac{\partial h'}{\partial v}(v)-\frac{\partial h'}{\partial v}(v''),\frac{\partial F'}{\partial u}\Big\rangle\\
&+\frac {\epsilon^2}2\mathscr{M}^*\int_0^1(1-t)\{\{H,F\},F\}\circ\Phi_{\epsilon F}^tdt,
\end{align*}
where $v''=(M_0M)^{-1}p''$. The function $H'$ determines its Hamiltonian equation
\begin{equation}\label{H-equation}
\frac{du}{dt}=\frac{\partial}{\partial v}H'(u,v),\qquad \frac{dv}{dt}=-\frac{\partial}{\partial u}H'(u,v).
\end{equation}
For this equation we introduce another transformation 
\begin{equation}\label{qicihua}
\tilde G_{\epsilon}=\frac 1{\epsilon}H', \qquad \tilde y=\frac 1{\sqrt{\epsilon}}\Big(v-v''\Big), \qquad \tilde x=u, \qquad s=\sqrt{\epsilon}t,
\end{equation}
where we use the notation $\tilde y=(y_1,y_2,y_3)=(y,y_3)$ and $\tilde x=(x_1,x_2,x_3)=(x,x_3)$.
In the new canonical variables $(\tilde x,\tilde y)$ and the new time $s$, Equation (\ref{H-equation}) turns out to be the Hamiltonian equation with the generating function as the following:
\begin{equation}\label{Hamiltonian}
\tilde G_{\epsilon}=\frac 1{\epsilon}\Big(h'(v''+\sqrt{\epsilon}\tilde y)-h'(v'')\Big)
-V(x_1,|\bar k''|x_2)+\tilde R_{\epsilon}(\tilde x,\tilde y),
\end{equation}
where $V(x_1,|\bar k''|x_2)=-Z'(v'',x_1,|\bar k''|x_2)$ and $\tilde R_{\epsilon}=\tilde R_{\epsilon,1}+\tilde R_{\epsilon,2}+\tilde R_{\epsilon,3}$ with
\begin{equation*}
\begin{aligned}
\tilde R_{\epsilon,1}&=Z'(v''+\sqrt{\epsilon}\tilde y,\tilde x)-Z'(v'',\tilde x),\\
\tilde R_{\epsilon,2}&=\Big\langle\frac{\partial h'}{\partial v}\Big(v''+\sqrt{\epsilon}\tilde y\Big)-\frac{\partial h'}{\partial v}(v''),\frac{\partial F'}{\partial u}\Big\rangle,\\
\tilde R_{\epsilon,3}&=\frac {\epsilon}2\mathscr{M}^*\int_0^1(1-t)\{\{\bar H,F\},F\}\circ\Phi_{\epsilon F}^tdt.
\end{aligned}
\end{equation*}
One step of KAM iteration makes the remainder $\tilde R_{\epsilon}$ lose two times of differentiability. Since the Hamiltonian $\bar H$ is defined in $\tilde\Sigma_{\epsilon}\times\mathbb{T}^2\times|\bar k''|\mathbb{T}$, we see from (\ref{xuanzuan}), (\ref{qicihua}) and $\Phi_{\epsilon F}$ that $\tilde G$ is defined on the domain that is contained in $\{|\tilde y|\le 2K^{-1}\epsilon^{\kappa-\frac 12}\}\times\mathbb{T}^2\times|\bar k''|\mathbb{T}$. Restricted in the domain, we claim that there exists a number $a_2>0$, depending on $h,E,k'$ only, such that
\begin{equation}\label{bound}
\|\tilde R_{\epsilon,i}\|_{C^{r-2}}\le a_2\epsilon^{\kappa}, \qquad i=1,2,3,
\end{equation}
for small $\epsilon>0$.
Indeed, let $m'$ be the upper bound of the largest eigenvalue of $\partial^2h(p)$ for all $p\in h^{-1}(E)$. Since $T\le K^*\epsilon^{-\frac 13(1-3\kappa)}$, one has
$$
\begin{aligned}
|\tilde R_{\epsilon,1}|&\le 2\|P'\|_{C^1}\sqrt{\epsilon}K^{-1}\epsilon^{\kappa-\frac 12}\le 2\|P'\|_{C^1}K^{-1}\epsilon^{\kappa},\\
|\tilde R_{\epsilon,2}|&\le m'\sqrt{\epsilon}K^{-1}\epsilon^{\kappa-\frac 12}T\|P'\|_{C^1}\le m'\|P'\|_{C^1}\epsilon^{\kappa}, \\
|(\mathscr{M}^*)^{-1}\tilde R_{\epsilon,3}|&\le \frac \epsilon2\|P'\|_{C^2}^2\|H'\|_{C^1}T^2=\frac 12(K^*\|P'\|_{C^2})^2\|H'\|_{C^1}\epsilon^{\frac 13+2\kappa}.
\end{aligned}
$$
The estimate on the derivatives of the terms can also be done inductively.

We introduce another coordinate rescaling further
\begin{equation}\label{energylevel}
\theta=\frac{\sqrt{\epsilon}}{\omega_{3}}x_3,\qquad I=\frac{\omega_{3}}{\sqrt{\epsilon}}y_3,
\end{equation}
By expanding $\tilde G_{\epsilon}$ in $O(K^{-1}\epsilon^{\kappa})$ neighborhood of $v''$ in Taylor formula, we obtain
\begin{equation}\label{local}
\begin{aligned}
\tilde G_{\epsilon}&(x,y,I,\theta)=I+\frac 12\Big\langle\tilde B\Big(y,\frac{\sqrt{\epsilon}}{\omega_{3}} I\Big),\Big(y,\frac{\sqrt{\epsilon}}{\omega_{3}}I\Big)\Big\rangle-V(x_1,|\bar k''|x_2)\\
&+\tilde R_h\Big(y,\frac{\sqrt{\epsilon}}{\omega_{3}} I\Big)+\tilde R_{\epsilon}\Big( x_1,x_2,\frac{\omega_{3}}{\sqrt{\epsilon}}\theta, p''+\Big(\sqrt{\epsilon}y,\frac{\epsilon}{\omega_{3}}I\Big)\Big)
\end{aligned}
\end{equation}
where $\tilde B=\frac{\partial^2h'}{\partial v^2}(v'')$ and term $\tilde R_h$ represents the following
$$
\frac 1{\epsilon}\left[h'\Big(v''+\Big(\sqrt{\epsilon}y,\frac{\epsilon}{\omega_{3}} I\Big)\Big)-\Big[h'(v'')+\epsilon I+\frac {\epsilon}2\Big\langle\tilde B\Big(y,\frac{\sqrt{\epsilon}}{\omega_{3}} I\Big),\Big(y,\frac{\sqrt{\epsilon}}{\omega_{3}}I\Big)\Big\rangle\Big]\right].
$$
By the construction, all entries of $\tilde B$ are of order $O(1)$, independent of the period $T$ of the rotation vector $\partial h(p'')$. We write
$$
\tilde B=\left[\begin{matrix} B & B'\\ B'^t & B''
\end{matrix}\right]
$$
where $B$ is a $2\times2$ matrix, $B'$ is a vector with two entries and $B''>0$.

Obviously, restricted on the domain $\{|y|\le 2K^{-1}\epsilon^{\kappa-\frac 12}, |I|\le 2K^{-1}|\omega_3|^{-1}\epsilon^{\kappa-1}\}$, there exists a constant $a_3=a_3(h,E)>0$ such that
\begin{equation}\label{bound2}
\|\tilde R_h\|_{C^{r-2}}\le a_4\epsilon^{\kappa}.
\end{equation}

Let $\Omega_{\epsilon}$ be the image of $\tilde\Sigma_{\epsilon}\times\mathbb{T}^3$ under the maps $\Phi_{\epsilon F'}$, (\ref{xuanzuan}), (\ref{qicihua}) and (\ref{energylevel}). Since the transformation $\Phi_{\epsilon F'}$ is close to identity, $|k''_2|\ge|k''_3|$ is assumed, each section of $\Omega_{\epsilon}$ where $(x,\theta)$ keeps constant lies in $\{|y|\le 2K^{-1}\epsilon^{\kappa-\frac 12}, |I|\le 2K^{-1}|\omega_3|^{-1}\epsilon^{\kappa-1}\}$.
Restricted in $\Omega_{\epsilon}$, we find by direct calculation that $\partial_I\tilde G_{\epsilon}=1+O(\epsilon^{\kappa})$. Therefore, there exists a function $G_{\epsilon}(x,y,\theta)$ solves the equation $\tilde G_{\epsilon}(x,y,-G_{\epsilon},\theta)=0$.

Indeed, restricted in the domain $\{|y|\le 2K^{-1}\epsilon^{\kappa-\frac 12}, |I|\le 2K^{-1}\omega_3^{-1}\epsilon^{\kappa-1}\}$, a constant $a_4=a_4(h)>0$ exists such that  $|\sqrt{\epsilon}\langle B',y\rangle|\le a_4K^{-1}\epsilon^\kappa$ and $|\epsilon\langle By,y\rangle|\le a_4K^{-2}\epsilon^{2\kappa}$. It guarantees that the solution $I=-G_0(x,y)$ of the following quadratic equation in $I$
$$
\begin{aligned}
\tilde G_0=&I+\frac 12\Big\langle\tilde B\Big(y,\frac{\sqrt{\epsilon}}{\omega_{3}} I\Big),\Big(y,\frac{\sqrt{\epsilon}}{\omega_{3}}I\Big)\Big\rangle-V(x_1,|\bar k''|x_2)\\
=&\Big(1+\frac{\sqrt{\epsilon}}{\omega_3}\langle B',y\rangle\Big)I+\frac 12\langle By,y\rangle + \frac {B''}{2\omega_3^2}\epsilon I^2-V(x_1,|\bar k''|x_2)=0
\end{aligned}
$$
has the form of
$$
\begin{aligned}
G_0=&\frac{\omega_3^2(1+\frac{\sqrt{\epsilon}}{\omega_3}\langle B',y\rangle)}{B''\epsilon}\Big(1-\Big(1-
\frac{2B''\epsilon(\frac 12\langle By,y\rangle-V(x_1,|\bar k''|x_2))}{\omega_3^2(1+\frac{\sqrt{\epsilon}}{\omega_3} \langle B',y\rangle)^2}\Big)^{\frac 12}\Big) \\
=&\frac 12\langle By,y\rangle-V(x_1,|\bar k''|x_2)+ R_0(x,y)
\end{aligned}
$$
and some $a_5(h,E)>0$ exists such that restricted in the domain $\mathbb{T}^2\times\{|y|\le 2K^{-1}\epsilon^{\kappa-\frac 12}\}$ one has $\|R_0\|_{C^{r-2}}\le a_5\epsilon^{\kappa}$, provided $P\in\mathfrak{B}_1$.

According to the estimates (\ref{bound}) and (\ref{bound2}), one has
$$
\|\tilde G_{\epsilon}-\tilde G_0\|_{C^{r-2}}=\|\tilde R_h+\tilde R_{\epsilon}\|_{C^{r-2}}\le (3a_2+a_4)\epsilon^{\kappa}
$$
when $(x,\vartheta,y,I)$ is restricted in $\mathbb{T}^2\times|\bar k''|\mathbb{T}\times\{|y|\le 2K^{-1}\epsilon^{\kappa-\frac 12}, |I|\le 2K^{-1}|\omega_3|^{-1}\epsilon^{\kappa-1}\}$ where $\vartheta= \sqrt{\epsilon}^{-1}\omega_3\theta$. It follows from the relations $\partial_I\tilde G_{\epsilon}=1+O(\epsilon^{\kappa})$, $\partial_I\tilde G_0=1+O(\epsilon^{\kappa})$ and the theorem of implicit function that some $a_6=a_6(h,d)>0$ exists such that $\|G_\epsilon-G_0\|_{C^{r-2}}\le a_6\epsilon^{\kappa}$ holds when $|y|\le 2K^{-1}\epsilon^{\kappa-\frac 12}$. Indeed, let $z=(x,y,\vartheta)$, we get from the equations $\tilde G_0(z,-G_0)=0$ and $\tilde G_{\epsilon}(z,-G_{\epsilon})=0$ that
\begin{equation}\label{implicit}
\partial _I\tilde G_0(z,-G_0+\lambda(G_{\epsilon}-G_0))(G_{\epsilon}-G_0)-(\tilde R_h+\tilde R_{\epsilon})(z,G_{\epsilon})=0,
\end{equation}
where $\lambda=\lambda(z,G_{\epsilon}-G_0)\in[0,1]$. It follows from the relation $\partial_IG_0=1+O(\epsilon^{\kappa})$ that $\max_z|G_{\epsilon}-G_0|\le 2(3a_2+a_4)\epsilon^{\kappa}$. For $\xi=0,\epsilon$, the derivative of $G_{\xi}$ in $z$ satisfies the equation
$$
\partial_z\tilde G_{\xi}(z,-G_{\xi}(z))-\partial_I\tilde G_{\xi}(z,-G_{\xi}(z))\partial_zG_{\xi}(z)=0.
$$
Because $\partial^2_I\tilde G_0=\omega_3^{-2}B''\epsilon$, $|\partial_zG_{\epsilon}|=|\partial_I\tilde G_{\epsilon}|^{-1}|\partial_z\tilde G_{\epsilon}|=O(1)$, by taking the difference of these equations we obtain the estimate on first derivative of $G_{\epsilon}-G_0$
$$
\begin{aligned}
|\partial_z(G_{\epsilon}-G_0)(z)|\le &|\partial_I\tilde G_0(z,-G_0)^{-1}|\Big(|{B''}{\omega_3^{-2}}\epsilon(G_\epsilon-G_0)\partial_zG_{\epsilon}|\\
&+|\partial_z\tilde G_{0}(z,-G_{\epsilon})-\partial_z\tilde G_{0}(z,-G_0)|+\Big|\frac d{dz}(\tilde R_h+\tilde R_\epsilon)(z,-G_\epsilon)\Big|\Big) \\
\le& 2(3a_2+a_4)\epsilon^{\kappa}\Big((2|\omega_3^{-2}B''|\epsilon+1)|\partial_zG_{\epsilon}| +2|B'||\omega_3^{-1}|\sqrt{\epsilon}+1\Big).
\end{aligned}
$$
The estimate on higher order derivatives can be done similarly.

From the formula (\ref{2dHmil}), we see that $|G_{\epsilon}(x,y,\theta)|\le a_4K^{-2}\epsilon^{2\kappa-1}$ for $|y|\le 2K^{-1}\epsilon^{\kappa-\frac 12}$ if $\epsilon>0$ is suitably small such that $|V|+|R_{\epsilon}|<\frac 12a_4K^{-2}\epsilon^{2\kappa-1}$. Consequently, the energy level set $\tilde G_{\epsilon}^{-1}(0)$ intersects the domain $\{|y|\le K^{-1}\epsilon^{\kappa-\frac 12}, |I|\le K^{-1}\omega_3^{-1}\epsilon^{\kappa-1}\}\times\mathbb{T}^3$ at the place where $|I|\le a_4\epsilon^{2\kappa-1}$.

Under the composition of $\Phi_{\epsilon F'}$, (\ref{xuanzuan}), (\ref{qicihua}) and (\ref{energylevel}), $\tilde\Sigma_{\epsilon}\times\mathbb{T}^3$ is mapped onto $\Omega_{\epsilon}$. If $\Phi_{\epsilon F}$ is an identity map, each section of $\Omega_{\epsilon}$ where $(x,\theta)$ keeps constant contains the disk  $\{|y|\le K^{-1}\epsilon^{\kappa-\frac 12},I=0\}$, because $|k''_2|\ge|k''_3|$ is assumed in (\ref{xuanzuan}). Since $\Phi_{\epsilon F'}$ approaches to identity and $\frac{a_4\epsilon^{2\kappa-1}}{K^{-1}\omega_3^{-1}\epsilon^{\kappa-1}}=a_4K\omega_3\epsilon^{\kappa}\to 0$ as $\epsilon\to 0$, some $\epsilon_0(\delta')>0$ exists such that, for any $\epsilon\in(0,\epsilon_0]$, the set $\{|y|\le (1-\delta')K^{-1}\epsilon^{\kappa-\frac 12},|I|\le a_4\epsilon^{2\kappa-1}\}$ is contained in each section of $\Omega_{\epsilon}$ where $(x,\theta)$ is fixed, i.e. $\Sigma_{\epsilon}\supset\{|y|\le (1-\delta')K^{-1}\epsilon^{\kappa-\frac 12}\}$. Similarly, we can show that $\Sigma_{\epsilon}\subset\{|y|\le (1+\delta')K^{-1}\epsilon^{\kappa-\frac 12}\}$.

The transformation $\mathscr{M}$ of (\ref{xuanzuan}) is uniformly bounded for all resonant conditions along $\Gamma_{k'}$ and all constants $a_\ell$ with $2\le\ell\le 7$ can be set to be independent of the second resonant condition. Let $a_0=a_6+a_7$, we then complete the proof of the lemma in the case that $|\bar k''_2|\ge|\bar k''_3|$.

If $|\bar k''_2|\le |\bar k''_3|$, in the new coordinates $(u,v)$, the frequency at the double resonance has the form of $M'^t\partial h(p'')=(0,\omega_3,0)$. From (\ref{xuanzuan}) one obtains
$$
\langle k',q\rangle=\langle k',(MM_0)^{-1}u\rangle=u_1,\ \qquad \langle k'',q\rangle=\langle k'',(MM_0)^{-1}u\rangle=\bar k''_3u_3.
$$
By introducing the permutation $u_2\leftrightarrow u_3$ and $v_2\leftrightarrow v_3$, we are again in the situation we have handled. The rest of the proof is the same as above.
\end{proof}

Let $A=B^{-1}$. By the Legendre transformation,
$$
L_{\epsilon}(\dot x,x,\theta)=\max_{y}\{\langle\dot x,y\rangle-G_{\epsilon}(x,y,\theta)\}.
$$
one obtains from the Hamiltonian $G_{\epsilon}$ the Lagrangian (1.1) defined in \cite{C17b}. 

Treated as the set in $T^*M$, it is shown in \cite{B07} that Mather set, Aubry set and Ma\~n\'e set are symplectic invariants. Denote by $\tilde{\mathcal M}_H(c)$ the Mather set of the Tonelli Hamiltonian $H: T^*M\to \mathbb R$ in the cohomology class $c\in H^1(M,\mathbb R)$.
\begin{theo}[\cite{B07}]\label{ThmBernard}
Let $\Phi: T^*M\to T^*M$ be a Hamiltonian diffeomorphism. Then one has
$$
\Phi\tilde{\mathcal{M}}_H(c)=\tilde{\mathcal{M}}_{H\circ\Phi}(\Phi^*c),\quad c\in H^1(M,\mathbb R).
$$
Similarly for Aubry set and Ma\~n\'e set.
\end{theo}
Applying the theorem, we find that the map $\Phi_{\epsilon F}$ does not induce the change of the structure of the Ma\~n\'e sets and the Aubry sets. The transformations $\mathscr{M}_0$ and (\ref{xuanzuan}) are linear, (\ref{qicihua}) and (\ref{energylevel}) are rescaling. Therefore, the conditions (1) and (2) in Definition \ref{chaindef1} remain unchanged under the coordinate transformations.

\section{The covering property}

We are going to show that the whole resonant path $\Gamma_{k'}$ can be covered by the disks where one obtains the normal form of (\ref{2dHmil}).

\begin{theo}[Covering property]\label{th4.1} Some $\epsilon_0>0$ exists such that for each $\epsilon\in(0,\epsilon_0]$ there exists a finite set of double resonant points $\{p''_i\in\Gamma_{k'}\}$ with the properties
\begin{enumerate}
  \item the period $T_i$ of the frequency $\partial h(p''_i)$ is not large than $T_i\le K^*\epsilon^{-\frac 13(1-3\kappa)}$ where $\kappa\in(0,\frac 16)$ and $K^*$ is independent of $\epsilon$;
  \item $\Gamma_{k'}$ is covered by the union of the disks $\{\|p-p''_i\|\le T_i^{-1}\epsilon^{\kappa}\}$.
\end{enumerate}
\end{theo}
\begin{proof}
If $\Gamma_{k'}$ is a set in $\mathbb{R}^n$ the theorem is proved in Chapter 3 of \cite{Lo} with the condition $\kappa<(3n+3)^{-1}$. We use their idea to prove the covering property under the condition $\kappa<\frac 16$. To do that, we use Dirichlet's approximation theorem.

For real $x$, let $[x]\in\mathbb{Z}$ denote the integer part and $\{x\}\in(0,1)$ denote the decimal part. So, one has
$$
x=[x]+\{x\},
$$
We use notation
$$
\|x\|_{\mathbb{Z}}=\min\{\{x\},1-\{x\}\}=\text{\rm dist}(x,\mathbb{Z}).
$$

\begin{pro}\label{normalpro1} {\rm (Dirichlet)} Let $\omega\in\mathbb{R}$ and $K>1$ be a real number. There exists an integer $k$, $1\le k<K$, such that
$$
\|k\omega\|_{\mathbb{Z}}\le K^{-1}.
$$
\end{pro}

For $\omega=(\omega_2,\omega_3)\in\mathbb{R}^2$, let $|\omega|=\max\{|\omega_2|,|\omega_3|\}$. Let $\mathbb{S}^1=\{\omega\in\mathbb{R}^2:|\omega|=1\}$ be the boundary of unit square, it has four sides.
By applying Dirichlet's approximation theorem, for any $\omega\in\mathbb{S}^1$ and any integer $K>0$ there exists $1\le k<K$ such that $\|k\min\{\omega_2,\omega_3\}\|_{\mathbb{Z}}\le K^{-1}$. In other words, given any $\omega\in\mathbb{S}^1$, some rational vector $\omega^*$ on the same side exists such that $T\omega^*\in\mathbb{Z}^2$ with $T\le K$ and
\begin{equation}\label{dis1}
\text{\rm dist}(T\omega,T\omega^*)=\|T\omega-T\omega^*\|\le K^{-1}.
\end{equation}
To apply the inequality, we notice that $\partial h(\Gamma_{k'})$ is a circle lying on certain 2-dimensional plane.

By the definition, $\Gamma_{k'}$ is a smooth circle. Under the transformation $\mathscr{M}_0$ introduced in the last section, $M_0^{-1}\partial h$ maps the circle $\Gamma_{k'}$ to a smooth circle $\Gamma_{\omega,k'}=M_0^{-1}\partial h(\Gamma_{k'})$ restricted on the plane $\{(\omega_1,\omega_2,\omega_3):\omega_1=0\}$.

Since $h$ is positive definite and $h(\Gamma_{k'})\equiv E>\min h$, each line $\{\lambda(0,\omega_2,\omega_3):\lambda\in\mathbb{R}\}$ intersects the circle $\Gamma_{\omega,k'}$ transversally and each $\omega\in\Gamma_{\omega,k'}$ determines a unique $\lambda_{\omega}>0$ such that $\lambda_{\omega}(\omega_2,\omega_3)\in\mathbb{S}^1$. Therefore, some number $d>0$ exists so that the distance between any two points $\omega=(0,\omega_2,\omega_3)$, $\omega^*=(0,\omega^*_2,\omega^*_3)\in\Gamma_{\omega,k'}$ is upper bounded by
$$
\|\omega-\omega^*\|\le d\|\lambda_\omega\omega-\lambda_{\omega^*}\omega^*\|.
$$
If $\omega^*$ is a rational vector with period $T$, the period of $\lambda_{\omega^*}\omega^*$ will be $\lambda_{\omega^*}^{-1}T$.

The map $\partial h$ establishes a diffeomorphism between $\Gamma_{k'}$ and $\Gamma_{\omega,k'}$. Given an integer $K>0$, it follows from (\ref{dis1}) that, for any rotation vector $\omega\in\Gamma_{\omega,k'}$, there exists some rational rotation vector $\omega^*\in\Gamma_{\omega,k'}$ such that $\lambda_\omega\omega$, $\lambda_{\omega^*}\omega^*$ lie on the same side, such that
$$
\|\omega-\omega^*\|\le d\|\lambda_\omega\omega-\lambda_{\omega^*}\omega^*\|\le \frac{d\lambda_{\omega^*}}{KT}.
$$
where $T>0$ is the period of $\omega^*$ such that $\lambda_{\omega^*}^{-1}T<K$.

For the ball $B_D\subset\mathbb{R}^3$, there are positive numbers $m'=m'(D)\ge m=m(D)>0$ such that
$$
m\|v\|^2\le\langle\partial^2h(y)v,v\rangle\le m'\|v\|^2, \qquad \forall\ y\in B_D,\ v\in\mathbb{R}^3.
$$
Because $h$ is assumed strictly convex, there exist exactly two points $y,y^*\in B_D$ such that $\partial h(y)=\omega$, $\partial h(y^*)=\omega^*$ and $\|y-y^*\|\le m^{-1}\|\omega-\omega^*\|$. Because $\lambda^{-1}_{\omega^*}T\le K$, the covering property $\|y-y^*\|\le T^{-1}\epsilon^{\kappa}$ is guaranteed if we choose
\begin{equation*}\label{normaleq8}
K=\frac {d\lambda_{\omega^*}}{m}\epsilon^{-\kappa}.
\end{equation*}
Again, because $\lambda^{-1}_{\omega^*}T\le K$, one has $T\le K^*\epsilon^{-\frac 13(1-3\kappa)}$ if $K=\frac {d\lambda_{\omega^*}}{m}\epsilon^{-\kappa}\le\frac {K^*}{\Lambda}\epsilon^{-\frac 13(1-3\kappa)}$. It holds for each $\epsilon\in(0,\epsilon_0]$ if $\epsilon_0$ satisfies the condition
$$
\epsilon_0^{\frac{1-3\kappa}3-\kappa}\le \frac{mK^*}{d\Lambda^2},
$$
where $\Lambda=\max_{\omega\in\Gamma_{\omega,k}}\lambda_{\omega}$. For $\kappa<\frac 16$, such $\epsilon_0>0$ exists. 
\end{proof}

\section{The finiteness of strong double resonance}
Because of Theorem \ref{th4.1}, the path of resonance $\Gamma_{k'}$ is covered by the discs $\{\|p-p''_i\|<T_i^{-1}\epsilon^{\kappa}\}$, where $\kappa<\frac 16$, $T_i\le K^*\epsilon^{-\frac 13(1-3\kappa)}$ is the period of the double resonance at $p''_i$, $K^*$ is independent of $\epsilon$. Therefore, the size of each disk is between $O(\epsilon^{1/3})$ and $O(\epsilon^{1/7})$. When the number $\epsilon>0$ decreases, the number of the disks increases. We are going to distinguish strong double resonant points from weak resonant points by the second resonant relation, and we will see that the number of strong double resonant points is finite, independent of $\epsilon$ in generic case.

We consider the resonant term (\ref{resonant term}) which takes the form
\begin{equation}\label{resonance}
Z(p,q)=Z_{k'}(p,\langle k',q\rangle)+Z_{k',k''_i}(p,\langle k',q\rangle,\langle k''_i,q\rangle)
\end{equation}
where $k'=k^\star$ or $k^*$, $k''_i$ is the additional resonant condition
\begin{equation}\label{decom}
\begin{aligned}
Z_{k'}&=\sum_{j\in\mathbb{Z}\backslash\{0\}}P_{jk'}(p)e^{j\langle k',q\rangle i}, \\
Z_{k',k''_i}&=\sum_{(j,l)\in\mathbb{Z}^2, l\neq 0}P_{jk'+lk''_i}(p)e^{(j\langle k',q\rangle+l\langle k''_i,q\rangle)i}.
\end{aligned}
\end{equation}
Because $|P_k|$ decrease fast as $\|k\|$ increases $|P_k|\le O(\|k\|^{-r})$, the term $Z_{k',k''_i}$ is treated as a small perturbation to $Z_{k'}$ provided $\|k''_i\|$ is large.

Treated as a function of $x=\langle k',q\rangle$, we consider $Z_{k'}(p,x)$ as a family of functions $Z_{k'}(p,\cdot)$: $\mathbb{T}\to\mathbb{R}$, where $p\in\Gamma_{k'}$ is treated as a parameter. Such an observation allows us to apply the result of \cite{Zh2}.
\begin{theo}[Theorem 1.1 of \cite{Zh2}]\label{ThmCZ}
{\it Let $F_{\lambda}:\mathbb{T}\to\mathbb{R}$ be a family of $C^4$-smooth functions so that $F_{\lambda}$ is Lipschitz in the parameter $\lambda\in[0,1]$. Then, there exists an open-dense set $\mathfrak{V}\subset C^r(\mathbb T,\mathbb{R})$ $(r\ge 4)$ such that for each $V\in \mathfrak{V}$ and each $\lambda\in [0,1]$, every global minimal point of $F_{\lambda}-V$ is non-degenerate.}
\end{theo}

To apply the theorem here, we notice that the path $\Gamma_{k'}$ induces a decomposition
$$
C^r(B_D\times \mathbb{T}^3,\mathbb{R})=C^r(B_D\times \mathbb{T},\mathbb{R})\oplus C^r(B_D\times \mathbb{T}^3,\mathbb{R})/C^r(B_D\times \mathbb{T},\mathbb{R})
$$
via
$$
P(p,q)=Z_{k'}(p,\langle k',q\rangle)+P'(p,q),\quad (k'=k^{\star},k^*),
$$
where $Z_{k'}$ is defined in (\ref{decom}) consisting of Fourier modes of $P$ in span$_\Z\{k'\}$, and $P'=P-Z_{k'}\in C^r(B_D\times \mathbb{T}^3,\mathbb{R})/C^r(B_D\times\mathbb{T},\mathbb{R})$.

Therefore, there exists an open-dense set $\mathfrak{V}\subset C^r(B_D\times \mathbb{T},\mathbb{R})$ and consequently an open-dense set $\mathfrak{P}=\mathfrak{V}\oplus C^r(B_D\times \mathbb{T}^3,\mathbb{R})/C^r(B_D\times\mathbb{T},\mathbb{R})\subset C^r(B_D\times\mathbb{T}^3)$ such that for all $P\in\mathfrak{P}$, it holds simultaneously for each $p\in\Gamma_{k'}$ that the resonant term $Z_{k'}(p,\cdot)$, treated as a function of $x$, is non-degenerate at its maximal point, namely,
the second derivative $\partial ^2_xZ_{k'}$ at its maximum is uniformly upper bounded below $0$.

Let $y_0$ be a vector such that $By_0=(0,1)^t$. A non-degenerate maximal point $x_{1,i}$ of $Z_{k'}(p''_i,\cdot)$ corresponds to a normally hyperbolic invariant cylinder (NHIC)
$$
\Pi_i=\{(x,y)\in\mathbb{T}^2\times\{\|y\|\le K_i^{-1}\epsilon^{\kappa-\frac 12}\}:x_1=x_{1,i},y=\lambda y_0,
\lambda\in\mathbb{R}\}
$$
of the Hamiltonian system $\frac 12\langle By,y\rangle+Z_{k'}(p''_i,x_1)$. In the $(p,q)$-coordinates, the cylinder passes through a neighborhood of the double resonant point $p''_i$. Such a phenomenon allows us to distinguish
strong double resonant points from weak ones by the existence of weakly invariant cylinder.

\begin{defi}[cf. \cite{B10}]
An open manifold is said to be weakly invariant for a flow if its vector field is tangent to the manifold.
\end{defi}

Applying the theorem of normally hyperbolic invariant manifold (NHIM), we see that, for the Hamiltonian $\frac 12\langle By,y\rangle+Z_{k'}+Z_{k',k''_i}$, there exists some weakly invariant cylinder $\Pi'_i$ lying in a small neighborhood of $\Pi_i\cap\{\|y\|\le K_i^{-1}\epsilon^{\kappa-\frac 12}-1\}$ provided $|k''_i|$ is large enough.

Indeed, if we introduce a cut-off $C^{\infty}$-function $\chi$: $[0,\infty)\to\mathbb{R}$ satisfying $\chi(\nu)=0$ for $\nu\ge K_i^{-1}\epsilon^{\kappa-\frac 12}$ and $\chi(\nu)=1$ for $\chi\le K_i^{-1}\epsilon^{\kappa-\frac 12}-1$. Then by applying the theorem of NHIM to the Hamiltonian $\frac 12\langle By,y\rangle+Z_{k'}(p''_i,x_1)+\chi(\|y\|)Z_{k',k''_i}(p''_i,x)$ we see that the cylinder $\Pi_i$ survives the small perturbation $\chi(\|y\|)Z_{k',k''_i}(p''_i,x)$ if $|k''_i|$ is sufficiently large. Restricted on the region $\{\|y\|\le K_i^{-1}\epsilon^{\kappa-\frac 12}-1\}$ the survived cylinder is obviously weakly invariant. Since the normally hyperbolic splitting still exists on the survived cylinder, we call it normally hyperbolic weakly invariant cylinder, or NHWIC for short.

\begin{defi}
A double resonance is said to be \emph{weak} if the double resonant term $Z_{k',k''_i}$ is so small that can be treated as a small perturbation
$$
\frac 12\langle By,y\rangle+Z_{k'}\to \frac 12\langle By,y\rangle+Z_{k'}+Z_{k',k''_i}
$$
such that one can apply the theorem of NHIM. In this case, the NHWIC survives the perturbation. Otherwise, the double resonance is said to be \emph{strong}.
\end{defi}

Although the number of the disks $\{\|p-p''_i\|<T_i^{-1}\epsilon^{\kappa}\}$ depends on $\epsilon$, we have

\begin{pro}\label{pro}
There exists a set $\mathfrak{P}$ open-dense in $\mathfrak{S}_1$ such that for each $P\in\mathfrak{P}$, the number of strong double resonances along $\Gamma_{k'}$ is finite and independent of $\epsilon$.
\end{pro}
\begin{proof}
For perturbation $P$, let $x_p$ be the maximal point of the single resonant term $Z_{k'}(p,\cdot)$ with respect to the variable $x$. By Theorem \ref{ThmCZ}, there exists an open-dense set $\mathfrak{P}'\subset C^r(B_D\times\mathbb{T}^3,\mathbb{R})$, for each $P\in\mathfrak{P}'$ there exists some $d_P>0$ such that the second derivative $\partial^2_xZ_{k'}(p,x_p)<-d_P$ holds for all $p\in\Gamma_{k'}$. Since $\lambda P\in\mathfrak{P}'$ for any $\lambda\ne 0$ if $P\in\mathfrak{P}'$, the restriction of $\mathfrak{P}'$ to $\mathfrak{S}_1$ is clearly open-dense. Since the double resonant term $Z_{k',k''_i}$ will be sufficiently small provided $\|k''_i\|$ is sufficiently large, we complete the proof.
\end{proof}

\begin{nota} Let $\Lambda_{\epsilon}$ denote the set of the subscripts $i$ such that the resonant circles $\Gamma_{k^\star}\cup\Gamma_{k^*}$ is covered by the disks
$$
\Gamma_{k^\star}\cup\Gamma_{k^*}\subset\cup_{i\in\Lambda_\epsilon}\{p\in\mathbb{R}^3:\|p-p''_i\|<T_i^{-1} \epsilon^{\kappa}\}
$$
where the period $T_i$ of the double resonant point $p''_i$ is not larger than $K^*\epsilon^{-(1-3\kappa)/3}$ with $\kappa\in(0,\frac 16)$.
Let $\Lambda_s\subset\Lambda_{\epsilon}$ be the subset so that the double resonance at $p''_i$ is strong if and only if $i\in\Lambda_s$.
\end{nota}

\section{Dynamics around double resonance}
Let us apply Lemma \ref{normalf} to the double resonant point $p''_i$, where the second resonant condition is denoted by $k''_i$, let $\bar k''_i=M_0^{-t}k''_i=(0,\bar k''_{i,2},\bar k''_{i,3})$. We denote by $G_{\epsilon,i}$ the normal form reduced from $H$ in a neighborhood of $p''_i$, which takes the form of (\ref{2dHmil})
\begin{equation}\label{2dHmili}
G_{\epsilon,i}(x,y,\theta)=\bar G_i(x,y)+ R_{\epsilon,i}(x,y,\vartheta(\theta)),
\end{equation}
where $\bar G_i=\frac 12\langle B_iy,y\rangle-V_i(x_1,|\bar k''_i|x_2)$, $V_i\in C^{r}$ is $2\pi$-periodic in $(x_1,|\bar k''_i|x_2)$ such that $\max V_i=0$, $\vartheta=\frac{\omega_{i,3}}{\sqrt{\epsilon}}\theta$, $R_{\epsilon,i}\in C^{r-2}(\mathbb{T}^2\times\Sigma'_{\epsilon,i}\times|\bar k''_i|\mathbb{T},\mathbb{R})$ with
$$
\{|y|\le (1-\delta')(\eta T_i)^{-1}\epsilon^{\kappa-\frac 12}\}\subset\Sigma'_{\epsilon,i}\subset\{|y|\le (1+\delta')(\eta T_i)^{-1}\epsilon^{\kappa-\frac 12}\},
$$
$\eta\in(0,1)$ is independent of the period $T_i$, the remainder $R_{\epsilon,i}$ is bounded by $a_0\epsilon^{\kappa}$ in $C^{r-2}$-topology if it is considered as a function of $(x,y,\vartheta)$.

Because the normal form (\ref{2dHmili}) is obtained by the transformation (\ref{xuanzuan}) where the matrix is denoted by $M_i$, the function $R_{\epsilon,i}$ respects the symmetry $M_i$. So we have
\begin{equation}\label{symmetryM}
G_{\epsilon,i}(x_1,x_2+\bar k''_{i,3}/\bar k''_{i,2},\vartheta+1)=G_{\epsilon,i}(x_1,x_2,\vartheta)
\end{equation}
Pull back to the original space of $(p,q)$, we have $V_i(x_1,|\bar k''_i|x_2)=V_i(\langle k',q\rangle,\langle k''_i,q\rangle)$.

Let $\alpha_{\epsilon,i}$ and $\beta_{\epsilon,i}$ denote the $\alpha$- and $\beta$-function for $G_{\epsilon,i}$ respectively, with which one defines the Fenchel-Legendre transformation $\mathscr{L}_{\alpha}$: $H_1(M,\mathbb{R})\to H^1(M,\mathbb{R})$ by
$$
c\in\mathscr{L}_{\alpha}(\omega)\ \ \iff \ \ \alpha(c)+\beta(\omega)=\langle c,\omega\rangle.
$$
By the definition, the $\beta$-function is the Fenchel-Legendre dual of the $\alpha$-function. By adding a constant to $G_{\epsilon,i}$, we can assume $\min\alpha_{\epsilon,i}=0$.

To understand the dynamics around strong double resonances, we apply the results obtained in \cite{C17a,C17b,CZ}.
\begin{theo}\label{C17bTh}
There exists a residual set $\mathfrak{V}_i\subset C^r(\mathbb{T}^2,\mathbb{R})$ with $r\ge 2$. Each $V_i\in\mathfrak{V}_i$ is associated with some positive numbers $\Delta_{V_i}, \epsilon_i$ such that for any $\xi\in (0,\Delta_{V_i})$, $\epsilon\in[0,\epsilon_i]$ the circle $\alpha^{-1}_{\epsilon,i}(\xi)$ establishes a transition chain $($of cohomology equivalence$)$. These circles make up an annulus $\mathbb{A}_{i}$ surrounding the set $\mathbb{F}_{i}=\{c:\alpha_{\epsilon,i}(c)=0\}$.
\end{theo}
\begin{proof}
If we are satisfied with the property that $V\in\mathfrak{V}_i$ is $2\pi$-periodic both in $x_1$ and in $x_2$, it is just an application of Theorem 1.1 of \cite{C17b}. We are going to show that the set $\mathfrak{V}_i$ is residual in the space of $C^r$-functions which are $2\pi$-periodic in $(x_1,|\bar k''_i|x_2)$. For the purpose, we introduce a canonical transformation $\mathscr{T}$: $(x_1,|\bar k''_i|x_2)\to(\phi_1,\phi_2)$, $(y_1,|\bar k''_i|^{-1}y_2)\to(I_1,I_2)$ and get the Hamiltonian from (\ref{2dHmili})
\begin{equation}\label{averaged}
\bar G'_i=\mathscr{T}_*\bar G'_i=\frac 12\langle B'_iI,I\rangle-V_i(\phi),\qquad (\phi,I)\in\mathbb{T}^2\times\mathbb{R}^2,
\end{equation}
where $B'_i=\mathrm{diag}(1,|\bar k''_i|)B_i\mathrm{diag}(1,|\bar k''_i|)$.
The theorem holds for $\bar G'_i$ under the condition that $\mathfrak{V}_i$ is residual in $C^r(\mathbb{T}^2,\mathbb{R})$.

Pull $\bar G'_i$ back to the space of $(x,y)$, we see that the theorem holds because $G_{\epsilon,i}$ is a small perutrbation of $\mathscr{T}^*\bar G'_i$: $G_{\epsilon,i}=\mathscr{T}^*\bar G'_i+R_{\epsilon,i}$ and Ma\~n\'e set is upper semi-continuous with respect to small perturbation.
\end{proof}

Given an irreducible class $g\in H_1(\mathbb{T}^2,\mathbb{Z})\backslash\{0\}$, let  $\mathbb{C}_{i,g}=\cup_{\nu>0}\mathscr{L}_{\alpha_{\epsilon,i}}(\nu g)$. To consider the Aubry set of $G_{\epsilon,i}$ for $c\in\mathbb{C}_{i,g}$, we consider the truncated Hamiltonian $\bar G'_i$ of (\ref{averaged})
and let $\bar\alpha_i$ be the $\alpha$-function for $\bar G'_i$. For $c\in\mathscr{L}_{\bar\alpha_{i}}(\lambda g)$, each $c$-minimal orbit is periodic. To stress its topological information, we also call it $\lambda g$-minimal orbit. The parameter $\lambda_\ell$ is called bifurcation point, if there are two or more $\lambda_\ell g$-minimal orbits. Applying Theorem 2.1 of \cite{CZ}, Theorem 3.1 and the argument of Section 4 of \cite{C17a} we have

\begin{pro}\label{pro6.1}
Given an irreducible class $g'\in H_1(\mathbb{T}^2,\mathbb{Z})\backslash\{0\}$ and $\lambda_0>0$, there exists an open-dense set $\mathfrak{V}_i\subset C^{r}(\mathbb{T}^2,\mathbb{R})$ with $r\ge 5$ such that for each $V_i\in\mathfrak{V}_i$, it holds simultaneously for all $\lambda\in[\lambda_0,\infty)$ that the Mather set of $\bar G'_i$ for each $c\in\mathscr{L}_{\bar\alpha_{i}}(\lambda g')$ consists of hyperbolic periodic orbits. Indeed, except for finitely many $\{\lambda_\ell\}$ the Mather set is made up two hyperbolic periodic orbits, for all other $\lambda\in[\lambda_0,\infty)$ the Mather set contains exactly one hyperbolic periodic orbit.
\end{pro}

Therefore, for each $V_i\in\mathfrak{V}_i$, there are finitely many bifurcation points $\lambda_0<\lambda'_1<\cdots<\lambda'_m$, the number $m$ is independent of $\epsilon$.  At each bifurcation point $\lambda'_\ell$, there exist exactly two $\lambda'_\ell g'$-minimal orbits of $\bar G'_i$. So, for $\lambda'\ge\lambda_0$, all $\lambda' g'$-minimal orbits make up $m+1$ pieces of NHIC. Let $\bar\Pi'_{i,\ell}$ denote the cylinder made up of $\lambda' g'$-minimal orbits of $\bar G_i$ for all $\lambda'\in[\lambda'_\ell,\lambda'_{\ell+1}]$ for $\ell<m$, let $\bar\Pi'_{i,m}$ be the cylinder made up of $\lambda'g'$-minimal orbits of $\bar G_i$ for all $\lambda'\in[\lambda'_m,\lambda'_{\epsilon}]$ where $\lambda'_\epsilon$ is chosen such that $\max_\theta|(I_1,|\bar k''_i|I_2)(\theta)|=(1-\delta')(\eta T_i)^{-1}\epsilon^{\kappa-\frac 12}$ holds for the $\lambda'_{\epsilon}g'$-minimal orbits $(I(\theta),\phi(\theta))$.

Return back to the coordinates $(x,y)$, each orbit of $\bar G'_i$ is pulled back to the orbit of $\mathscr{T}^*\bar G'_i$ in the way
$$
(x_1,x_2,y_1,y_2)=\Big(\phi_1,\frac 1{|\bar k''_i|}(\phi_2+2\ell\pi),I_1,|\bar k''_i|I_2\Big)
$$
for $\ell=0,1,\cdots,|\bar k''_i|-1$. Consequently, each cylinder $\bar\Pi'_{i,\ell}$ is pulled back to a cylinder $\bar\Pi_{i,\ell}$ of $\bar G_i$, modulo a shift $(x,y)\to (x+(0,2\pi/|\bar k''_i|),y)$. Each cylinder is made up of $\lambda g$-minimal orbits of $\bar G_i$ with $\lambda g=\frac{\lambda'}{|\bar k''_i|}(|\bar k''_i|g_1,g_2)$ if $\lambda'g'=\lambda'(g_1,g_2)$.

Let $E_\ell$ be the energy such that $\lambda_\ell g$-minimal orbit lies in the energy level set $\bar G_i^{-1}(E_\ell)$. Since each $\lambda_\ell g$-minimal orbit is hyperbolic, each cylinder $\bar\Pi_{i,\ell}$ can be extended by the hyperbolic orbits lying in the level set $\bar G_i^{-1}(E)$ with $E\in[E_{\ell+1},E_{\ell+1}+2d]\cup[E_\ell-2d,E_\ell]$. Since there are finitely many bifurcation points, such a number $d>0$ exists.

Under the time-periodic perturbation $G_{\epsilon,i}=\bar G_i+R_{\epsilon,i}$, major part of these cylinders survives,  {\it weakly invariant} for the Hamiltonian flow $\Phi^{\theta}_{G_{\epsilon,i}}$. Notice that $\vartheta=\omega_{i,3}\sqrt{\epsilon}^{-1}\theta$ and $G_{\epsilon,i}$ is $2\pi$ in $x$, $2|\bar k''_i|\pi$ in $\vartheta$ and symmetric for $M_i$, see (\ref{symmetryM}), we introduce a shift
\begin{equation}\label{shift}
\sigma_i: \ (x,y,\theta)\to\Big(x+\Big(0,2\pi\frac{\bar k''_{i,3}}{\bar k''_{i,2}}\Big),y,\theta+\frac{\sqrt{\epsilon}}{\omega_{i,3}}\Big),
\end{equation}
then $\sigma_i^*G_{\epsilon,i}=G_{\epsilon,i}$. Let $\tilde\Pi_{i,E_\ell-d,E_{\ell+1}+d}=(\bar\Pi_{i,\ell} \times\frac{|\bar k''_i|\sqrt{\epsilon}}{\omega_{i,3}}\mathbb{T})\cap\{\bar G_i\in[E_\ell-d,E_{\ell+1}+d]\}$.

\begin{pro}
For sufficiently small $\epsilon$, there is a cylinder $\tilde\Pi^{\epsilon}_{i,E_\ell-d,E_{\ell+1}+d}$ modulo the shift $\sigma_i$, which is weakly invariant and normally hyperbolic for the flow $\Phi_{G_{\epsilon,i}}^{\theta}$. The cylinder lies in a small neighborhood of $\tilde\Pi_{i,E_\ell-d,E_{\ell+1}+d}$.
\end{pro}
\begin{proof}
We modify the Hamiltonian $G_{\epsilon,i}$. Let $\rho$ be a $C^2$-function such that $\rho(\mu)=1$ for $\mu\ge 1$ and $\rho(\mu)=0$ for $\mu\le 0$. By defining $\rho_1(x,y)=\rho((\bar G_i(x,y)-E^-_\ell+2d)/d)$, $\rho_2(x,y)=1-\rho((\bar G_i(x,y)-E_{\ell+1}-d)/d)$
we introduce
\begin{equation}\label{modification}
G'_{\epsilon,i}=\begin{cases}
\bar G_i+\rho_1R_{\epsilon,i}, &\mathrm{if}\  \bar G_i(x,y)\in[E^-_\ell-2d,E_\ell-d],\\
\bar G_i+\rho_2R_{\epsilon,i},&\mathrm{if}\  \bar G_i(x,y)\in[E_{\ell+1}+d,E_{\ell+1}+2d],\\
G_{\epsilon,i}, &\mathrm{elsewhere}.
\end{cases}
\end{equation}
Because $\|G'_{\epsilon,i}-\bar G_i\|_{C^2}\ll1$ if $|y|\le O(\epsilon^{\kappa-\frac 12})$ and $\epsilon\ll1$, it follows that the NHIC $\tilde\Pi_{i,\ell}$ survives the perturbation $\bar G_i\to G'_{\epsilon,i}$, denoted by $\tilde\Pi^{\epsilon}_{i,\ell}$.

Let $\tilde\Pi^{\epsilon}_{i,E_\ell-d,E_{\ell+1}+d}=\tilde\Pi_{i,\ell}^{\epsilon}\cap\{\bar G_i\in[E_\ell-d,E_{\ell+1}+d]\}$, which is a weakly invariant for $\Phi_{G_{\epsilon,i}}^{\theta}$ because $G_{\epsilon,i}=G'_{\epsilon,i}$ when they are restricted in the region $\{(x,y):\bar G_i(x,y)\in[E_\ell-d,E_{\ell+1}+d]\}\times\frac{|\bar k''_i|\sqrt{\epsilon}}{\omega_{i,3}}\mathbb{T}$. The normal hyperbolicity is obvious.

Due to the symmetry (\ref{symmetryM}), the Hamiltonian vector field of $\Phi^{\theta}_{G_{\epsilon,i}}$ is invariant under the shift $\sigma_i$. It guarantees the invariance of $\tilde\Pi^{\epsilon}_{i,E_\ell-d,E_{\ell+1}+d}$ under the shift $\sigma_i$.
\end{proof}

Notice that $\lambda_{\epsilon}>0$ is set such that $\max_\theta|y(\theta)|=(1-\delta')(\eta T_i)^{-1}\epsilon^{\kappa-\frac 12}$ holds for the $\lambda_{\epsilon}g$-minimal orbits $(y(\theta),x(\theta))$. By applying Theorem 1.2 of \cite{C17b}, we have

\begin{theo}\label{C17aTh}
Given a class $g\in H_1(\mathbb{T}^2,\mathbb{R})$ and small $E_0>0$, there is an open-dense set $\mathfrak{V}_i\subset C^{r}(\mathbb{T}^2,\mathbb{R})$ $(r\ge 5)$. For each $V_i\in\mathfrak{V}_i$, there exists $\epsilon_0>0$ such that for each $\epsilon\in(0,\epsilon_0)$

$1)$ there are finitely many NHWICs for the flow $G_{\epsilon,i}$:
$\tilde\Pi^{\epsilon}_{i,E_\ell-d,E_{\ell+1}+d}$ $(\ell=0,\cdots,m)$ and $\tilde\Pi^{\epsilon}_{i,E_m-d,E_{\epsilon}}$, modulo the shift $\sigma_i$, where the integer $m$, the numbers $E_1<E_2<\cdots<E_m$, $d>0$ and the normal hyperbolicity of each cylinder are all independent of $\epsilon$, $E_{\epsilon}=\bar\alpha_i(\lambda_{\epsilon}g)=O(\epsilon^{2\kappa-1})$ for small $\epsilon>0$;

$2)$ some $E_{\ell,\epsilon}\to E_\ell$ as $\epsilon\to 0$ exists such that for each $c\in\mathbb{C}_{i,g}$
\begin{enumerate}
  \item if $\alpha_{\epsilon,i}(c)\in (E_{\ell,\epsilon}, E_{\ell+1,\epsilon})$, the Aubry set lies on $\tilde\Pi^{\epsilon}_{i,E_\ell-d,E_{\ell+1}+d}$ modulo $\sigma_i$;
  \item if $\alpha_{\epsilon,i}(c)=E_{\ell,\epsilon}$, the Aubry set contains at least two connected components, one is on $\tilde\Pi^{\epsilon}_{i,E_{\ell-1}-d,E_{\ell}+d}$ and the other one is on $\tilde\Pi^{\epsilon}_{i,E_{\ell}-d,E_{\ell+1}+d}$ modulo $\sigma_i$;
  \item if $\alpha_{G_{\epsilon}}(c)\in (E_{m,\epsilon},E_{\epsilon})$, the Aubry set lies on $\tilde\Pi_{E_{m}-d,E_{\epsilon}}$ modulo $\sigma_i$.
\end{enumerate}
\end{theo}
\noindent{\bf Remark}. Because of the symmetry $M_i$, all shifts of the cylinder $\tilde\Pi^{\epsilon}_{i,E_\ell-d,E_{\ell+1}+d}$ are projected to the same cylinder, if we pull them back to the space of $(p,q)$.

At a strong double resonant point $p''_i,\ i\in\Lambda_s$, we consider two classes $g^-_i,g^+_i$ which determines two channels $\mathbb{C}_i^\pm=\mathbb{C}_{i,g^\pm}$. The choice of $g^-,g^+$ depends on where $p''_i$ is.
There are three possible locations of a strong double resonant point $p''_i$ along $\Gamma_{k^*}\cup \Gamma_{k^\star}$:
\begin{enumerate}
  \item the point $p''_i$ is on the path $\Gamma_{k^{\star}}$ where $\Gamma_{k^{\star}}$ does not intersect with $\Gamma_{k^*}$;
  \item the point $p''_i$ is on the path $\Gamma_{k^*}$ where $\Gamma_{k^*}$ does not intersect with $\Gamma_{k^{\star}}$;
  \item $p''_i\in\Gamma_{k^{\star}}\cap\Gamma_{k^*}$.
\end{enumerate}
After the linear coordinate transformations $\mathscr{M}_0$ and $\mathscr{M}_i$ defined in (\ref{xuanzuan}), we have that the first two components of the frequency $\partial h(p)$, i.e. the frequency of the system $G_{\epsilon,i}$,  is proportional to $(0,1)$ along $\Gamma_k,\ k=k^*,k^\star$ in the first two cases. In the third case, $\partial h(p)$ is proportional to $(0,1)$ along $\Gamma_{k^{\star}}$ and to $(1,0)$ along $\Gamma_{k^*}$.

At a double resonance $p_i''$, the first two components of the frequency $\partial h(p''_i)$ vanish after the linear transforms. So by {\it crossing a strong double resonance},  we mean that there exists an orbit of the system $G_{\epsilon,i}$, such that along the orbit, its ``frequency" changes from being in the set $\{(0,\nu),\ \nu>0\}$ to the set $\{(0,\nu),\ \nu<0\}$ in the first two cases, and changes
from being in the set $\{(0,\nu),\ \nu>0\}$ to the set $\{(\nu,0),\ \nu>0\}$ in the third case. Since $G_{\epsilon, i}$ is not nearly-integrable, here the role of ``frequency" of the system $G_{\epsilon,i}$ will be played by the rotation vector of the Mather sets shadowed by the orbit.

We introduce the homology classes $g_i^\pm\in H_1(\T^2,\Z)$ as follows:
\begin{enumerate}
  \item
  $g_i^+=(0,1)$ and $g_i^-=(0,-1)$ in the first two cases,
  \item $g_i^+=(0,1)$ and $g_i^-=(1,0)$ in the third case.
\end{enumerate}

By adding a constant to $G_{\epsilon,i}$ we assume $\min\alpha_{\epsilon,i}=0$. Applying Theorem \ref{C17bTh} and \ref{C17aTh} we obtain the following result for the dynamics around the strong double resonance.
\begin{theo}\label{localchainS}
There exists an open-dense set $\mathfrak{V}_i\subset C^{r}(\mathbb{T}^2,\mathbb{R})$ $(r\ge 5)$ such that for each $V_i\in\mathfrak{V}_i$, there exist $0<E_{i,0}<\Delta_{V_i}$, $\Delta_i'>0$ and $\epsilon_i>0$ such that $\forall$ $\epsilon\in (0,\epsilon_i]$ the following holds.

$1)$ there is an annulus $\mathbb{A}_{i}$ around a double resonant point $p''_i$, made up of the circles $\{c\in\alpha_{\epsilon,i}^{-1}(E):E\in(0,\Delta_{V_i})\}$, each of which is a path of cohomology equivalence;

$2)$ the following two channels are connected by the annulus $\mathbb{A}_{i}$
\begin{equation*}\label{strongC}
\mathbb{C}_i^{\pm}=\cup_{\lambda\in[\lambda_i^{\pm},\bar\lambda_i^{\pm}]}\mathscr{L}_{\alpha_{\epsilon,i}}(\lambda g_i^{\pm}), \qquad\ g_i^{\pm}\in H_1(\mathbb{T}^2,\mathbb{Z}),
\end{equation*}
where $\bar\lambda_i^{\pm},\lambda_i^{\pm}>0$ satisfy $\alpha_{\epsilon,i}(\mathscr{L}_{\alpha_{\epsilon,i}}(\lambda_i^\pm g_i^{\pm}))=E_{i,0}<\Delta_{V_i}$ and
\begin{enumerate}
\item for $c\in\mathscr{L}_{\alpha_{\epsilon,i}}(\bar\lambda_i^\pm g_i^{\pm})$ it holds for every orbit $\{(x(\theta),y(\theta)),\ \theta\in \mathbb R\}$ in $\tilde{\mathcal{A}}(c)$ that
    $$
    (1-\delta')(\eta T_i)^{-1} \epsilon^{\kappa-\frac 12} -\Delta'_i\le|y(\theta)|\le(1-\delta')(\eta T_i)^{-1}\epsilon^{\kappa-\frac 12};
    $$
\item for each $c\in\mathbb{C}_i^{\pm}$ except for finitely many classes $\{c_j\}$, the Aubry set $\tilde{\mathcal{A}}(c)$ lies in certain cylinder $\tilde\Pi^{\epsilon}_{i,E_\ell-d,E_{\ell+1}+d}$, modulo the shift $\sigma_i$, while $\tilde{\mathcal{A}}(c_j)$ lies in two of the cylinders. For each orbit $\{(x(\theta),y(\theta)),\ \theta\in\mathbb\R\}$ in the Aubry set one has
    $$
    \mathrm{Osc}|y(\theta)|=\max|y(\theta)-y(\theta')|\le \Delta'_i.
    $$
\end{enumerate}
\end{theo}
\begin{proof}
We only need to verify the estimate on the oscillation of $y(\theta)$ if $(x(\theta),y(\theta))$ is an orbit in the Aubry sets. Let us consider the problem for $\bar G_i$ first.

We claim that some constant $\Delta'_i>0$ exists such that it holds along any $\lambda g$-minimal orbit $(\bar y(\theta),\bar x(\theta))$ of $\bar G_i$ that
\begin{equation}\label{fluctuation}
|\bar y(\theta)-\bar y(\theta')|\le\frac 12\Delta'_i.
\end{equation}
Indeed, each $\lambda g$-minimal orbit is entirely contained in certain energy level set $\bar G_i^{-1}(E)$. The higher the energy increases, the shorter the period becomes. Let $E_1$ be the energy that the set $\bar G_i^{-1}(E_1)$ contains $\lambda g$-minimal orbit with period 1. One obtains from the equation $\dot y=\partial V_i(x)$ that $|\bar y(\theta)-\bar y(\theta')|\le\max_{x\in\mathbb{T}^2}\{|\partial_{x_1}V_i(\bar x(\theta))|,|\bar k''_i||\partial_{x_2} V_i(\bar x(\theta))|\}$ if $(x(\theta),y(\theta))\in\bar G_i^{-1}(E)$ with $E\ge E_1$. Let $m_i$ be the smallest eigenvalue of $B_i$, one has $\|y\|\le \frac 1{m_i}\sqrt{E+V_i(x)}$ if $(x,y)\in \bar G_i(E)$. Therefore, the estimate (\ref{fluctuation}) holds if we set
\begin{equation}\label{6.7}
\Delta'_i=\max\Big\{2\max_{x\in\mathbb{T}^2}\{|\partial_{x_1}V_i(\bar x(\theta))|,|\bar k''_i||\partial_{x_2} V_i(\bar x(\theta))|\},\frac 4{m_i}\sqrt{E_1+\max_xV_i(x)}\Big\}.
\end{equation}

Next, we consider $G_{\epsilon,i}$ which is a $O(\epsilon^{\kappa})$-perturbation of $\bar G_i$. Each $\theta$-section of NHIC of $G_{\epsilon,i}$ is located in $O(\epsilon^{\kappa})$-neighborhood of NHIC of $\bar G_i$. By Proposition 5.2 of \cite{C17a}, the Aubry set $\tilde{\mathcal{A}}(c)$ does not hit the level set $G_{\epsilon,i}^{-1}(E\pm \epsilon^{\frac 13\kappa})$ if $c\in\alpha^{-1}_{\epsilon,i}(E)\cap \cup_\lambda\mathscr{L}_{\alpha_{\epsilon,i}}(\lambda g)$ and $E-\min\alpha_{\epsilon,i}\ge O(1)$. Because the Aubry set $\tilde{\mathcal{A}}(c)$ for $G_{\epsilon,i}$ stays in the cylinder, it falls into $O(\epsilon^{\frac 13\kappa})$-neighborhood of the Aubry set for $\bar G_i$. 
For any $y(\theta)$ there exists $\bar y(\theta^*)$ such that $|y(\theta)-\bar y(\theta^*)|<\frac 14\Delta'_i$ provided $\epsilon>0$ is suitably small and it holds for any orbit $(x(\theta),y(\theta))\in\tilde{\mathcal{A}}(c)$ with $c\in\mathbb{C}_i^{\pm}$ that
\begin{equation*}
\mathrm{Osc}|y(\theta)|=\max_{\theta,\theta'}|y(\theta)-y(\theta')|\le\max_{\theta,\theta'}|\bar y(\theta)-\bar y(\theta')|+2O(\epsilon^{\frac 13\kappa})\le\frac34 \Delta'_i.
\end{equation*}
If we set $\bar\lambda^\pm_i$ such that $\max_\theta|\bar y(\theta)|=(1-\delta')(\eta T_i)^{-1}\epsilon^{\kappa-\frac 12}-\frac 14\Delta'_i$ holds along $\bar\lambda^\pm_ig_i^{\pm}$-minimal orbit of $\bar G_i$, one obtains that
$$
(1-\delta')(\eta T_i)^{-1}\epsilon^{\kappa-\frac 12}-\Delta'_i\le |y(\theta)|\le (1-\delta')(\eta T_i)^{-1}\epsilon^{\kappa-\frac 12}
$$
if $(x(\theta),y(\theta))$ lies in $\tilde{\mathcal{A}}(c)$ with $c\in\mathscr{L}_{\alpha_{\epsilon,i}}(\bar\lambda_i^\pm g_i^{\pm})$. 
\end{proof}

For weak double resonant point $p''_i$, the truncated part of $G_{\epsilon,i}$ is
\begin{equation}\label{weakterm}
\bar G_i=\frac 12\langle B_iy,y\rangle-V'_i(x_1)-V''_i(x_1,|\bar k''_i|x_2),
\end{equation}
where $V'_i(x_1)=-Z_{k'}(p''_i,x_1)$, $V''_i(x_1,|\bar k''_i|x_2)=-Z_{k',k''_i}(p''_i,x_1,|\bar k''_i|x_2)$ and the term $V''_i$ is treated as a small perturbation.
\begin{theo}\label{localchainw}
There exists an open-dense set $\mathfrak{V}_{k'}\subset C^{r}(B_d\times\mathbb{T},\mathbb{R})$ $(r\ge 5)$, for each $Z_{k'}\in\mathfrak{V}_{k'}$, there exist $\Delta'_i>0$ and $\epsilon_{k'}>0$ which are independent of $k''_i$ such that for each $\epsilon\in (0,\epsilon_{k'}]$ and each $i\in\Lambda_{\epsilon}\backslash\Lambda_s$ there is a channel
\begin{equation}\label{weakC}
\mathbb{C}^w_i=\cup_{\lambda\in[-\lambda_i,\bar\lambda_i]}\mathscr{L}_{\alpha_{\epsilon,i}}(\lambda g)\subset H^1(\mathbb T^2,\mathbb R),\qquad g=(0,1)
\end{equation}
with the properties that
\begin{enumerate}
  \item for each $c\in\mathbb{C}^w_i$, the Aubry set $\tilde{\mathcal{A}}(c)$ lies on some NHWIC entirely contained in the region $\Sigma'_{\epsilon,i}$, modulo the shift $\sigma_i$. Along each orbit $\{(x(\theta),y(\theta)),\ \theta\in \R\}$ in the Aubry set one has
      $$
      \mathrm{Osc}|y(\theta)|=\max_{\theta,\theta'}|y(\theta)-y(\theta')|\le \Delta'_i,
      $$
  \item the numbers $\lambda_i$ and $\bar\lambda_i>0$ are chosen such that
  for $c\in\mathscr{L}_{\alpha_{\epsilon,i}}(\bar\lambda_i g)\cup \mathscr{L}_{\alpha_{\epsilon,i}}(-\lambda_i g)$ it holds for $c$-minimal orbit $\{(x(\theta),y(\theta)),\ \theta\in \R\}$ that
      $$
      (1-\delta')(\eta T_i)^{-1}\epsilon^{\kappa-\frac 12} -\Delta'_i\le|y(\theta)|\le(1-\delta')(\eta T_i)^{-1}\epsilon^{\kappa-\frac 12}.
      $$
\end{enumerate}
\end{theo}
\begin{proof}
According to Theorem \ref{ThmCZ}, there exists an open-dense set $\mathfrak{V}_{k'}\subset C^{r}(B_d\times\mathbb{T},\mathbb{R})$, for each $Z_{k'}\in\mathfrak{V}_{k'}$ it holds simultaneously for all $p\in\Gamma_{k'}$ that the maximal point of $Z_{k'}(p,\cdot)$ in $x$ is non-degenerate, namely, the second derivative of $Z_{k'}(p,\cdot)$ in $x$ at the maximal point is uniformly upper-bounded below zero for all $p\in\Gamma_{k'}$.

In this case, the Hamiltonian system $\frac 12\langle By,y\rangle -V'_i(x_1)$ admits a normally hyperbolic invariant cylinder made up of the minimal periodic orbits of type $(0,1)$. As the second resonant term $V''_i(x_1,|\bar k''_i|x_2)$ and the remainder $R_{\epsilon,i}$ are small, the weakly invariant cylinder survives the perturbation $\frac 12\langle By,y\rangle -V'_i\to \frac 12\langle By,y\rangle -V'_i-V''_i+R_{\epsilon,i}$.

To study the oscillation of $y(t)$, we also consider the truncation $\bar G_i=\frac 12\langle By,y\rangle -V'_i(x_1)-V''_i(x_1,|\bar k''_i|x_2)$ first. Let $(\bar x(\theta),\bar y(\theta))$ be the $\lambda g$-minimal orbit, then the second component of $\bar y(\theta)$ satisfies the following relation
$$
\bar y_2(\theta')-\bar y_2(\theta)=\int_{\theta}^{\theta'}|\bar k''_i|\partial_2 V''_i(x_1(\theta),|\bar k''_i|x_2(\theta)) d\theta.
$$
Although the number $|\bar k''_i|$ approaches infinity if the second resonant condition becomes weaker, it does not make trouble to control the oscillation of $\bar y(\theta)$. Indeed, the term $V''_i(x_1,|\bar k''_i|x_2)= Z_{k'_i,k''_i}(p_i,x_1,|\bar k''_i|x_2)$, which is defined in (\ref{decom}). Because $|P_k|$ decrease fast as $\|k\|$ increases: $|P_k|\le O(\|k\|^{-r})$, one has $|\bar k''_i||\partial_2 V''_i|\to 0$ as $\|\bar k''_i\|\to\infty$.

The rest of the proof is similar to the proof of Theorem \ref{localchainS}.
\end{proof}

\section{The estimate of the deviation of Aubry set}
For integrable Hamiltonian $h$, the location of its $c$-minimal orbits is clear. Each $c$-minimal orbit is nothing else but the orbit $(p(t)=c,q(t)=\partial h(c)t+q_0)$. We want to know the deviation of Aubry set when $h$ is under small perturbation $h\to H=p+\epsilon P$.

For nearly integrable Hamiltonian $H=h+\epsilon P$ with convex $h$, one has (see Formulae (4.3) and (4.4) of \cite{C11})
\begin{equation}\label{estimate}
|\tilde\alpha_H(\tilde c)-\tilde\alpha_h(\tilde c)|<\epsilon\|P\|,\qquad |\tilde\beta_H(\rho)-\tilde\beta_h(\rho)|<\epsilon\|P\|
\end{equation}
where $\|P\|=\max_{(p,q)\in B_D\times\mathbb{T}^3}|P(p,q)|$ denote the $C^0$-norm of $P$, $\tilde\alpha_H$ and $\tilde\alpha_h$ denote the $\alpha$-function for $H$ and $h$, $\tilde\beta_H$ and $\tilde\beta_h$ denote the $\beta$-function for $H$ and $h$ respectively.

Since $H^1(\mathbb{T}^3,\mathbb{R})=\mathbb{R}^3$, we treat $\tilde c\in H^1(\mathbb{T}^3,\mathbb{R})$ as a point in $\mathbb{R}^3$. Since $h$ is positive definite, there exists $m>0$ such that
$$
h(p')-h(p)\ge\langle \partial h(p),p'-p\rangle+\frac m2\|p'-p\|^2.
$$
\begin{lem}\label{size}
Let $p=\partial h(\omega)\in B_D$, then the set $\mathscr{L}_{\tilde\alpha_H}(\omega)$ falls into $C_s\sqrt{\epsilon}$-neighborhood of $p$ with $C_s\le 2\sqrt{\|P\|/m}$, namely
$$
\mathrm{dist}(\mathscr{L}_{\tilde\alpha_H}(\omega),p)\le C_s\sqrt{\epsilon}.
$$
\end{lem}
\begin{proof}
Assume $\tilde c+p\in\mathscr{L}_{\tilde\alpha_H}(\omega)$, then by the definition one has
\begin{align*}
\langle \omega,\tilde c+p\rangle&=\tilde\beta_H(\omega)+\tilde\alpha_H(\tilde c+p)\ge \tilde\beta_h(\omega)+\tilde\alpha_h(\tilde c+p)-2\|P\|\epsilon\\
&\ge\tilde\beta_h(\omega)+\tilde\alpha_h(p)+\langle\omega,\tilde c\rangle+\frac m2\|\tilde c\|^2-2\|P\|\epsilon\\
&=\langle\omega,\tilde c+p\rangle+\frac m2\|\tilde c\|^2-2\|P\|\epsilon
\end{align*}
from which one obtains $\|\tilde c\|\le C_s\sqrt{\epsilon}$.
\end{proof}

Small perturbation may induce small rescaling of the rotation vector when both $h$ and $H$ are restricted on the level set with the same energy.
\begin{lem}\label{rescalingomega}
Given a rotation vector $\omega\neq 0$, let $h(\partial h^{-1}(\omega))=\tilde\alpha_H(\mathscr{L}_{\tilde\alpha_H} (\nu\omega))$, then some constant $C_r>0$ exists such that $|\nu-1|\le C_r\sqrt{\epsilon}$.
\end{lem}
\begin{proof}
Let $\nu$ be a number close to $1$. For each rotation vector $\omega$, $\exists$ unique $p,p_\nu$ such that $\omega=\partial h(p)$ and $\nu\omega=\partial h(p_\nu)$. It follows that
$$
p-p_\nu=(1-\nu)B_\nu^{-1}\omega, \qquad \|p-p_{\nu}\|\le |1-\nu|\|\omega\|m^{-1},
$$
where $B_\nu=\partial^2h(\lambda_\nu p+(1-\lambda_\nu)p_\nu)$ is positive definite, $\lambda_{\nu}\in[0,1]$. For $p_\nu+\tilde c\in\mathscr{L}_{\tilde\alpha_H}(\nu\omega)$, one obtains from the relation $|\tilde\alpha_H-\tilde\alpha_h|\le\epsilon\|P\|$ and $\|\tilde c\|\le C_s\sqrt{\epsilon}$ that
$$
\begin{aligned}
\epsilon\|P\|&\ge|\tilde\alpha_H(p_\nu+\tilde c)-\tilde\alpha_h(p_\nu+\tilde c)|=|h(p)-h(p_\nu+\tilde c)|\\
&\ge|\langle\omega,p_{\nu}-p+\tilde c\rangle|-\frac {m'}2\|p-p_{\nu}-\tilde c\|^2\\
&\ge|1-\nu|\langle B_\nu^{-1}\omega,\omega\rangle-C_s\|\omega\|\sqrt{\epsilon}-\frac {m'}2\|p-p_{\nu}-\tilde c\|^2.
\end{aligned}
$$
Therefore, some number $C_r=C_r(\omega,C_s)>0$ exists such that $|1-\nu|\le C_r\sqrt{\epsilon}$.
\end{proof}

Let $\mathbb{F}_{\omega}$ denote the Fenchel-Legendre dual of a rotation vector $\omega$, i.e. $\mathbb{F}_{\omega}=\mathscr{L}_{\tilde\alpha_H}(\omega)$.
The following lemma establishes the location of $\tilde{\mathbb{C}}_k$, it lies in $O(\sqrt{\epsilon})$-neighborhood of $\Gamma_k$. The rescaling $\omega\to\nu\omega$ is bounded by Lemma \ref{rescalingomega}.
\begin{lem}\label{location}
For $E>\min h$, there is a constant $C_H>0$ such that $\forall$ $p\in h^{-1}(E)$, $\omega=\partial h^{-1}(p)$, the set $\mathbb{F}_{\nu\omega}\subset\tilde\alpha_H^{-1}(E)$ lies in $C_H\sqrt{\epsilon}$-neighborhood of $p$, i.e. $\mathbb{F}_{\nu\omega}\subset B_{C_H\sqrt{\epsilon}}(p)$.
\end{lem}
\begin{proof}
By Lemma \ref{size}, one has $\mathbb{F}_{\nu\omega}\subset B_{C_s\sqrt{\epsilon}}(p_v)$. Since $\|p-p_{\nu}\|\le|1-\nu|\|\omega\|m^{-1}$, by Lemma \ref{rescalingomega} and setting $C_H=C_s+C_r\max_{p\in\Gamma_k}\|\omega(p)\|m^{-1}$, we finish the proof.
\end{proof}

\begin{lem}\label{pianyi}
Some number $D_H>0$ exists such that each orbit $(p(t),q(t))$ of $\Phi^t_H$ can not be $\tilde c$-minimal if $\|p(t)-\tilde c\|>D_H\sqrt{\epsilon}$ for all $t\in\mathbb{R}$.
\end{lem}
\begin{proof}
Let $L(\dot q,q)$ be the Lagrangian related to the Hamiltonian $H=h+\epsilon P$ through the
Legendre transformation, then $L(\dot q,q)=\langle p,\dot q\rangle-H(p,q)$ where $\dot q=\partial _pH(p,q)$. If $(p(t),q(t))$ is $\tilde c$-minimal for $\tilde c\in\tilde\alpha^{-1}_H(E)$, one obtains from the identity $H(p(t),q(t))\equiv\tilde\alpha_H(\tilde c)$
$$
L(\dot q(t),q(t))-\langle \tilde c,\dot q(t)\rangle+\alpha_H(\tilde c)=\langle p(t)-\tilde c,\dot q(t)\rangle.
$$
One obtains from Taylor's formula that for certain $\lambda\in[0,1]$ the following holds
$$
h(\tilde c)-h(p(t))=\langle\tilde c-p,\partial h(p)\rangle+\frac 12\Big\langle\partial^2h(\lambda\tilde c+(1-\lambda)(\tilde c-p))(\tilde c-p),(\tilde c-p)\Big\rangle.
$$
Since $h$ is positive definite and $\dot q=\partial_ph+\epsilon\partial P$, we get from the formula as above that
$$
\begin{aligned}
\langle p(t)-\tilde c,\dot q(t)\rangle&=\langle p(t)-\tilde c,\partial h(p)\rangle+\langle p(t)-\tilde c,\epsilon \partial _pP\rangle\\
&\ge\frac m2\|\tilde c-p(t)\|^2-2\epsilon\|P\|-\epsilon\|\partial_pP\|
\end{aligned}
$$
where $m>0$ is the lower bound of the eigenvalues of $\partial^2h$ for $p\in h^{-1}(E)$. We set
$$
D_H>\sqrt{2m^{-1}(2\|P\|+\|\partial_pP\|)},
$$
if $\|p(t)-\tilde c\|>D_H\sqrt{\epsilon}$ $\forall$ $t\in\mathbb{R}$, then $L(\dot q(t),q(t))-\langle \tilde c,\dot q(t)\rangle+\tilde\alpha_H(\tilde c)>0$ holds along the whole orbit $(p(t),q(t))$. It contradicts the minimality of the orbit.
\end{proof}

\section{The construction of global transition chain}

We come to the stage to construct a global transition chain that connects a small neighborhood of $\tilde c^{\star}= p^{\star}$ to a small neighborhood of $\tilde c^*=p^*$.
\subsection{Invariance of the $\alpha$-function}
Let $\tilde{\alpha}_{\epsilon,i}$ be the $\alpha$-function for $\tilde G_{\epsilon,i}$ with the form of (\ref{Hamiltonian}). The isoenergetic reduction from systems with three degrees of freedom to two and half establishes the relation between $\tilde\alpha_{\epsilon,i}^{-1}(0)$ and the graph of $\alpha_{\epsilon,i}$:
\begin{theo}\label{flatthm5}
For the Hamiltonian $\tilde G_{\epsilon}(x,x_n,y,y_n)$, we assume that $\partial_{y_n}\tilde G_{\epsilon}\ne 0$ holds on $(\mathbb{T}^n\times B)\cap\{\tilde G_{\epsilon}^{-1}(E)\}$ where $E>\min\tilde\alpha_{\epsilon}$, $B\subset\mathbb{R}^n$ is a ball. Let $y_n=-\lambda G_{\epsilon}(x,y,t)$ be the solution of $\tilde G_{\epsilon}(x,\frac 1{\lambda}t,y,-\lambda G_{\epsilon})=E$ where $\lambda>0$ is a real number. For a class $c\in H^1(\mathbb{T}^{n-1},\mathbb{R})$, if the $c$-minimal curve $x(t)$ of $G_{\epsilon}$ satisfies the condition
$$
(x(t),\lambda^{-1}t,y(t),-\lambda G_{\epsilon}(x(t),y(t),t))\in\mathbb{T}^n\times B, \qquad \forall\ t\in\mathbb{R}
$$
then one has $\tilde c=(c,-\lambda \alpha_{\epsilon}(c))\in \tilde\alpha_{\epsilon}^{-1}(E)$.
\end{theo}
It was proved in \cite{C17b} (Theorem 3.3 there). Let $\tilde x=(x,\lambda^{-1}t)$, $\tilde y=(y,-\lambda G_{\epsilon})$ and $\tau$ denote the time of $\tilde G_{\epsilon}$, the theorem follows from the identity
$$
\int\Big(\Big\langle\frac{dx}{dt},y-c\Big\rangle-G_{\epsilon}+\alpha_{\epsilon}(c)\Big)dt
=\int\Big(\Big\langle\frac{d\tilde x}{d\tau},\tilde y-\tilde c\Big\rangle -\tilde G_{\epsilon}+E\Big)d\tau.
$$
For the application of the theorem in the paper, one has $n=3$. If we regard the graph of $\alpha_{\epsilon,i}$ over  $\mathbb{A}_{i}\cup\mathbb{C}_i^-\cup\mathbb{C}_i^+$ as a set in $\mathbb{R}^3$,
$$
\{(c,-\alpha_{\epsilon,i}(c)):c\in \mathbb{A}_{i}\cup\mathbb{C}_i^-\cup\mathbb{C}_i^+\}
$$
it precisely lies in the surface $\tilde\alpha_{\epsilon,i}^{-1}(0)$. The graph of $\mathbb{A}_{i}$, $\mathbb{C}_i^-$ and $\mathbb{C}_i^+$ are denoted by $\tilde{\mathbb{A}}_{i}$, $\tilde{\mathbb{C}}^-_i$ and $\tilde{\mathbb{C}}_i^+$ respectively. Formula (\ref{xuanzuan}) induces a linear transformation in $H^1(\mathbb{T}^3,\mathbb{R})$ under which the sphere $\tilde{\alpha}_{\epsilon}$ undergoes a linear transformation
\begin{equation}\label{c-rotate}
\Psi^\ell_i:\ \tilde c\to M_i\tilde c.
\end{equation}

Let $\tilde{\alpha}_{\Phi_{\epsilon F_i}^*H}$ be the $\alpha$-function for the Hamiltonian $\Phi_{\epsilon F_i}^*H$, where $\epsilon F_i$ is the generating function for the KAM iteration. Because the rescaling (\ref{energylevel}) induces
$$
\int\tilde pd\tilde x-\Phi_{\epsilon F_i}^*Hdt=\sqrt{\epsilon}\Big(\int\tilde yd\tilde x-\tilde G_{\epsilon,i}ds\Big),
$$
one obtains the rescaling of the first cohomology class, from $\tilde{\alpha}_{\epsilon,i}^{-1}(0)$ to $\tilde{\alpha}_{\Phi_{\epsilon F_i}^*H}^{-1}(0)$
\begin{equation}\label{rescaling}
\Psi^r_i:\ c-c_i\to\sqrt{\epsilon}(c-c_i), \qquad c_3-c_{i,3}\to \frac{\epsilon}{\omega_{3,i}}(c_3-c_{i,3}),
\end{equation}
where $\tilde c_i=(c_i,c_{i,3})=(c_{i,1},c_{i,2},c_{i,3})=p''_i$ if we treat both as the points in $\mathbb{R}^3$.

Because $\Phi_{\epsilon F_i}$ is a Hamiltomorphism, it does not change the Mather set, Aubry set and Ma\~n\'e set, due to Theorem \ref{ThmBernard}.

\subsection{Construction of the global transition chain }
In this section, we show how to construct a global transition chain from the local transition chains.

Recall the circle $\Gamma_{k^\star}$ and $\Gamma_{k^*}$ constructed in Section 3. For generic perturbation $P$, due to Proposition \ref{pro}, the number of strong double resonant points is independent of $\epsilon$. For each $i\in\Lambda_s$, the composition of $\Psi_i^\ell$ and $\Psi_i^r$ maps $\mathbb{A}_{i}$ to an annulus of cohomology equivalence $\tilde{\mathbb{A}}_{i}\subset\tilde\alpha^{-1}_H(E)$ where $\tilde\alpha_H$ denotes the $\alpha$-function of $H$. It also maps $\mathbb{C}_i^\imath$ for all $i\in\Lambda_{\epsilon}$ to a local channel $\tilde{\mathbb{C}}_i^\imath\subset\tilde\alpha^{-1}_H(E)$ ($\imath=\pm,w$), namely
$$
\tilde{\mathbb{A}}_i=\Psi_i^\ell\Psi_i^r(\mathbb{A}_i),\qquad  \tilde{\mathbb{C}}_i^\imath=\Psi_i^\ell\Psi_i^r(\mathbb{C}_i^\imath), \qquad \imath=\pm,w.
$$

Since $H$ is a small perturbation of $h$, the sphere $\tilde\alpha_H^{-1}(E)$ lies in $O(\epsilon)$-neighborhood of $\tilde\alpha_h^{-1}(E)=h^{-1}(E)$ if we treat both as the set in $\mathbb{R}^3$. Let
$$
\tilde{\mathbb{C}}_{k}=\{\tilde c\in\tilde\alpha_H^{-1}(E):\langle k,\omega\rangle=0\  \forall\, \omega\in\mathscr{L}_{\tilde\alpha_H}^{-1}(\tilde c)\},
$$
Under generic perturbation $\epsilon P$, it looks like a channel made up of flats. A subset is said to be a flat of $\tilde\alpha_H$ if $\tilde\alpha_H$ is affine when it is restricted on the set, no longer affine on any set properly containing the set. Since $H$ is autonomous, $E>\min\alpha_H$ and each $\omega\in\mathscr{L}^{-1}_{\tilde\alpha_H}(\tilde{\mathbb{C}}_{k})$ is resonant, $\mathbb{F}_{\omega}=\mathscr{L}_{\tilde\alpha_H}(\omega)$ is a flat of dimension one or two.

The set $\tilde{\mathbb{C}}_{k^{\star}}\cup\tilde{\mathbb{C}}_{k^*}$ obviously contains the flats $\mathbb{F}_i$ ($i\in\Lambda_s$) for strong double resonance we are concerned about
$$
\mathbb{F}_i=\{\tilde c\in\tilde\alpha_H^{-1}(E):\langle k'_i,\omega\rangle=\langle k''_i,\omega\rangle=0,\  \forall \omega\in\mathscr{L}_{\tilde\alpha_H}^{-1}(\tilde c),\ k'_i=k^{\star}\ \mathrm{or}\ k^*\}.
$$
Let $\mathbb{F}_i+d_i\sqrt{\epsilon}=\{\tilde c\in\tilde\alpha^{-1}_H(E):\mathrm{dist}(\tilde c,\mathbb{F}_i)\le d_i\sqrt{\epsilon}\}$ with $d_i<\Delta_i$, the set
$$
\tilde{\mathbb{C}}=\Big(\tilde{\mathbb{C}}_{k^\star}\cup\tilde{\mathbb{C}}_{k^*}\backslash(\cup_{i\in\Lambda_s} \mathbb{F}_i +d_i\sqrt{\epsilon})\Big) \cup\Big(\cup_{i\in\Lambda_s}\tilde{\mathbb{A}}_i\Big)
$$
is path-connected (Theorem \ref{localchainS} and Lemma \ref{jointlemma} below).
According to Lemma \ref{location}, $\tilde{\mathbb{C}}_k$ lies in $C_H\sqrt{\epsilon}$-neighborhood of $\Gamma_k$. The rescaling of the corresponding frequencies is bounded by Lemma \ref{rescalingomega}.

For any $\delta>0$, there exists $\epsilon_0>0$ such that $\forall$ $\epsilon\le\epsilon_0$ there exist totally irreducible $k^\star,k^*\in\mathbb{Z}^3\backslash\{0\}$ such that $\tilde{\mathbb{C}}_{k^\star}\cup\tilde{\mathbb{C}}_{k^*}$ contains two points $\tilde c^{\star},\tilde c^*\in\tilde{\mathbb{C}}$ with $\|\tilde c^{\star}-p^{\star}\|<\frac{\delta}2$ and $\|\tilde c^*-p^*\|<\frac{\delta}2$.
\begin{figure}[htp] 
  \centering
  \includegraphics[width=5.0cm,height=3.0cm]{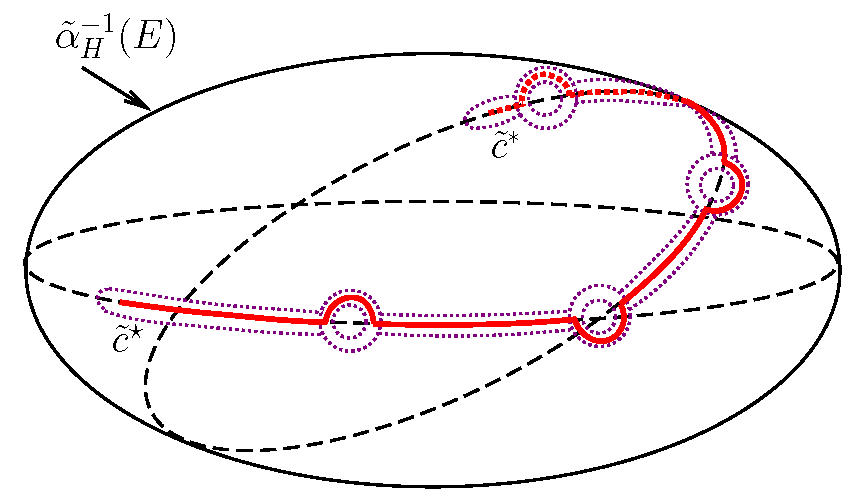}
\end{figure}
Let $\Gamma_c=\Gamma_c(\epsilon P)$ be a path lying inside $\tilde{\mathbb{C}}$, connecting $\tilde c^{\star}$ to $\tilde c^*$. It keeps away from the boundary of $\tilde{\mathbb{C}}$ and moves along a circle of cohomology equivalence when it turns around the strong double resonance. By definition, $\tilde\alpha_H(\tilde c)\equiv E>\min h$ for all $\tilde c\in\tilde{\mathbb{C}}$. We are going to show that $\Gamma_c$ is a global transition chain.

\subsection{The covering property}
According to Theorem \ref{th4.1}, the path $M_0^{-1}\Gamma_k$  is covered by the disks $\cup_{i\in\Lambda_{\epsilon}}\tilde\Sigma_{\epsilon,i}$ where
$$
\tilde\Sigma_{\epsilon,i}=\{p:|p-M_0^{-1}p''_i|<(\eta T_i)^{-1}\epsilon^{\kappa},\ \eta\in (0,1]\}.
$$
We assume the subscripts $\{i\in\Lambda_{\epsilon}\}$ are well-ordered such that $\tilde\Sigma_{\epsilon,i}$ is adjacent to $\tilde\Sigma_{\epsilon,i\pm1}$. Restricted on the energy level set $\bar H^{-1}(E)$ contained in $\tilde\Sigma_{\epsilon,i}\times\mathbb{T}^3$, the Hamiltonian $\bar H=\bar h+\epsilon\bar P$ is reduced to the normal form $G_{\epsilon,i}$ defined on $\Sigma'_{\epsilon,i}$ as shown in Lemma \ref{normalf}. The corresponding channels $\mathbb{C}^\pm_i$ and $\mathbb{C}^w_i$ are described in Theorem \ref{localchainS} and \ref{localchainw}.

Since $T_i^{-1}\epsilon^{\kappa}\ge\epsilon^{\frac 13}$, by replacing $\eta$ with $\frac{\eta}2$, we can assume some $p'\in M_0^{-1}\Gamma_k$ exists, between $p''_i$ and $p''_{i+1}$, such that the disc $\{\|p-p'\|\le \xi\sqrt{\epsilon}\}$ is contained in $\tilde\Sigma_{\epsilon,i}\cap\tilde\Sigma_{\epsilon,i+1}$ and is mapped into $\Sigma'_{\epsilon,i}$ and into $\Sigma'_{\epsilon,i+1}$ by the transformations to get the normal form (\ref{2dHmili}), as shown in Section 3. The number $\xi>0$ can be large if $\epsilon>0$ is small.

Let $\Gamma_{k,i}=M_0^{-1}\Gamma_k\cap\tilde\Sigma_{\epsilon,i}$. As $\Gamma_k$ is smooth and $\epsilon$ is small, $\Gamma_{k,i}$ is $o(\epsilon)$-close to a straight line. For each disk $\tilde\Sigma_{\epsilon,i}$ we have local channel $\tilde{\mathbb{C}}_i^\imath$ made up of the flats $\mathscr{L}_{\tilde\alpha_H}(\nu_p\partial h(p))$ for $p\in\Gamma_{k,i}$, where $\nu_p$ is close to 1 such that $\tilde\alpha_H(\mathscr{L}_{\tilde\alpha_H} (\nu_p\partial h(p)))=h(p)$. If $\tilde{\mathbb{C}}_i^\imath$ and $\tilde{\mathbb{C}}_{i+1}^\imath$ overlap, each flat $\mathscr{L}_{\tilde\alpha_H}(\nu_p\partial h(p))\subset \tilde{\mathbb{C}}_i^\imath\cup \tilde{\mathbb{C}}_{i+1}^\imath$ either entirely lies in the intersection $\tilde{\mathbb{C}}_i^\imath\cap\tilde{\mathbb{C}}_{i+1}^\imath$, or completely stays outside.

\begin{lem}\label{jointlemma}
There exist two numbers: a suitably small $\epsilon_0>0$ and a suitably large $\xi>0$, such that for each $\epsilon\in(0,\epsilon_0]$, the adjacent channels $\tilde{\mathbb{C}}_i^\imath,\tilde{\mathbb{C}}_{i+1}^\imath$ overlap over a channel
$$
\tilde{\mathbb{C}}_i^\imath\cap\tilde{\mathbb{C}}_{i+1}^\imath\supseteq \cup_{p\in\Gamma_{k,i}\cap\Gamma_{k,i+1}}\mathscr{L}_{\tilde\alpha_H}(\nu_p\partial h(p)),
$$
the length of the segment $\Gamma_{k,i}\cap\Gamma_{k,i+1}$ is not shorter than $\frac {\xi}2\sqrt{\epsilon}$.
\end{lem}
\begin{proof}
We study the case that the double resonance at both adjacent points $p''_i, p''_{i+1}$ are weak. Other cases can be treated similarly.

Because of Lemma \ref{localchainw}, the oscillation of $y(\theta)$ is bounded by $\Delta'_j$ if $(x(\theta),y(\theta))$ is an orbit lying in $\tilde{\mathcal{A}}(c)$ with $c\in\mathbb{C}_j^w$. Since $|\bar k''_i||\partial_2 V''_i|\to 0$ as $|\bar k''_i|\to\infty$, we find from (\ref{6.7}) that $\Delta'_j$ is uniformly bounded for all $j$. In the original coordinate, because of the three steps of coordinate changes, $\Phi_{\epsilon F}$, (\ref{xuanzuan}) and (\ref{qicihua}), there exists a constant $C_V>0$ such that $\mathrm{Osc}\|p(t)\|\le C_V\sqrt{\epsilon}$ if $(p(t),q(t))$ is an orbit in $\tilde{\mathcal{A}}(\tilde c)$ for $\tilde c\in\tilde{\mathbb{C}}^w_j$.

Let us investigate the location of $p(t)$ if $(p(t),q(t))$ is a $\tilde c$-minimal orbit. If both $\tilde c$ and $p$ are treated as points in $\mathbb{R}^3$, one has the relation $\tilde\alpha_h=h$.
It follows from Lemma \ref{location} that the channel $\tilde{\mathbb{C}}_k$ ($k=k^*,k^\star$) is entirely contained in $C_H\sqrt{\epsilon}$-neighborhood of $\Gamma_k$. It follows from Lemma \ref{pianyi} that $(p(t),q(t))$ would not be $\tilde c$-minimal orbit of $\Phi_H^t$ if $p(t)$ entirely keeps away from $D_h\sqrt{\epsilon}$-neighborhood of $\tilde c$. Since $\mathrm{Osc}|p(t)|\le C_V\Delta'_j\sqrt{\epsilon}$ if $(p(t),q(t))\subset\tilde{\mathcal{A}}(\tilde c)$ for $\tilde c=\Psi_j^\ell\Psi_j^r(c)$, each Aubry set $\tilde{\mathcal{A}}(\tilde c)$ is entirely contained in $(C_V+D_h+C_H)\sqrt{\epsilon}$-neighborhood of $\tilde c\in\Gamma_k$ if we treat both as the points in $\mathbb{R}^3$.

The parameters $\xi$ and $\epsilon_0$ in the lemma are set such that
$$
\xi\ge 4(C_V+D_h+C_H), \qquad \xi\sqrt{\epsilon_0}\le\frac 12\epsilon_0^{\frac 13}.
$$
Then, there exists a segment $\Gamma'_{k,i}$ of $\Gamma_k$ between $p''_i$ and $p''_{i+1}$ with the length $\frac \xi2\sqrt{\epsilon}$ such that for each $p'\in\Gamma'_{k,i}$ one has $\{\|p-p'\|<\frac \xi4\sqrt{\epsilon}\}\subseteq\tilde\Sigma_{\epsilon,i}\cup\tilde\Sigma_{\epsilon,i+1}$. Indeed, $p'\in\Gamma'_{k,i}$ implies that the distance between $p'$ and the boundary of $\tilde\Sigma_{\epsilon,j}$ is not smaller than $\frac \xi4\sqrt{\epsilon}$ for $j=i,i+1$.

For each $p'\in\Gamma'_{k,i}$, let $\tilde c'\in\mathscr{L}_{\tilde\alpha_H}(\nu_{p'}\partial h(p'))$. Then, the Aubry set $\tilde{\mathcal{A}}(\tilde c')$ for $H$ lies in the disc $\{\|p-p'\|<\frac \xi{4}\sqrt{\epsilon}\}$. Under the composition of the maps $\Phi_{\epsilon F}$, (\ref{xuanzuan}), (\ref{qicihua}) and (\ref{energylevel}), for $\tilde\Sigma_{\epsilon,j}\to\Sigma'_{\epsilon,j}$ for $j=i,i+1$ the domain $\{\|p-p'\|<\frac \xi{4}\sqrt{\epsilon}\}$ is mapped onto a domain entirely contained $\Sigma'_{\epsilon,j}$. The Aubry set of $(\Psi_j^\ell\Psi_j^r)^{-1}\tilde c'$ for $G_{\epsilon,j}$ is contained in the domain $\mathbb{T}^2\times\Sigma'_{\epsilon,j}\times|\bar k''_j|\mathbb{T}$. It implies that $\tilde c'\in\tilde{\mathbb{C}}^w_j$ for $j=i,i+1$.
\end{proof}

Since $\tilde{\mathbb{C}}^+_i$ and $\tilde{\mathbb{C}}^-_i$ are joined by $\Psi_i^\ell\Psi_i^r(\mathbb{A}_i)$, the whole path $\Gamma_c$ is covered by
\begin{equation}\label{covering}
\Gamma_c\subset\bigcup_{i\in\Lambda_s}\Big(\Psi_i^\ell\Psi_i^r(\mathbb{A}_{i})\cup\tilde{\mathbb{C}}_i^- \cup\tilde{\mathbb{C}}_i^+\Big)\cup \bigcup_{i\in\Lambda_{\epsilon}\backslash\Lambda_s}\tilde{\mathbb{C}}_i^w.
\end{equation}

\subsection{The genericity}
Before the proof of Theorem \ref{chain}, we focus on a prescribed path. The following theorem is Theorem 5.1 of \cite{C17b}.

\begin{theo}\label{liangouzao}
There exists a set $\mathfrak{C}_{\epsilon_0}$ cusp-residual in $\mathfrak{B}_{\epsilon_0}\subset C^r(B_D\times\mathbb{T}^3,\mathbb{R})$ with $r\ge 6$ such that for each $\epsilon P\in\mathfrak{C}_{\epsilon_0}$, the path $\Gamma_c$ is a transition chain.
\end{theo}
\begin{proof}
To check if $\Gamma_c$ is a transition chain under generic perturbation $\epsilon P$, one only needs to check, for generic perturbation $\epsilon P$, the condition of transition chain for each $c\in\mathbb{A}_{i}\cup\mathbb{C}_i^-\cup\mathbb{C}_i^+$ if $i\in\Lambda_s$ and for each $c\in\mathbb{C}_i^w$ if $\Lambda_{\epsilon}\backslash\Lambda_s$.

As the first step, we show the cusp-residual property that, for every $c\in\mathbb{C}_i^-\cup\mathbb{C}_i^+$ with $i\in\Lambda_s$ and for $c\in\mathbb{C}_i^w$ with $i\in\Lambda_{\epsilon}\backslash\Lambda_s$, the Aubry set $\tilde{\mathcal{A}}(c)$ lies on some NHIC
(candidate of transition chain):

({\bf CT}: $i\in\Lambda_s$). {\it There is an annulus of cohomology equivalence $\mathbb{A}_i$  connecting channel $\mathbb{C}_i^-$ to $\mathbb{C}_i^+$, and $\tilde{\mathcal{A}}(c)$ lies on some NHIC for $c\in\mathbb{C}_i^-\cup\mathbb{C}_i^+$}.

({\bf CT}: $i\in\Lambda_{\epsilon}\backslash\Lambda_s$). {\it The Aubry set $\tilde{\mathcal{A}}(c)$ lies on certain NHIC for each $c\in\mathbb{C}_i^w$}.

At strong resonant point $p_i''$, one has a decomposition $C^r(B_D\times \T^3,\R)=C^r(\T^2,\R)\oplus C^r(B_D\times \T^3,\R)/C^r(\T^2,\R)$ via $P(p,q)=V_i(\langle k',q\rangle,\langle k''_i,q\rangle)+P''(p,q)$ where the resonant term $V_i(\langle k',q\rangle,\langle k''_i,q\rangle)=Z(p''_i,\langle k',q\rangle,\langle k''_i,q\rangle)$ is defined in (\ref{resonance}), $P''_i=P-V_i\in C^r(B_D\times \T^3,\R)/C^r(\T^2,\R)$ is the non-resonant term and $k'=k^*,k^{\star}$. 

\begin{lem}\label{strong}
Let $\mathfrak{V}_i$ be the residual set used in Theorem \ref{C17bTh}, Proposition \ref{pro6.1}, Theorem \ref{C17aTh} and Theorem \ref{localchainS}. Then, the following set
$$
\mathfrak{P}_i=\{V_i+P''_i: V_i\in\mathfrak{V}_i,\ (\mathbf{CT}: i\in\Lambda_s)\ \mathrm{holds}\ \mathrm{for}\ h+\epsilon(V_i+P''_i)\}
$$
is cusp-residual in $\mathfrak{B}_1$.
\end{lem}
\begin{proof}
According to Lemma \ref{normalf}, the Hamiltonian $H$ is reduced to the normal form $G_{\epsilon,i}(x,y,\theta)$ of Formula (\ref{2dHmili}) when it is restricted on the energy level set $H^{-1}(E)$ lying in $\mathbb{T}^3\times\{|p-p''_i|\le (\eta T_i)^{-1}\epsilon^{\kappa}\}$, the neighborhood of double resonant point $p''_i$. In such a normal form, the main part is independent of $\epsilon$ and the remainder $R_{\epsilon,i}$ is uniformly bounded in the sense that $\|R_{\epsilon,i}\|_{C^{r-2}}\le a_0\epsilon^{\kappa}$ holds for all $P\in\mathfrak{B}_1$ if we consider it as the function of $(x,y,\vartheta)$ where $\vartheta=\omega_{3}|k''_i|^{-1}\sqrt{\epsilon}^{-1}\theta$.

By applying Theorem 1.1 and 1.2 of \cite{C17b} to the Lagrangian obtained from $G_{\epsilon,i}$, we find that there exists an open-dense set $\mathfrak{V}_i$ in $C^r(\mathbb{T}^2,\mathbb{R})$, each $V_i\in\mathfrak{V}_i$ is associated some small $\epsilon_{V_i}>0$ such that the condition $(${\bf CT}: $i\in\Lambda_s)$ holds for each $R_{\epsilon,i}\in\mathfrak{B}_{\epsilon_{V_i}}\subset C^{r-2}(\Sigma'_{\epsilon,i}\times\mathbb{T}^3,\mathbb{R})$. Because $\|R_{\epsilon,i}\|_{C^{r-2}}\le a_0\epsilon^{\kappa}$, for $P=V_i+P''_i\in\mathfrak{B}_1$ with $V_i\in\mathfrak{V}_i$, the condition $(${\bf CT}: $i\in\Lambda_s)$ holds for each $\epsilon\in(0,(a_0^{-1}\epsilon_{V_i})^{1/\kappa}]$.
\end{proof}

Next, we consider NHICs away from strong double resonance. The path $\Gamma_{k'}$ induces a decomposition $C^r(B_d\times \T^3,\R)=C^r(B_D\times \T^1,\R)\oplus C^r(B_D\times \T^3,\R)/C^r(B_D\times \T^1,\R)$
via $P(p,q)=Z_{k'}(p,\langle k',q\rangle)+P'(p,q)$ ($k'=k^{\star},k^*$) where $Z_{k'}$ is defined in (\ref{decom}) consisting of Fourier modes of $P$ in span$_\Z\{k'\}$, and $P'=P-Z_{k'}\in C^r(B_D\times \T^3,\R)/C^r(B_D\times \T^1,\R)$. Treating $\langle k',q\rangle$ as a scalar variable $x$ defined on $\mathbb{T}$, there exists an open-dense set $\mathfrak{Z}_{k'}$ of $C^r(\Gamma_{k'}\times\mathbb{T},\mathbb{R})$ such that for each $Z_{k'}\in\mathfrak{Z}_{k'}$ it holds simultaneously for all $p\in\Gamma_{k'}$ that the second derivative of $Z_{k'}$ in $x$ at its maximal point is uniformly upper bounded below zero, because of Theorem 1.1 of \cite{Zh2}. In this case, the number of strong double resonant points is independent of $\epsilon$.

Given $Z_{k'}\in\mathfrak{Z}_{k'}$, let $\{p'_i\in\Gamma_{k'}\}$ denote the set of bifurcation points, i.e. $Z_{k'}(p'_i,x)$ has two maximal points in $x$. Let $N_w$ denote the cardinality of $\{p'_i\in\Gamma_{k'}\}$, it is finite. Let $H_1=h(p)+\epsilon Z_{k'}(p,\langle k,q\rangle)$, then $\Phi_{H_1}^t$ admits $N_w+1$ pieces of NHIC consisting of minimal periodic orbits of $\Phi_{\bar H}^t$. Because of the presence of strong double resonances, the NHICs may break into more pieces of NHICs if we take the second resonant term into account. However, outside of the neighborhoods of strong double resonances, the number of NHICs will be not more than $N\le N_s+N_w+2$, where $N_s=\#(\Lambda_s)$. Indeed,

\begin{lem}\label{LmGenericNHIC} In $\mathfrak{B}_1$ the following set
$$
\mathfrak{P}_{k'}=\cup\{Z_{k'}+P':Z'_{k'}\in\mathfrak{Z}_{k'},\ (\mathbf{CT}: i\in\Lambda_{\epsilon}\backslash\Lambda_s)\ \mathrm{holds}\ \mathrm{for}\ h+\epsilon(Z_{k'}+P')\}
$$
is cusp-residual.
\end{lem}
\begin{proof}
The normal form around weak resonant point $p''_i$ takes the form of (\ref{2dHmili}) where
$V_i(x)=V'_i(x_1)+V''_i(x_1,|\bar k''|x_2)$ with small $V''_i$, the NHIC for the Hamiltonian flow of $\frac 12\langle B_iy,y\rangle -V'_i(x_1)$ survives the perturbations $V''_i$ and $R_{\epsilon,i}$ if the second derivative of $V'_i$ at its minimal point is positive and some $\epsilon$ is sufficiently small.

We notice $V'_i(x)=-Z'_{k'}(p''_i,x)$. Although the number $\#(\Lambda_{\epsilon}\backslash\Lambda_s)$ increases to infinity as $\epsilon\downarrow 0$, it does not cause trouble.
There is an open-dense set $\mathfrak{Z}_{k'}\subset C^r(B_d\times\mathbb{T},\mathbb{R})$, for each $Z_{k'}(p,x)\in\mathfrak{Z}_{k'}$, the second derivative of $Z_{k'}(p,x)$ in $x$ at its maximal point is uniformly upper bounded below zero. So, the normal hyperbolicity of all NHICs of $G_{\epsilon,i}$ is uniformly bounded from below as $\epsilon$ decreases to $0$. Consequently, some $\epsilon_{Z_{k'}}>0$ exists such that for each $i\in\Lambda_{\epsilon}\backslash\Lambda_s$ the NHIC for $G_{\epsilon,i}$ persists provided $\|R_{\epsilon,i}\|_{C^2}\le \epsilon_{Z_{k'}}$.

By Lemma \ref{normalf} again, given $P=Z'_{k'}+P'_{k'}\in\mathfrak{B}_1$ with $Z'_{k'}\in\mathfrak{Z}_{k'}$, the condition $(${\bf CT}: $i\in\Lambda_{\epsilon}\backslash\Lambda_s)$ holds for each $\epsilon\in(0,(a_0^{-1}\epsilon_{Z_{k'}})^{1/\kappa}]$.

For the coordinate transformation (\ref{xuanzuan}) the matrix $M_i$ is set according to whether $|\bar k''_{i,2}|\ge |\bar k''_{i,3}|$ or not. The path $\partial h(\Gamma_{k'})$ admits a partition of four arcs. The condition $|\bar k''_{i,2}|>|\bar k''_{i,3}|$ holds on two arcs and $|\bar k''_{i,2}|<|\bar k''_{i,3}|$ holds on other two arcs.
\end{proof}

Let $\bar{\mathfrak{P}}_{k'}=\cup_{\nu\in\mathbb{R}}\{\nu P:P\in \mathfrak{P}_{k'}\}\cap\mathfrak{B}_1$ and $\bar{\mathfrak{P}}_i=\cup_{\nu\in\mathbb{R}}\{\nu P:P\in \mathfrak{P}_i\}\cap\mathfrak{B}_1$, then the set $\bar{\mathfrak{P}}_{k^\star}\cap\bar{\mathfrak{P}}_{k^*}\cap(\cap_i\bar{\mathfrak{P}}_{i})$ is residual in $\mathfrak{B}_1$. Applying the Kuratowski-Ulam theorem (categorical analogue of the Fubini theorem, c.f. Chapter 15 of \cite{Ox}), one obtains a residual set $\mathfrak{R}\subset\mathfrak{S}_1$, each $P\in\mathfrak{R}$ is associated with a set $I_P$ residual in $[0,1]$ such that $\cup_{\lambda\in I_P}\lambda P\subset \bar{\mathfrak{P}}_{k^\star}\cap\bar{\mathfrak{P}}_{k^*}\cap(\cap_i\bar{\mathfrak{P}}_{i})$.

Each $P\in\mathfrak{R}$ determines finitely many strong double resonant points $\{p''_i,i\in\Lambda_s(P)\}$.
For each $i\in\Lambda_s(P)$, one obtains the potential $V_i=V_i(P)$ which determines a number $\epsilon_{V_i}>0$. Let $a_P=\min\{\epsilon_{k^\star},\epsilon_{k^*},\epsilon_{V_i},i\in\Lambda_s(P)\}$ where $\epsilon_{k^\star},\epsilon_{k^*}>0$ are determined by the first resonant conditions $k^\star$ and $k^*$ respectively (see Lemma \ref{LmGenericNHIC}). For any $\epsilon\in (0,a_P)$, the flow $\Phi_H^t$ with $H=h+\epsilon P$ admits the conditions ({\bf CT}:$i\in\Lambda_s$) and ({\bf CT}:$i\in\Lambda_{\epsilon}\backslash\Lambda_s$). Therefore, {\it a cusp-residual set $\mathfrak{C}'_{\epsilon_0}$ exists such that for each $\epsilon P\in\mathfrak{C}'_{\epsilon_0}$ the conditions $(${\bf CT}$:i\in\Lambda_s)$ and $(${\bf CT}$:i\in\Lambda_{\epsilon}\backslash\Lambda_s)$ hold}.

Some cohomology equivalence exists around each class $c$ lying in the channels if $\tilde{\mathcal{A}}(c)$ is not a 2-dimensional torus, as it was shown in \cite{CY1}. To finish the proof of Theorem \ref{liangouzao}, we need to verify the condition $({\bf HA})$ for each class $c$ lying in the channels if $\tilde{\mathcal{A}}(c)$ is a 2-dimensional torus. For each Hamiltonian of normal form $G_{\epsilon,i}$, it has been proved in Theorem 1.2 of \cite{C17b}. However, we shall not apply that result since the residual set obtained there is for $C^{r-2}$-topology instead of $C^r$-topology. One step KAM iteration for the construction of $\Phi_{\epsilon F}$ makes $\tilde G_{\epsilon,i}$ lose two times of differentiability.

We apply the following lemma to finish the proof of Theorem \ref{liangouzao}. The proof of the lemma will be presented afterward.
\begin{lem}\label{lem7.4}
Each perturbation $\epsilon P\in\mathfrak{C}'_{\epsilon_0}$ is associated with a set $\mathfrak{P}'_{\epsilon P}$ residual in a ball $\mathfrak{B}_{\delta(\epsilon P)}\subset C^r(B_D\times\mathbb{T}^3,\mathbb{R})$ with small radius $\delta(\epsilon P)>0$, such that for each $\epsilon P'\in\mathfrak{P}'_{\epsilon P}$ the Hamiltonian $h+\epsilon(P+P')$ possesses the property: the condition $({\bf HA})$ holds for each first cohomology classes $\tilde c\in\tilde{\mathbb{C}}_{k^\star}\cup\tilde{\mathbb{C}}_{k^*}\backslash(\cup_{i\in\Lambda_s} \mathbb{F}_i +d_i\sqrt{\epsilon})$ if the Aubry set $\tilde{\mathcal{A}}(\tilde c)$ consists of a two-dimensional torus.
\end{lem}

Obviously, $\cup_{\epsilon P\in\mathfrak{C}'_{\epsilon_0}} (\epsilon P+\mathfrak{B}_{\delta(\epsilon P)})\supseteq\mathfrak{C}'_{\epsilon_0}$. Let $\mathfrak{R}'=\cup_{\epsilon P\in\mathfrak{C}'_{\epsilon_0}} (\epsilon P+\mathfrak{P}'_{\epsilon P})$, because of the Kuratowski-Ulam theorem, it contains a cusp-residual set: there is a set $\bar{\mathfrak{R}}'$ residual in $\mathfrak{S}_1$, each $P\in\mathfrak{S}_1$ is associated with a set $R_p$ residual in $(0,a_P)$ such that $\epsilon P\in\mathfrak{R}'$ holds for all $\epsilon\in R_p$.

For $P\in\bar{\mathfrak{R}}\cap\bar{\mathfrak{R}}'$, $\Lambda_s(P)$ is a finite set. Therefore, a set $R_P$ residual in $(0,a_P)$ exists such that for each $\epsilon\in R_P$, the conditions ({\bf CT}:$i\in\Lambda_s$), ({\bf CT}:$i\in\Lambda_{\epsilon}\backslash\Lambda_s$) and $({\bf HA})$ hold, namely, the flow $\Phi_H^t$ with $H=h+\epsilon P$ admits a transition chain. It verifies the cusp-residual property of the transition chain $\Gamma_c$. The proof of Theorem \ref{liangouzao} is completed.
\end{proof}

\begin{proof}[Proof of Lemma \ref{lem7.4}]
For $\epsilon P\in\mathfrak{C}'_{\epsilon_0}$, the Hamiltonian $h+\epsilon P$ behaves like an {\it a priori} unstable system when it is restricted in neighborhood of $\tilde c$-minimal orbits with $\tilde c\in\tilde{\mathbb{C}}_{k^\star}\cup\tilde{\mathbb{C}}_{k^*}\backslash(\cup_{i\in\Lambda_s} \mathbb{F}_i +d_i\sqrt{\epsilon})$. Although $\epsilon$ is small, it is treated as a fixed number since $\epsilon P$ is fixed. There exists a small number $\delta=\delta(\epsilon P)$ such that for any $\epsilon P'\in\mathfrak{B}_{\delta(\epsilon P)}$ the condition ({\bf CT}: $i\in\Lambda_s$) and ({\bf CT}: $i\in\Lambda_{\epsilon}\backslash\Lambda_s$) holds for $h+\epsilon(P+P')$.

Due to the covering property (\ref{covering}), we only need to check the condition ({\bf HA}) for  $\Gamma_c\cap\tilde{\mathbb{C}}_i^\pm$ and $\Gamma_c\cap\tilde{\mathbb{C}}_i^w$. Restricted in $\tilde\Sigma_{\epsilon,i}\cap\{H^{-1}(E)\}$, we reduce $H(p,q)$ to a system with two and half degrees of freedom so that we can apply the result of (\cite{CY1,CY2}).

We consider the case $\Gamma_c\cap\tilde{\mathbb{C}}_i^{+}$, the proof for the cases $\Gamma_c\cap\tilde{\mathbb{C}}_i^{-}$ and $\Gamma_c\cap\tilde{\mathbb{C}}_i^{w}$ is the same. Restricted on $(\tilde\Sigma_{\epsilon,i}\times\mathbb{T}^3)\cap\{\bar H^{-1}(E)\}$, the Hamiltonian $\bar H=\bar h(\bar p)+\epsilon \bar P(\bar p,\bar q)$ is reduced to the normal form $G_{\epsilon,i}$ of (\ref{2dHmili}), due to Lemma \ref{normalf}. Because of Theorem \ref{C17aTh}, there exist finitely many normally hyperbolic weakly invariant cylinders, denoted by $\tilde\Pi_{\ell}=\tilde\Pi^{\epsilon}_{i,E_\ell-d,E_{\ell+1}+d}$,  modulo the shift $\sigma_i$. For each $c\in\mathbb{C}_i^+=(\Psi_i^\ell\Psi_i^r)^{-1}\tilde{\mathbb{C}}_i^+$, the Aubry set lies in these cylinders. Denoted by $\Pi_{\ell}$ the time-0-section of $\tilde\Pi_{\ell}$, i.e. $\Pi_{\ell}=\tilde\Pi_{\ell}|_{\theta=0}$.

Similar to the argument in Section 4.1 of \cite{C17b}, the cylinder $\Pi_{\ell}$ can be treated as a part of the image of a standard cylinder $\Pi=\{(x,y):(x_2,y_2)=0,x_1\in\mathbb{T},y_1\in [0,1]\}$ under a map $\psi$: $\Pi\to\Pi_{\ell}$. This map induces a $2$-form $\psi^*\omega$ on $\Pi$
$$
\psi^*\omega=D\psi dx_1\wedge dy_1
$$
where $D\psi$ is the Jacobian of $\psi$. Since the second de Rham cohomology group of $\Pi$ is trivial, it follows from Moser's argument on the isotopy of symplectic forms \cite{Mo} that there exists a diffeomorphism $\psi_1$ on $\Pi$ such that
$$
(\psi\circ\psi_1)^*\omega=dx_1\wedge dy_1.
$$
Let $\theta_1=\frac{2|k''_i|\sqrt{\epsilon}\pi}{\omega_{3,i}}$. Since $\Pi_{\ell}$ is invariant for the time-periodic map $\Phi_{G_{\epsilon,i}}=\Phi_{G_{\epsilon,i}}^{\theta_1}$ and $\Phi_{G_{\epsilon,i}}^*\omega=\omega$, one has
\begin{equation}\label{image}
((\psi\circ\psi_1)^{-1}\circ\Phi_{G_{\epsilon}}\circ(\psi\circ\psi_1))^* dx_1\wedge dy_1=dx_1\wedge dy_1
\end{equation}
i.e. $(\psi\circ\psi_1)^{-1}\circ\Phi_{G_{\epsilon}}\circ(\psi\circ\psi_1)$ preserves the standard area. Each invariant circle $\Gamma_{\sigma}\subset\Pi_{\ell}$ is pulled back to the standard cylinder, denoted by $\Gamma^*_{\sigma}$ which is Lipschitz. The parameter $\sigma$ is set to be the algebraic area bounded by the circle and a prescribed one, $\|\Gamma^*_{\sigma}-\Gamma^*_{\sigma'}\|_{C^0}\le b_0\sqrt{|\sigma-\sigma'|}$ (see \cite{CY1}). Since the maps $\psi,\psi_0$ are smooth, back to the current coordinates one has $\|\Gamma_{\sigma}-\Gamma_{\sigma'}\|_{C^0}\le b_1\sqrt{|\sigma-\sigma'|}$. We notice that the cylinder $\Pi_{\ell}$ may be crumple and slanted, the constant $b_1$ might approach infinity if the crumpled cylinder extends to the homoclinics. However, since the cylinder is kept away from the double resonance for certain distance (see Lemma \ref{localchainS}, $\alpha_{\epsilon,i}(\mathscr{L}_{\alpha_{\epsilon,i}}(\lambda_i^\pm g_i^{\pm}))=d_i>0$), the cylinders are moderately crumpled. The constant $b_1$ is therefore uniformly bounded for $\sigma$ if we are restricted on the cylinder $\Pi_{\ell}$.

Recall the rescaling (\ref{energylevel}), (\ref{qicihua}) and let $v''_i=(v''_{i,1},v''_{i,2},v''_{i,3})$ be the double resonant point, we have a transformation $\mathscr{R}_i$
\begin{equation*}
\mathscr{R}_i: \
\begin{cases}
&x_j=u_j, \qquad y_j=\sqrt{\epsilon}^{-1}(v_j-v''_{i,j}), \qquad j=1,2,\\
&\theta=\frac{\sqrt{\epsilon}}{\omega_{i,3}}u_3, \qquad I=\frac{\omega_{i,3}}{\epsilon}(v_3-v''_{i,3}).
\end{cases}
\end{equation*}
Replacing $(x,y,\theta)$ in the normal form $G_{\epsilon,i}$ of (\ref{2dHmili}) by $(u,v_1,v_2)$ defined as above, we obtain a Hamiltonian with two and half degrees of freedom
\begin{equation}\label{newH}
G^\star_{\epsilon,i}(u_1,u_2,u_3,v_1,v_2)=\epsilon G_{\epsilon,i}(x(u_1,u_2),\theta(u_3),y(v_1,v_2))
\end{equation}
where $u_3$ plays the role of time. Because of (\ref{newH}), the function $G^\star_{\epsilon,i}$ solves the equation $\mathscr{M}_i^*\Phi_{\epsilon F}^*\mathscr{M}_0^*H(u,v_1,v_2,-G^\star_{\epsilon,i})=E$, see Formula (\ref{local}).

Obviously, $\tilde\Pi^\star_{\ell}=\mathscr{R}_i\tilde\Pi_{\ell}$ is weakly invariant for the flow of $G^\star_{\epsilon,i}$. Let $\Pi^\star_{\ell}=\tilde\Pi^\star_{\ell}|_{u_3=0}$. Compared with $\Pi_{\ell}$, it shrinks in the coordinate $v$ by the scale $\sqrt{\epsilon}$. The cylinder may become more crumpled, but it is controlled by the factor $1/\sqrt{\epsilon}$. Let $\Gamma^\star_{i,\sigma}$ denote the invariant circle in $\Pi^\star_{\ell}$, one has $\|\Gamma^\star_{\sigma}-\Gamma^\star_{\sigma'}\|_{C^0}\le \frac{b_1}{\sqrt{\epsilon}}\sqrt{|\sigma-\sigma'|}$. Let $\tilde\Gamma^\star_{\sigma}=\cup_{u_3\in[0,|\bar k''_i|]}\Phi_{G^\star_{\epsilon,i}}^{u_3}\Gamma^\star_{\sigma}$, one has
$$
\|\tilde\Gamma^\star_{\sigma}-\tilde\Gamma^\star_{\sigma'}\|_{C^0}\le\frac{b_1}{\sqrt{\epsilon}} \max_{u_3\in[0,|\bar k''_i|]}\|D\Phi_{G^\star_{\epsilon,i}}^{u_3}\| \sqrt{|\sigma-\sigma'|}
$$
Let $H^\star=\mathscr{M}_i^*\Phi_{\epsilon F}^*\mathscr{M}_0^*H$. Corresponding to the cylinder $\tilde\Pi^\star_{\ell}$, there is a cylinder $\hat\Pi^\star_{\ell}$ modulo the shift $\sigma_i$, lying in the energy level set $\{H^{\star-1}(E)\}$
$$
\hat\Pi^\star_{\ell}=\{(u,v)\in\mathbb{T}^3\times\mathbb{R}^3:(u,v_1,v_2)\in\tilde\Pi^\star_{\ell},\ v_3=G^\star_{\epsilon,i}(u,v_1,v_2)\}.
$$
Those invariant 2-tori $\{\tilde\Gamma^\star_{\sigma}\}$ lying in the cylinder. Under the inverse of $\mathscr{M}_i$ and $\Phi_{\epsilon F}$, the tori $\{\tilde\Gamma^\star_{\sigma}\}$ and the cylinder $\hat\Pi^\star_{\ell}$ are mapped to the invariant tori and weakly invariant cylinder for the flow of the Hamiltonian $\mathscr{M}^*_0H$:
$$
\hat\Pi_{\ell}=\Phi_{\epsilon F}^{-1}\mathscr{M}_i^{-1}\hat\Pi^\star_{\ell},\qquad \hat\Gamma_{\sigma}=\Phi_{\epsilon F}^{-1}\mathscr{M}_i^{-1}\hat\Gamma_{\sigma}^\star.
$$
Since $\Phi_{\epsilon F}$ is a diffeomorphism close to identity and $\mathscr{M}_i$ is linear, there exists a number $b_2>0$ such that
\begin{equation}\label{modulus}
\|\hat\Gamma_{\sigma}-\hat\Gamma_{\sigma'}\|_{C^0}\le\frac{b_2}{\sqrt{\epsilon}}\sqrt{|\sigma-\sigma'|}
\end{equation}
Because $\partial\mathscr{M}_0^*h(\bar p''_i)=\bar\omega=(0,\bar\omega_2,\bar\omega_3)$ satisfies the resonant condition $\bar k''_i=(0,\bar k''_{i,2},\bar k''_{i,3})$, we construct another canonical transformation
$\bar{\mathscr{M}}_i$: $\phi=\bar M_i^{-t}\bar q$ and $I=\bar M_i\bar p$, where
$$
\bar M_i=\mathrm{diag}(1,\bar M_{i,2}), \qquad \mathrm{with}\ \
\bar M_{i,2}=\left[
\begin{matrix} \bar k''_{i,2} & j_2\\
\bar k''_{i,3} & j_3
\end{matrix}
\right]
$$
the integers $j_2,j_3$ are chosen such that $\bar k''_{i,2}j_3-\bar k''_{i,3}j_2=1$. Clearly, $\bar M_i$ is  uni-modular and the first derivative of $\bar{\mathscr{M}}^*_i\mathscr{M}_0^*h$ in $I_3$ is not equal to zero at the point $I''_i=\bar M_i\bar p''_i$. Therefore, there exists a $C^r$-function $G^*_{\epsilon,i}$ which solves the equation
\begin{equation}\label{reduction}
\bar{\mathscr{M}}_i^*\mathscr{M}_0^*H(\phi,I_1,I_2,-G^*_{\epsilon,i}(\phi,I_1,I_2))=E
\end{equation}
when it is restricted in a neighborhood of $I''_i$ which covers the domain $\tilde\Sigma_{\epsilon,i} \times\mathbb{T}^3$ under the map $\bar{\mathscr{M}}_i$. The function $G^*_{\epsilon,i}$ defines a Hamiltonian system with two and half degrees of freedom where $\phi_3$ plays the role of time, there is a normally hyperbolic and weakly invariant cylinder $\hat\Pi^*_{\ell}$ for $G^*_{\epsilon,i}$ such that
$$
\hat\Pi_{\ell}=\bar{\mathscr{M}}_i^{-1}\{(\phi,I)\in\mathbb{T}^3\times\mathbb{R}^3:(\phi,I_1,I_2) \in\hat\Pi^*_{\ell},\ I_3=G^*_{\epsilon,i}(\phi,I_1,I_2)\}.
$$
Let $\hat\Gamma^*_{\sigma}=\bar{\mathscr{M}}_i^{-1}\hat\Gamma_{\sigma}$, they lie in the cylinder $\hat\Pi^*_{\ell}$ for which the modulus of continuity of (\ref{modulus}) still holds, probably with a larger coefficient $b^*_2\ge b_2$.

Let $\delta_i\downarrow 0$ be a countable sequence. Because of the modulus of continuity of (\ref{modulus}), it is proved in \cite{CY1,CY2} that, for any small $\delta>0$, it is an open-dense condition for $G^*_{\epsilon,i}$ in $C^r$-topology that the diameter of each connected component of the set
\begin{equation}\label{tobereviewd}
\mathcal{N}(c(\sigma),\check M)|_{\phi_3=0}\backslash(\mathcal{A}(c(\sigma),\check M)+\delta)|_{\phi_3=0}
\end{equation}
is not larger than $\delta_i$ (For the convenience of reader, we shall present the proof of this property in the appendix). The residual set is obtained by taking the intersection of the countably many open-dense sets.
Since $G^*_{\epsilon,i}$ solves Equation (\ref{reduction}), the perturbation $G^*_{\epsilon,i}\to G^*_{\epsilon,i}+G_{\delta}$ can be achieved by the perturbation $\bar{\mathscr{M}}_i^*\mathscr{M}_0^*H(I_1,I_2,I_3,\phi)\to \bar{\mathscr{M}}_i^*\mathscr{M}_0^*H(I_1,I_2,I_3+G_{\delta},\phi)$. Therefore, a set $\mathfrak{R}_{i,\ell}$ residual in $\mathfrak{B}_{\delta}$ is correspondingly obtained for $H$ such that for each $\epsilon P'\in \mathfrak{B}_{\delta}$ the condition ({\bf HA}) holds for cylinder $\tilde\Pi_{\ell}$. Taking the intersection $\cap\mathfrak{R}_{i,\ell}$ which is still residual in $\mathfrak{B}_{\delta}$.
\end{proof}


\begin{proof}[Proof of Theorem \ref{chain}]
Given $\delta>0$ there exists an integer $K=K(\delta)$ such that any point $p\in h^{-1}(E)$ falls into $\delta$-neighborhood of some resonant path $\Gamma_k$ with $|k|\le K$. Since there are finitely resonant circles, we take the intersection of finitely many cusp-residual sets obtained in Theorem \ref{liangouzao} which is still a cusp-residual set.
\end{proof}

\section{Proof of Theorem \ref{mainth}}

The proof of Theorem \ref{mainth} is simple once one has Theorem 5.1 of \cite{C17b} (Theorem \ref{liangouzao} here) and Theorem 3.1 of \cite{LC} as follows.

\noindent{\bf Theorem 3.1} of \cite{LC}. Suppose that there is a generalized transition chain $\Gamma$: $[0,1]\to H^1(\mathbb{T}^n,\mathbb{R})$ joining $c$ to $c'$. Then, there exists an orbit of the Lagrange flow $\Phi_L^t$ $d\gamma$: $\mathbb{R}\to T\mathbb{T}^n$ that connects the two Aubry sets: $\alpha(d\gamma)\subset\tilde{\mathcal{A}}(c)$ and $\omega(d\gamma)\subset\tilde{\mathcal{A}}(c')$.

\begin{proof}[Proof of Theorem \ref{mainth}]
If both $p$ and $\tilde c$ are treated as point in $\mathbb{R}^3$, $p(t)\equiv\tilde c$ holds along each $\tilde c$-minimal orbit $(p(t),q(t))$ of $\Phi_h^t$. By Lemma 2.1 of \cite{CY2}, the set of minimal orbits is upper-semi continuous with respect to perturbation. Therefore, along each orbit $(p(t),q(t))$ in the Aubry set $\tilde{\mathcal{A}}(\tilde c)$ one has $\|p(t)-\tilde c\|<\frac\delta2$ if $\epsilon>0$ is suitably small. So, an orbit connecting $\tilde{\mathcal{A}}(\tilde c^{\star})$ to $\tilde{\mathcal{A}}(\tilde c^*)$ with $\|\tilde c^{\star}-p^{\star}\|<\frac\delta2$ and $\|\tilde c^*-p^*\|<\frac\delta2$ satisfies the condition: some $t^{\star},t^*\in\mathbb{R}$ exist such that $\|p(t^{\star})-p^{\star}\|<\delta$ and $\|p(t^*)-p^*\|<\delta$.

According to Theorem 5.1 of \cite{C17b}, for each $\epsilon P\in\mathfrak{C}_{\epsilon_0}$, there is a transition chain that connects the first cohomology classes $\tilde c^{\star},\tilde c^*\in H^1(\mathbb{T}^3,\mathbb{R})$ satisfying the condition $\alpha(\tilde c^{\star})=\alpha(\tilde c^*)=E$, $\|p^{\star}-\tilde c^{\star}\|<\frac \delta2$ and $\|p^*-\tilde c^*\|<\frac \delta2$. In this case, one obtains from Theorem 3.1 of \cite{LC} an orbit connecting the Aubry set $\tilde{\mathcal{A}}(\tilde c^{\star})$ to $\tilde{\mathcal{A}}(\tilde c^*)$. It completes the proof of the theorem.
\end{proof}

Theorem \ref{mainth} is the elaboration and justification of the sentence in the end of Section 5 of \cite{C17b}: {\it the conjecture of Arnold diffusion for positive definite Hamiltonian turns out to be a theorem for $n=3$}. It is an immediate consequence of Theorem 5.1 of \cite{C17b}.

Theorem \ref{chain} leads to the existence of certain $\delta$-dense of the diffusion orbits, slightly stronger than Theorem \ref{mainth}. A diffusion orbit is said to be $\delta$-dense if it passes through $\delta$-neighborhood of any point $p\in H^{-1}(E)$.

\appendix
\renewcommand{\theequation}{A.\arabic{equation}}
\section{The proof of genericity}
\setcounter{equation}{0}
For the convenience of reader, we present a proof of the property (\ref{tobereviewd}) by applying the ideas and the techniques of \cite{CY1,CY2}. Another version appeared in \cite{BKZ}.

Given a Tonell $C^r$-Hamiltonian $H(p,q,t)$ where $(p,q,t)\in\mathbb{R}^2\times\mathbb{T}^3$, $r\ge 3$, we have a Tonelli Lagrangian $L(\dot q,q,t)=\max_p\langle \dot q,p\rangle -H(p,q,t)$. Let $\Phi_H^{t,t'}$ denote the Hamiltonian flow of $H$, it maps the initial value at the time-$t$-section to the time-$t'$-section. For $\Phi_H^{t,t'}$ we assume that
\begin{enumerate}
  \item there exists a normally hyperbolic and weakly invariant cylinder $\tilde\Pi$, which is a deformation of a standard cylinder $\{(p,q,t)\in\mathbb{R}^2\times\mathbb{T}^2\times\mathbb{T}:\ (p_1,q_1)=0.\}$;
  \item there is a continuous path $\Gamma_c$: $[0,1]\to H^1(\mathbb{T}^2,\mathbb{R})$ such that for any $c\in\Gamma_c$, the Aubry set entirely lies in the cylinder $\tilde\Pi$;
  \item for $c\in\Gamma_c$ if $\tilde{\mathcal{A}}(c)$ is an invariant 2-dimensional torus $\tilde\Upsilon_c$ lying in $\tilde\Pi$, it is a deformation of the torus $\{(p,q,t)\in\mathbb{R}^2\times\mathbb{T}^3:\ (p_1,q_1)=0,\ p_2=\mathrm{const}.\}$
\end{enumerate}
Let $\check{\pi}$: $\check{M}=\{q:\,q_1\,\mathrm{mod}\, 4\pi,q_2\,\mathrm{mod}\, 2\pi\}\to\mathbb{T}^2$ be a finite covering space of $\mathbb{T}^2$. The lift of $\tilde\Pi$ to $T^*\check{M}\times\mathbb{T}$ consists two copies, denoted by $\tilde\Pi_\ell$ and $\tilde\Pi_r$. For $c\in\Gamma_c$, if the Aubry set $\tilde{\mathcal{A}}(c)$ is an invariant torus $\tilde\Upsilon_c\subset\tilde\Pi$, its lift also consists of two components, $\tilde\Upsilon_{c,\ell}\subset\tilde\Pi_\ell$ and $\tilde\Upsilon_{c,r}\subset\tilde\Pi_r$. Let $\tilde\Pi_0$, $\tilde\Pi_{\ell,0}$, $\tilde\Pi_{r,0}$, $\tilde\Upsilon_{c,0}$, $\tilde\Upsilon_{c,\ell,0}$ and $\tilde\Upsilon_{c,r,0}$ denote the time-0-section of $\tilde\Pi$, $\tilde\Pi_{\ell}$, $\tilde\Pi_{r}$, $\tilde\Upsilon_{c}$, $\tilde\Upsilon_{c,\ell}$ and $\tilde\Upsilon_{c,r}$ respectively. Denote by $\pi$ the projection such that $\pi(p,q,t)=(q,t)$, let $\Upsilon=\pi\tilde\Upsilon$. Let $\Gamma_c^*\subset\Gamma_c$ such that
$$
\Gamma_c^*=\{c\in\Gamma_c:\,\tilde{\mathcal{A}}(c)\, \text{\rm is an invariant torus}\}.
$$

Let $B_D\in\mathbb{R}^2$ denotes a ball about the origin of radius $D$. We assume that $D>0$ is suitably large, such that for all $c\in\Gamma_c$ the $c$-minimal orbits of $H$ entirely stay in $B_D\times\mathbb{T}^3$. Let $\mathfrak{B}_{\epsilon}\subset C^r(B_D\times\mathbb{T}^3,\mathbb{R})$ denote a ball about the origin of radius $\epsilon>0$.

\begin{theo}\label{fundamental}
For any small $d_1>0$, there exists a set $\mathfrak{O}$ open-dense in $\mathfrak{B}_{\epsilon}$ such that for each $H_\delta\in\mathfrak{O}$, it holds for $H+H_{\delta}$ and simultaneously for all $c\in\Gamma^*_c$ that the diameter of each connected component of the set
$$
\mathcal{N}(c,\check M)|_{t=0}\backslash(\mathcal{A}(c,\check M)+\delta)|_{t=0}\ne\varnothing
$$
is not larger than $d_1$.
\end{theo}
Before the proof, we review some properties of the barrier functions. Starting from every point $x=(q,\tau)\in\mathbb{T}^3$ there exists at least one backward minimal curve $\gamma^-_{c,x}$: $(-\infty,\tau]\to\mathbb{T}^2$, namely $\gamma^-_{c,x}(\tau)=q$ and
$$
\begin{aligned}
&\int^{\tau}_tL(\dot\gamma^-_{c,x}(s),\gamma^-_{c,x}(s),s)-\langle c,\dot\gamma^-_{c,x}(s)\rangle ds\\
&\le\int^{\tau}_{t'}L(\dot\xi(s),\xi(s),s)-\langle c,\dot\xi(s)\rangle ds+(t-t')\alpha(c)
\end{aligned}
$$
holds for any absolutely continuous curve $\xi$: $[t',\tau]\to\mathbb{T}$ with $\xi(\tau)=\gamma^-_{c,x}(\tau)$, $\xi(t')=\gamma^-_{c,x}(t)$ with $t'=t\mod 2\pi$. It produces an orbit $(\dot\gamma^-_{c,x}(t)),\gamma^-_{c,x}(t))$ which approaches the Aubry set for $c$ as $t\to-\infty$. Similarly, starting from every point $x=(q,\tau)\in\mathbb{T}^3$ there exists one forward minimal curve $\gamma^+_{c,x}$: $[\tau,\infty)\to\mathbb{T}^3$, the orbit $(\dot\gamma^+_{c,x}(t)),\gamma^+_{c,x}(t))$ approaches the Aubry set for $c$ as $t\to\infty$.

According to the weak KAM theory, for almost every point $(q,t)\in\mathbb{T}^3$, the forward and backward orbit is uniquely determined by the forward and backward weak KAM solution respectively. The initial condition of the orbit is determined by the solution $u_c^{\pm}$ such that $\dot\gamma^\pm_{c,x}(\tau)=\partial_pH(\partial_q u^{\pm}_c(q,\tau)+c,q,\tau)$.

Given an Aubry class for $c\in\Gamma_c$ we can define its elementary weak KAM solution. In the covering space $\check{M}$, there are two Aubry classes for $c\in\Gamma_c$, $\tilde\Upsilon_{c,\ell}$ and $\tilde\Upsilon_{c,r}$. To define the elementary weak KAM solution $u^\pm_{c,\ell}$ for $\tilde\Upsilon_{c,\ell}$, we construct a perturbation $L(\dot q,q,t)\to L(\dot q,q,t)+V_\ell(q,t)$, where $V_\ell\ge 0$ and $\Upsilon_{c,r}\subset\mathrm{supp}V_\ell \subset(\Upsilon_{c,r}+\delta)=\{(q,t): \mathrm{dist}((q,t),\Upsilon_{c,r})\le\delta\}$. For $V_\ell\ne 0$, there exists a unique weak KAM solution $u^\pm_{c,V_\ell}$ modulo constant. For almost every point $(q,t)\in\check{M}\times\mathbb{T}$, the function $u^\pm_{c,V_\ell}$ determines a unique forward and backward minimal orbit $(\dot\gamma^\pm_{c,x}(t)),\gamma^\pm_{c,x}(t))$ such that $\dot\gamma^\pm_{c,x}(\tau)=\partial_pH(\partial_q u^{\pm}_{c,V_\ell}(q,\tau)+c,q,\tau)$ and the orbits approaches $\tilde\Upsilon_{c,\ell}$ as $t\to\pm\infty$ respectively. Let $V_{c,\ell}\downarrow 0$, the function $u^\pm_{c,V_\ell}$ converges to a function $u^\pm_{c,\ell}$ which is obviously a weak KAM solution for $H$, it is called elementary weak KAM solution for $\tilde\Upsilon_{c,\ell}$. The elementary weak KAM solution for $\tilde\Upsilon_{c,r}$ is defined in the same way, denoted by $u^\pm_{c,r}$.

For almost every point $(q,t)\in\check{M}\times\mathbb{T}\backslash\Upsilon_{c,\ell}$ the initial condition $(\partial_pu_{c,r}^\pm(q,t)+c,q,t)$ determines a forward (backward) $c$-minimal orbit that approaches $\tilde\Upsilon_{c,r}$ as $t\to\pm\infty$.
For points $(q,t)\in\check{M}\times\mathbb{T}\backslash\Upsilon_{c,r}$, $u_{c,\ell}^\pm$ determines $c$-minimal orbit approaching $\tilde\Upsilon_{c,\ell}$.

\begin{defi}
The barrier functions for $c\in\Gamma_c$ are defined as follows
$$
B^\ell_c(q,t)=u^-_{c,\ell}(q,t)-u^+_{c,r}(q,t), \qquad B^r_c(q,t)=u^-_{c,r}(q,t)-u^+_{c,\ell}(q,t).
$$
\end{defi}
In the following, we only study $B_c^\ell$. The arguments for $B_c^r$ are the same.
Since the backward weak KAM is semi-concave and the forward weak KAM is semi-convex, the barrier function is semi-concave. Therefore,
\begin{lem}
At each minimal point of $B^\ell_c$, both $u^-_{c,r}$ and $u^+_{c,\ell}$ are differentiable.
\end{lem}
\begin{proof}
By the definition, semi-concave function admits a local decomposition as the sum of a smooth function and a concave function. For a concave function $u$, one can define its sup-derivative $D^+u(x)$ at a point $x$ such that
$u(x+x')-u(x)\le\langle p,x'\rangle$ holds for any $p\in D^+u(x)$ which is a convex set. The function $u$ is differentiable at $x$ iff $D^+u(x)$ is a singleton.

Since $B^\ell_c$ is a sum of two semi-concave functions, its sup-derivative is the sum of the sup-derivatives of $u^-_{c,\ell}$ and $-u^+_{c,r}$. Therefore, $D^+B^\ell_c$ is a single point iff both $D^+u^-_{c,\ell}$ and $D^+(-u^+_{c,r})$ are singleton \cite{CaC}.
\end{proof}

\begin{lem}
If $(q,t)\in\check{M}\times\mathbb{T}\backslash((\Upsilon_{c,\ell}\cup\Upsilon_{c,r})+\delta)$ is a global minimal point of $B^\ell_c$, then $(q,t)\subset\mathcal{N}(c,\check{M})$, namely, passing through the point $(q,t)$ there is a $c$-semi-static curve in the covering space $\check{M}\times\mathbb{T}$.
\end{lem}
\begin{proof}
By the definition, $\partial u^-_{c,\ell}=\partial u^+_{c,r}$ holds at a global minimal point of $B^\ell_c$, denoted by $x=(q,t)$. Therefore, the backward minimal curve $\gamma^-_{c,x}$ is joined smoothly to the forward minimal curve $\gamma^+_{c,x}$. They make up a $c$-semi-static curve for $\check{M}$.
\end{proof}

For a class $c\in\Gamma^*_c$, the covering space $\check{M}\times\mathbb{T}$ is divided into two annuli $\mathbb{A}_{c,r}$ and $\mathbb{A}_{c,\ell}$, bounded by $\Upsilon_{c,\ell}$ and $\Upsilon_{c,r}$. Clearly, one has $\check{\pi}\mathbb{A}_{c,r}=\check{\pi}\mathbb{A}_{c,\ell}$.
The set $\mathcal{N}(c,\check{M})\backslash\mathcal{A}(c,\check{M})$ contains $c$-minimal curves which cross the annulus from one side to another side or vice versa. Each of the curves produces a homoclinic orbit to the torus $\tilde\Upsilon_c$.

\begin{lem}\label{A3}
There is a finite partition of $\Gamma_c$: $\Gamma_c=\cup I_k$, each $I_k$ is a segment of $\Gamma_c$. For each $I_k$ there is an annulus $N_k\subset\mathbb{A}_{c,r}|_{t=0}$, two numbers $\delta>0$ and $d>0$ such that for each $c\in I_k\cap\Gamma^*_c$
\begin{enumerate}
  \item $\mathrm{dist}(N_k,\Upsilon_{c,\ell}\cup\Upsilon_{c,r})\ge\delta$;
  \item each curve $(\gamma(t),t)$ lying in $(\mathcal{N}(c,\check{M})\backslash\mathcal{A} (c,\check{M}))\cap\mathbb{A}_{c,r}$ passes through $N_k$;
  \item for each backward (forward) $c$-minimal curve $\gamma$, let $\{q_i=\gamma(2i\pi)\in N_k\}$, then $|q_i-q_j|\ge d$ if $i\ne j$.
\end{enumerate}
\end{lem}
\begin{proof}
Because $\Gamma_c$ is compact, the speed of each $c$-minimal orbit is uniformly upper bounded for all $c\in\Gamma_c^*$. Given an integer $m>0$, there will be small $\delta_c>0$ such that the period for each $c$-minimal curve to cross the annulus $N_c=\mathbb{A}_{c,r}\backslash((\Upsilon_{c,\ell}\cup\Upsilon_{c,r})+\delta_c)$ is not shorter than $4m\pi$. Because of the upper semi-continuity of Ma\~n\'e set in $c$, there exists some $\delta'_c>0$ such that $\Upsilon_{c',\ell}\cup\Upsilon_{c',r}$ does not touch $N_c$ and the period for each $c'$-minimal curve to cross the annulus $N_c$ is not shorter than $2m\pi$ provided $|c-c'|\le\delta'_c$ and $c'\in\Gamma^*_c$. The first two items are then proved if we notice $\Gamma^*_c$ is compact.

For the third one, we notice that the condition $\gamma(2i\pi)=\gamma(2j\pi)$ for $i\ne j$ implies that $\gamma$ is a curve in the Aubry set. It contradicts the assumption. Since both $N_k$ and $I_k$ are compact, such a constant $d>0$ exists.
\end{proof}

By the definition, the Aubry set $\tilde{\mathcal{A}}(c)$ is an invariant torus if $c\in\Gamma^*_c$. Its time-$2\pi$-section is an invariant circle lying in the cylinder. Fix one of the circles, we are able to parameterize other circle by the algebraic area bounded by the circles. Let us consider the twist map on the standard cylinder first. It is well-known that all invariant circles are Lipschitz with the constant $C_L$ which depends on the twist condition only. Treating each circle as the graph of some periodic function and fixing one as $\tilde\Upsilon_{0,0}$ one can parameterize another circle by the algebraic area bounded by these two circles. The annulus bounded by the circle $\tilde\Upsilon_{\sigma,0}$ and $\tilde\Upsilon_{\sigma',0}$ contains a diamond, the height of the vertical diagonals is $\max_q|\tilde\Upsilon_{\sigma,0}(q)-\tilde\Upsilon_{\sigma',0}(q)|$ and the length of the horizontal diagonal is not shorter than $\frac 1{C_L}\max_q|\tilde\Upsilon_{\sigma,0}(q)-\tilde\Upsilon_{\sigma',0}(q)|$. So, one has
$$
\max_q|\tilde\Upsilon_{\sigma,0}(q)-\tilde\Upsilon_{\sigma',0}(q)|\le\sqrt{2C_L|\sigma-\sigma'|}.
$$
A non-standard cylinder can be regarded as the image of the standard cylinder under some diffeomorphism. So, the $\frac 12$-H\"older continuity still holds, refer to the argument for Formula (\ref{image}).

Each invariant circle corresponds to a unique $c\in\Gamma_c$ such that the Aubry set is the circle. The parameter $\sigma$ is usually defined on a Cantor set, denoted by $\Sigma$. Because of the normal hyperbolicity of the cylinder, we have

\begin{lem}\label{moduluscontinuity}
For $\sigma,\sigma'\in\Sigma$, let $c=c(\sigma)$, $c'=c(\sigma')$. If $c,c'\in I_k$ and $q\in N_k$, then
$$
|B^\ell_{c(\sigma)}(q,0)-B^\ell_{c(\sigma')}(q,0)|\le C(\sqrt{|\sigma-\sigma'|}+|c-c'|).
$$
\end{lem}
\begin{proof}
For $c=c(\sigma)$ with $\sigma\in\Sigma$, the minimal measure is uniquely ergodic. There is only one pair of weak KAM solutions $u^{\pm}_c$ for the configuration space $\mathbb{T}^2$. With respect to the covering space $\check{M}$, we have introduced the elementary weak KAM solutions $u^{\pm}_{c,\ell}$ and $u^{\pm}_{c,r}$. Since the projection $\check\pi$ is an injection when it is restricted in the neighborhood $\Upsilon_{c,\imath}+\delta$ for $\imath=\ell,r$ respectively, for $(q,t)\in\Upsilon_{c}+\delta$ one has
\begin{equation}\label{A10}
u^{\pm}_{c,\ell}(\check\pi^{-1}(q,t)\cap(\Upsilon_{c,\ell}+\delta)) =u^{\pm}_{c,r}(\check\pi^{-1}(q,t)\cap(\Upsilon_{c,r}+\delta))
=u^\pm_c(q,t).
\end{equation}

By the definition of weak KAM solutions, for any $t'<t$ one has
$$
u^-_{c,\ell}(\gamma(t),t)-u^-_{c,\ell}(\gamma(t'),t')\le\int_{t'}^t(L(\dot\gamma(s),\gamma(s),s)-\langle c,\dot\gamma(s)\rangle)ds+(t-t')\alpha(c)
$$
which becomes an equality when $\gamma$ is a backward $c$-semi static curve. Assume $\gamma^-_{c,q}$ is a backward $c$-minimal curve such that $\gamma^-_{c,q}(0)=q$, we have
$$
\begin{aligned}
u^-_{c,\ell}(q,0)-u^-_{c,\ell}(\gamma^-_{c,q}(-2K\pi),0)=&\int_{-2K\pi}^0(L(\dot\gamma^-_{c,q}(s), \gamma^-_{c,q}(s),s)-\langle c,\dot\gamma^-_{c,q}(s)\rangle)ds\\
&+2K\pi\alpha(c),\\
u^-_{c',\ell}(q,0)-u^-_{c',\ell}(\gamma^-_{c,q}(-2K\pi),0)\le&\int_{-2K\pi}^0(L(\dot\gamma^-_{c,q}(s), \gamma^-_{c,q}(s),s) -\langle c',\dot\gamma^-_{c,q}(s)\rangle)ds\\
&+2K\pi\alpha(c').
\end{aligned}
$$
Since $N_k$ keeps away from $\Upsilon_{c,\ell}$, some $K>0$ exists such that for each $q\in N_k$, $c\in I_k$ and $q\in N_k$ one has $\gamma^-_{c,q}(-2K\pi)\in(\Upsilon_{c,\ell}+\delta)$.
Since $c$ and $c'$ are located in a compact set $\Gamma_c$, the $\alpha$ function is convex and finite everywhere, there is some constant $C_1$ such that $|\alpha(c')-\alpha(c)|\le C_1|c-c'|$. Let $\bar\gamma^-_{c,q}$ be the lift of $\gamma^-_{c,q}$ to the universal covering space, one has $|\bar\gamma^-_{c,q}(0)-\bar\gamma^-_{c,q}(-2K\pi)|\le 2C_2K\pi$.
$$
\begin{aligned}
&u^-_{c',\ell}(q,0)-u^-_{c,\ell}(q,0)-(u^-_{c',\ell}(\gamma^-_{c,q}(-2K\pi),0)-u^-_{c,\ell} (\gamma^-_{c,q}(-2K\pi),0))\\
&\le 2K\pi(C_1+C_2)|c-c'|.
\end{aligned}
$$
In the same way one can also obtain
$$
\begin{aligned}
&u^-_{c,\ell}(q,0)-u^-_{c',\ell}(q,0)-(u^-_{c,\ell}(\gamma^-_{c',q}(-2K\pi),0)-u^-_{c',\ell} (\gamma^-_{c',q}(-2K\pi),0))\\
&\le 2K\pi(C_1+C_2)|c-c'|.
\end{aligned}
$$
For $u^+_{c,r}$, $u^+_{c'r}$ we also have similar inequalities. Therefore, it follows from (\ref{A10}) that some points $(q_\ell,0),(q_r,0)\in\Upsilon_{c}+\delta$ exist such that
$$
\begin{aligned}
|B_c(q,0)-B_{c'}(q,0)|&\le 4K\pi(C_1+C_2)|c-c'|\\
&+|u^-_c(q_\ell,0)-u^-_{c'}(q_\ell,0)-u^+_c(q_r,0)+u^+_c(q_r,0)|.
\end{aligned}
$$
By the assumption, both $u_c^-$ and $u_c^+$ are $C^{1,1}$ when they are restricted in $\Upsilon_c+\delta$. Due to the normal hyperbolic property, each $(p,q)\in\tilde\Pi_{0}$ has its stable and unstable fiber which is $C^{r-1}$-smoothly depends on the point $(p,q)$. The fibers are defined by $\partial_q u^{\pm}_c+c$ and one has that
$$
|\partial _qu^\pm_{c}-\partial_qu^\pm_{c'}+c-c'|\le C_3\sqrt{|\sigma-\sigma'|}
$$
holds for some constant $C_3>0$, independent of $c,c'$. Combining above two inequalities, one finishes the proof of the lemma.
\end{proof}

We consider the $c$-minimal curves for $c\in I_k$. Because $I_k$ is compact, there exists a constant $D>0$ such that $|\dot\gamma(t)|\le D$ holds for any $c$-minimal curve with $c\in I_k$.
Let $\Omega_{\tau}=\{(q',q)\in\mathbb{R}^2\times\mathbb{R}^2:|q'-q|\le 2D\tau\ \mathrm{with}\ \tau>0\}$. We consider the action
$$
S_{-\tau}(q',q)=\min_{\stackrel {\xi(-\tau)=q'}{\scriptscriptstyle \xi(0)=q}}\int_{-t}^0 L(\dot\xi(s),\xi(s),s)ds.
$$
For suitably small $\tau>0$, there exists a unique minimal curve if $(q',q)\in\Omega_\tau$. Indeed, because $L$ is Tonelli, the second derivative of any solution $q(t)$ of the Euler-Lagrange equation is bounded by $|\ddot q|\le |\partial_{\dot q\dot q}L^{-1}(\partial_qL-\partial^2_{\dot qq}L\dot\gamma-\partial_{\dot qt}L)|$. Recall the Taylor formula
$$
q(t')=q(t)+\dot q(t)(t'-t)+\frac 12\ddot q(\lambda t+(1-\lambda)t')(t'-t)^2
$$
holds for small $|t'-t|$, where both entries of $\lambda\in\mathbb{R}^2$ takes value in $[0,1]$. Therefore, for small $|t'-t|$, there is an one to one correspondence the initial speed $\dot\gamma(t)$ and the end point $\gamma(t')$. In this case, $S_{-\tau}(q',q)$ is $C^r$-differentiable in both $q'$ and $q$. By the definition of weak KAM, for $c\in I_k$ one has
$$
u^-_c(q,0)=\min_{q'\in\mathbb{T}^2,\,|q'-q|\le 2D\tau}(S_{-\tau}(q',q)-\langle c,q-q'\rangle+u^-_c(q',-\tau))
$$
We extend $S_{-\tau}$ smoothly to the whole $\mathbb{R}^2\times\mathbb{R}^2$ such that it satisfies the twist condition. Recall the quantities defined in Lemma \ref{A3} such as the annulus $N_k$ and the number $d>0$.
\begin{lem}\label{A5}
Let $S_\delta(q)$ be a $C^r$-function such that $\max\{|q-q'|:q,q'\in\mathrm{supp} S_\delta\}\le d$, $\mathrm{supp} S_\delta\subset N_k$ and $\|S_\delta\|_{C^r}$ is sufficiently small. Then, restricted on $I_k$, there exists a perturbation $H\to H'=H+H_\delta$ such that $\|H_\delta\|_{C^r}$ is small and the barrier function is subject to a translation
$$
B_c(q,0)\to B_c(q,0)+S_\delta(q) \qquad \forall\ c\in I_k,\ q\in\mathrm{supp} S_\delta.
$$
\end{lem}
\begin{proof}
The function $S_{-\tau}(q',q)$ induces a symplectic map between the time $-\tau$ section and the time-0-section $\Phi$: $(p',q')\to(p,q)$
$$
p=\frac{\partial S_{-\tau}}{\partial q}(q',q)\qquad p'=-\frac{\partial S_{-\tau}}{\partial q'}(q',q).
$$
We introduce a smooth function $\kappa$ such that $\kappa(q',q)=1$ if $|q'-q|\le K$ and $\kappa(q',q)=0$ if $|q'-q|\ge K+1$. Let $\Phi'$ be the map determined by the generating function $S_{-\tau}+\kappa S_\delta$, the symplectic diffeomorphism $\Psi= \Phi'\circ\Phi^{-1}$ is close to identity if $S_\delta$ is $C^r$-small. We choose a smooth function $\rho(s)$ with $\rho(-\tau)=0$, $\rho(0)=1$ and let $\Phi_s'$ be the symplectic map produced by $S_{-\tau}+\rho(s)\kappa S_\delta$ and let $\Psi_s=\Phi_s'\circ\Phi^{-1}$. Clearly, $\Psi_s$ defines a symplectic isotopy between the identity map and $\Psi$. Thus, there is a unique family of symplectic vector fields $X_s$: $T^*\mathbb{T}^2\to TT^*\mathbb{T}^2$ such that
$$
\frac d{ds}\Psi_s=X_s\circ\Psi_s.
$$
By the choice of perturbation, there is a simply connected and compact domain $D$ such that $\Psi_s|_{T^*\mathbb{T}^2\backslash D}=id$. It follows that there exists a Hamiltonian $H_1(p,q,s)$ such that $X_s=J\nabla H_1(p,q,s)$. Re-parametrizing $s$ by $t$, we can make $X_s$ smoothly depend on $t$ and smoothly connected to the zero vector field at $t=-\tau,0$. To show the smallness of $dH'$ we apply a theorem of Weinstein \cite{W}. A neighborhood of the identity in the symplectic diffeomorphism group of a compact symplectic manifold can be identified with a neighborhood of the zero in the vector space of closed 1-forms on the manifold. Since Hamiltomorphism is a subgroup of symplectic diffeomorphism, there is a function $H'$, sufficiently close to $H$, such that $\Phi^{-\tau,0}_{H'}=\Phi_{H_1}^{-\tau,0}\circ\Phi_{H}^{-\tau,0}$.

For all $c\in\Gamma_c$, by the assumption, any backward (forward) $c$-minimal curve will not return back to $\mathrm{supp}S_{-\tau}$ if its initial point falls into the support. Let $u^{\pm,S_\delta}_{c,\imath}$ denotes the elementary weak KAM solution for the perturbed Hamiltonian
$$
\begin{aligned}
u^{-,S_\delta}_{c,\imath}(q,0)=&\min_{|q'-q|\le 2D\tau}(S_{-\tau}(q',q)+S_\delta(q)-\langle c,q-q'\rangle+u^-_{c,\imath}(q',-\tau))\\
=&S_\delta(q)+\min_{|q'-q|\le 2D\tau}(S_{-\tau}(q',q)-\langle c,q-q'\rangle+u^-_{c,\imath}(q',-\tau))\\
=&S_\delta(q)+u^{-}_{c,\imath}(q,0).
\end{aligned}
$$
Obviously, one has $u^{+,S_\delta}_{c,\imath}(q,0)=u^{+}_{c,\imath}(q,0)$. The lemma is proved because the barrier function is the difference of the two functions.
\end{proof}

\begin{proof}[Proof of Theorem \ref{fundamental}]

Given $q^*\in\mathbb{T}^2$, let $\mathbb{S}_{d_1}(q^*)=\{|q-q^*|\le d_1\}$ denote a square. Given a function $B\in C^0(\mathbb{S}_{d_1}(q^*),\mathbb{R})$, let
$$
\mathrm{Argmin}(\mathbb{S}_{d_1}(q^*),B)=\{q\in\mathbb{S}_{d_1}(q^*):B(q)=\min B\}.
$$
Let $\pi_i$ be the projection so that $\pi_i(q_1,q_2)=q_i$ ($i=1,2$). A connected set $V$ is said to be non-trivial for $\mathbb{S}_{d_1}(q^*)$ if $\pi_iV\cap\mathbb{S}_{d_1}(q^*)=\pi_i\mathbb{S}_{d_1}(q^*)$ holds for $i=1$ or $2$. Otherwise, it is said to be trivial for $\mathbb{S}_{d_1}(q^*)$. Let $B^\ell_{c,\delta}$ be the barrier function for the Hamiltonian $H+H_{\delta}$ and the class $c$, we have

\begin{lem}
For any small $\epsilon>0$, there is a set $\mathfrak{O}$ open-dense in $\mathfrak{B}_{\epsilon}$ such that for each $H_{\delta}\in\mathfrak{O}$, it holds simultaneously for all $c\in I_k\cap\Gamma_c^*$ that the set $\mathrm{Argmin}(\mathbb{S}_{d_1}(q^*),B^\ell_{c,\delta})$ is trivial for $\mathbb{S}_{d_1}(q^*)$ provided $\mathbb{S}_{d_1}(q^*)\subset N_k$ and $d_1<d/3$ is suitably small.
\end{lem}
\begin{proof}
The openness is obvious. To show the density, we construct the perturbations $H_\delta\in\mathfrak{B}_{\epsilon}$ such that the barrier function is under a translation $B_c(q,0)\to B_c(q,0)+S_\delta(q)$ for all $c\in I_k\cap\Gamma_c^*$ and $q\in\mathrm{supp}S_\delta$. Because of Lemma \ref{A5}, it works.

Recall the number $d>0$ defined in Lemma \ref{A3}. Given a square $\mathbb{S}_{d_1}(q^*)\subset N_k$ with $3d_1<d$, we consider the space of $C^r$-functions $\mathfrak{S}_1$, a function $S\in\mathfrak{S}_1$ if it satisfies the conditions that $\mathrm{supp}S\subset B_{d/2}(q^*)$ and $S$ is constant in $q_2$ when it is restricted in $\mathbb{S}_{d_1}(q^*)$. Similarly, we can define $\mathfrak{S}_2$ such that $S\in\mathfrak{S}_2$ implies that $\mathrm{supp}S\subset B_{d/2}(q^*)$ and it is constant in $q_1$ when it is restricted in $\mathbb{S}_{d_1}(q^*)$.

In $\mathfrak{S}_i$ we define an equivalent relation $\sim$, two functions $S_1\sim S_2$ implies $S_1-S_2=\mathrm{constant}$ when they are restricted on $\mathbb{S}_{d_1}(q^*)$. Obviously, $\mathfrak{S}_i/\sim$ is a linear space with infinite dimensions. For $S_1,S_2\in\mathfrak{S}_i/\sim$,
$\|S_1-S_2\|_r$ measures the $C^r$-distance if they are regarded as the functions defined on $\mathbb{S}_{d_1}(q^*)$. We also use $\mathfrak{B}_{i,\epsilon}$ to denote a ball in $\mathfrak{S}_i/\sim$, about the origin of radius $\epsilon$ in the sense of the $C^r$-topology.

We claim that there exists a set $\mathfrak{O}_{1,\epsilon}$ open-dense in $\mathfrak{B}_{1,\epsilon}$ such that for each $S_{\delta}\in\mathfrak{O}_{1,\epsilon}$ it holds simultaneously for all $c\in I_k\cap\Gamma_c^*$ that
\begin{equation}\label{Z1}
\pi_1\mathrm{Argmin}(\mathbb{S}_{d_1}(q^*),B^\ell_{c}+S_{\delta})\subsetneqq[q_1^*-d_1,q_1^*+d_1]
\end{equation}
Let $\mathfrak{F}_c=\{B^\ell_c(q,0):c\in\Gamma^*_c\}$ be the set of barrier functions. For $i=1,2$ we set
$$
\mathfrak{Z}_i=\{B\in C^0(\mathbb{S}_{d_1}(q^*),\mathbb{R}):\pi_i\mathrm{Argmin}(\mathbb{S}_{d_1}(q^*),B)=[q_i^*-d_1,q_i^*+d_1]\},
$$
where $q^*=(q^*_1,q^*_2)$.
If the density does not exist, there would be small $\epsilon>0$, for each $S_\delta\in\mathfrak{B}_{1,\epsilon}$, some $c\in\Gamma_c^*$ exists such that $B^\ell_c+S_\delta\in\mathfrak{Z}_1$. Let $\mathfrak{B}_{1,\epsilon}^k$ be the intersection of $\mathfrak{B}_{1,\epsilon}$ with a $k$-dimensional subspace. The box-dimension of $\mathfrak{B}_{1,\epsilon}^k$ in $C^0$-topology will not be smaller than $k$.

For any $B^\ell_c\in\mathfrak{F}_c$ there is only one $S_\delta\in\mathfrak{B}_{1,\epsilon}$ such that $B^\ell_c+S_\delta\in\mathfrak{Z}_1$. Otherwise, there would be $S_\delta'\ne S_\delta$ such that $B^\ell_c+S_\delta'\in\mathfrak{Z}_1$ also. As we have $B^\ell_c+S'_\delta=B^\ell_c+S_\delta+S'_\delta-S_\delta$ where $B^\ell_c+S_\delta\in\mathfrak{Z}_1$ and $S'_\delta\sim S_\delta$, which contradicts the definition of $\mathfrak{S}_1$. For $S_\delta\in\mathfrak{B}_{1,\epsilon}$, let $\mathfrak{S}_{S_\delta}=\{B_c^\ell\in\mathfrak{F}_c:B_c^\ell+S_\delta\in\mathfrak{Z}_1\}$. If the density does not exist, $\mathfrak{S}_{S_\delta}$ is non-empty. For any $S_{\delta},S'_{\delta}\in\mathfrak{B}^k_{1,\epsilon}$, each $B_c^\ell\in\mathfrak{S}_{S_\delta}$ and each $B_{c'}^\ell\in\mathfrak{S}_{S'_\delta}$ one has
\begin{equation}\label{isometric}
\begin{aligned}
d(B_c^\ell,B_{c'}^\ell)&=\max_{q\in\mathbb{S}_{d_1}(q^*)}|B_{c}^\ell(q,0)-B_{c'}^\ell(q,0)|\\
&\ge\max_{|q_1-q_1^*|\le d_1}\Big|\min_{|q_2-q_2^*|\le d_1}B_{c}^\ell(q,0)-\min_{|q_2-q_2^*|\le d_1}B_{c'}^\ell(q,0)\Big|\\
&=\max_{|q_1-q_1^*|\le d_1}|S_{\delta}(q)-S'_{\delta}(q)|=d(S_{\delta},S'_{\delta})
\end{aligned}
\end{equation}
where $q=(q_1,q_2)$ and $d(\cdot,\cdot)$ denotes the $C^0$-metric. It implies that the box-dimension of the set $\mathfrak{F}_c$ is not smaller than the box-dimension of $\mathfrak{B}_{1,\epsilon}^k$ in $C^0$-topology. Guaranteed by the modulus continuity of Lemma \ref{moduluscontinuity}, the box dimension of the set $\mathfrak{F}_c$ is not larger than 3. Therefore, we will obtain an absurdity if we choose $k\ge 4$.

In the same way, we can show that there exists a set $\mathfrak{O}_{2,\epsilon}$ open-dense in $\mathfrak{B}_{2,\epsilon}$ such that for each $S_{\delta}\in\mathfrak{O}_{2,\epsilon}$ it holds simultaneously for all $c\in I_k\cap\Gamma_c^*$ that
\begin{equation}\label{Z2}
\pi_2\mathrm{Argmin}(\mathbb{S}_{d_1}(q^*),B^\ell_{c}+S_{\delta})\varsubsetneq[q_2^*-d_1,q_2^*+d_1].
\end{equation}
Therefore, $\exists$ arbitrarily small $S_{i,\delta}\in\mathfrak{B}_{i,\epsilon}$ such that  $\pi_i\mathrm{Argmin}(\mathbb{S}_{d_1}(q^*),B^\ell_c+S_{1,\delta}+S_{2,\delta})$ is trivial for $\mathbb{S}_{d_1}(q^*)$ and for all $c\in I_k\cap\Gamma_c^*$. Due to Lemma \ref{A5} we obtain the density.
\end{proof}

To finish the proof of Theorem \ref{fundamental}, we split the annulus $N_k$ equally into squares $\{\mathbb{S}_j=|q-q_j|\le\frac {d_1}5\}$. For each $\mathbb{S}_j$, there exists an open-dense set $\mathfrak{O}_{k,j}\subset\mathfrak{B}_{\epsilon}$, for each $H_{\delta}\in\mathfrak{O}_{k,j}$ it holds simultaneously for all $c\in I_k\cap\Gamma^*_c$ that the set $\mathrm{Argmin}(\mathbb{S}_j,B^\ell_{c,\epsilon})$ is trivial for $\mathbb{S}_j$. The intersection $\cap\mathfrak{O}_{k,j}$ is still open-dense in $\mathfrak{B}_{\epsilon}$. For each $H_{\delta}\in\cap_{k,j}\mathfrak{O}_{k,j}$, it holds simultaneously for all $c\in\Gamma^*_c$ that the diameter of each connected component of the Ma\~n\'e set is not larger than $\frac 45d_1$ if it keeps away from the Aubry set.
\end{proof}

\noindent{\bf Acknowledgement}. The author is grateful to the referees for their comments which help a lot
in the revision. He also got benefits from the comments of the referees for the papers \cite{C17a,C17b,CZ}. The work is supported by NNSF of China (No.11790272 and No.11631006) and a program PAPD of Jiangsu Province, China.

\end{document}